\documentclass[onefignum,onetabnum]{siamonline250211}

\newif\ifcolorthree
\colorthreefalse 

\ifcolorthree
\newcommand{\chgthree }[1]{\textcolor{blue}		{#1}}   
\newcommand{\chmthree }[1]{\mathcolor{blue}		{#1}}   
\else
\newcommand{\chgthree }[1]{#1} 
\newcommand{\chmthree }[1]{#1} 
\fi

\newif\ifcolortwo
\colortwofalse 

\ifcolortwo
\newcommand{\chgtwo }[1]{\textcolor{blue}		{#1}}   
\newcommand{\chmtwo }[1]{\mathcolor{blue}		{#1}}   
\else
\newcommand{\chgtwo }[1]{#1} 
\newcommand{\chmtwo }[1]{#1} 
\fi

\newif\ifcolorone
\coloronefalse 

\ifcolorone
\newcommand{\chg }[1]{\textcolor{blue}		{#1}}   
\else
\newcommand{\chg }[1]{#1} 
\fi

\headers{ODE approximations of DDEs with distributed delays}{T. K. S. Ritschel}

\title{On Erlang \chgthree{ODE} approximations \chgthree{of} differential equations\\with distributed time delays%
	\thanks{Submitted to the editors \chgthree{07/19}-2026.}}

\author{Tobias K. S. Ritschel%
	\thanks{Department of Applied Mathematics and Computer Science, Technical University of Denmark, DK-2800 Kgs. Lyngby, Denmark (\email{tobk@dtu.dk}).}
}

\ifpdf
\hypersetup{
	pdftitle={On Erlang \chgthree{ODE} approximations \chgthree{of} differential equations with distributed time delays},
	pdfauthor={T. K. S. Ritschel}
}
\fi

\usepackage{amsmath}
\usepackage{amsfonts}
\usepackage{epstopdf}
\usepackage[caption=false]{subfig}
\usepackage{graphicx}
\usepackage{algorithmic}
\usepackage{tikz}
\usepackage{enumitem}
\usepackage{mathtools} 
\usepackage{booktabs} 
\usepackage{xcolor}
\usepackage{multirow} 

\ifpdf
\DeclareGraphicsExtensions{.eps,.pdf,.png,.jpg}
\else
\DeclareGraphicsExtensions{.eps}
\fi

\setlist[enumerate]{leftmargin=.5in}
\setlist[itemize]{leftmargin=.5in}

\Crefname{ALC@unique}{Line}{Lines}

\usetikzlibrary{positioning}
\usetikzlibrary{calc}
\usetikzlibrary{shapes.geometric}
\usetikzlibrary{angles}
\usetikzlibrary{decorations.pathreplacing}
\usetikzlibrary{calligraphy}

\newcommand{\incr}{\,\mathrm{d}}
\newcommand{\diff}[2]{\frac{\mathrm{d}{#1}}{\mathrm{d}{#2}}}
\newcommand{\pdiff}[3][]{\frac{\partial^{#1}{#2}}{\partial^{#1}{#3}}}

\DeclareMathOperator{\erf}{erf}

\newcommand{\R}  {\mathbb{R}}
\newcommand{\Rp} {\R_{> 0}}
\newcommand{\Rnn}{\R_{\geq 0}}
\newcommand{\Rnp}{\R_{\leq 0}}
\newcommand{\N}  {\mathbb{N}}
\newcommand{\Nnn}{\mathbb{N}_{\geq 0}}
\newcommand{\C}  {\mathbb C}

\DeclareMathOperator{\blkdiag}{blkdiag}

\DeclareMathOperator{\real}{Re}

\newcommand{\matlab}	{\textsc{Matlab}}

\newcommand{\odeff}		{\texttt{ode45}}
\newcommand{\odeofs}	{\texttt{ode15s}}
\newcommand{\fsolve}	{\texttt{fsolve}}
\newcommand{\integral}	{\texttt{integral}}

\newsiamremark{remark}{Remark}
\newsiamremark{hypothesis}{Hypothesis}
\crefname{hypothesis}{Hypothesis}{Hypotheses}
\newsiamthm{assumption}{Assumption}
\newsiamthm{claim}{Claim}

\begin{document}
	\maketitle

	\begin{abstract}
	In this paper, we propose a general approach for approximate simulation and analysis of delay differential equations (DDEs) with distributed time delays based on methods for ordinary differential equations (ODEs). The key innovation is that we 1)~\chgtwo{propose an Erlang mixture approximation of the kernel in the DDEs} and 2)~use the linear chain trick to transform the \chgtwo{resulting} approximate DDEs to ODEs. \chgthree{We refer to this as \emph{the Erlang ODE approximation} of the DDEs, and} we prove that \chgtwo{the} \chgthree{Erlang mixture} approximation converges for continuous and bounded kernels \chgtwo{if the number of terms increases sufficiently fast}.
	\chg{\chgthree{Furthermore, w}e show that if the kernel is also exponentially bounded\chgthree{,} the \chgthree{Erlang ODE} approximat\chgthree{ion} can be used to assess the stability of the steady states of the original DDEs and that the solution to the ODE \chgthree{approximation} converges.}
	\chg{Additionally,} we propose an approach based on bisection and least-squares estimation for determining optimal parameter values in the approximation. Finally, we present numerical examples that demonstrate the accuracy and convergence rate\chgthree{s of the approximations} and the efficacy of the proposed approach for bifurcation analysis and Monte Carlo simulation. The numerical examples involve a modified logistic equation\chg{, chemotherapy-induced myelosuppression,} and a point reactor kinetics model of a molten salt nuclear fission reactor.
\end{abstract}




	\section{Introduction}\label{sec:intro}
Many industrial and natural processes exhibit time delays~\cite{Kolmanovskii:Myshkis:1992}, e.g., due to advection (flow in a pipe or river), diffusion (mixing), or feedback mechanisms (natural or man-made), and as time delays can have a significant impact on the dynamics and stability of a process~\cite{Niculescu:Gu:2004}, it is important to account for them when developing and analyzing mathematical models. Such models typically involve differential equations, and the time delays can either be absolute (discrete) or distributed (continuous)~\cite{Smith:2011}. In the former case, the right-hand side function depends on the state at specific points in the past, and in the latter, it depends on a weighted integral of all past states, i.e., a convolution. The weight function is called the kernel or memory function. Although discrete time delays are more commonly used, they typically arise from a simplification of the underlying phenomena. For instance, assuming plug-flow in a pipe leads to an absolute time delay whereas a nonuniform flow velocity (e.g., Hagen-Poiseuille flow) leads to a distributed delay~\cite{Ritschel:2025}. In this work, we focus on the latter.

In recent decades, there has been a significant interest in distributed time delays, and they have been used to model a variety of processes, e.g., in pharmacokinetics and pharmacodynamics (PK/PD)~\cite{Hu:etal:2018}, neuroscience~\cite{Darabsah:etal:2024}, side-effects of chemotherapy~\cite{Krzyzanski:etal:2018}, the Mackey-Glass system~\cite{Nevermann:Gros:2023, Zhang:Xiao:2016}, which can describe respiratory and hematopoietic diseases, population models in biology~\cite{Cassidy:etal:2019, Diekmann:etal:2020}, and pollution in fisheries~\cite{Bergland:etal:2022}. Mechanical~\cite{Aleksandrov:etal:2023} and economic~\cite{Guerrini:etal:2020} processes have also been considered, and we refer to the book by Kolmanovskii and Myshkis~\cite{Kolmanovskii:Myshkis:1992} for more examples.
Furthermore, many theoretical results have been developed specifically for DDEs with distributed time delays. Rahman et al.~\cite{Rahman:etal:2015} studied the stability of networked systems with distributed delays, Yuan and Belair~\cite{Yuan:Belair:2011} present general stability and bifurcation results, Bazighifan et al.~\cite{Bazighifan:etal:2019} analyze oscillations in higher-order DDEs with distributed delays, Cassidy~\cite{Cassidy:2021} demonstrate the equivalence between cyclic ordinary differential equations (ODEs) and a scalar DDE with a distributed delay, and Aleksandrov et al.~\cite{Aleksandrov:etal:2024} study the stability of systems with state-dependent kernels.

Despite the many applications, there exists far less theory and fewer numerical methods and software for DDEs with distributed time delays than for ODEs. Even for DDEs with absolute delays, there exist many numerical methods~\cite{Bellen:2000, Polyanin:etal:2023} and significant amounts of off-the-shelf software, e.g., for simulation~\cite{Shampine:Thompson:2001}. However, some modeling software does contain functionality for distributed time delays, e.g., NONMEM~\cite{Yan:etal:2021} and Phoenix~\cite{Hu:etal:2018}, and customized numerical methods have also been proposed. Typically, they 1)~approximate the integral in the convolution using a quadrature rule, e.g., a trapezoidal~\cite{Huang:Vandewalle:2004} or Gaussian~\cite{Torkamani:etal:2013} rule, and 2)~discretize the differential equations using a one-step method, e.g., a Runge-Kutta method~\cite{Zhang:Xiao:2016, Eremin:2019}\chgthree{,} or a linear multistep method~\cite{Wang:etal:2018}. Methods based on splines have also been proposed~\cite{Zhou:2016}. However, compared to similar methods for ODEs, approximating the convolution is requiring in terms of computations and memory, and it is not straightforward to choose the time step size adaptively.

Consequently, it is desirable to identify approaches that can directly take advantage of existing theory, methods, and software for ODEs, and for a specific class of kernels, it is possible to transform DDEs with distributed delays to ODEs. Specifically, the transformation is called the \emph{linear chain trick} (LCT)~\cite{MacDonald:1978, Ponosov:etal:2004, Smith:2011}, and it is applicable when the kernel is given by the probability density function of an Erlang distribution (referred to as an Erlang kernel).
\chg{See also~\cite{Diekmann:etal:2018} on sufficient and necessary conditions under which DDEs can be transformed to ODEs.}
The LCT has been widely used, and Hurtado and Kirosingh~\cite{Hurtado:Kirosingh:2019} and Hurtado and Richards~\cite{Hurtado:Richards:2020, Hurtado:Richards:2021} generalized it to transform a wider range of stochastic mean field models to ODEs. Additionally, both Cassidy et al.~\cite{Cassidy:etal:2022} and Krzyzanski~\cite{Krzyzanski:2019} have proposed to simplify the simulation and analysis of specific DDEs by approximating the involved gamma kernels by a hypoexponential kernel and a truncated binomial series, respectively\chg{\chgthree{. Furthermore,} Guglielmi and Hairer~\cite{Guglielmi:Hairer:2025} consider exponential mixture approximations \chgthree{of kernels involving factors of $t^{-\gamma}$ for positive $\gamma$,} and \chgthree{they} demonstrate their approach for DDEs involving both gamma and Pareto kernels}.

\chgthree{Although they consider specific exponential mixture approximations, Guglielmi and Hairer also show that if the integral of the absolute error of an Erlang mixture kernel approximation converges to zero, then the solution to the ODEs obtained with the LCT converges pointwise to that of the original DDEs~\cite[Thm.~1]{Guglielmi:Hairer:2025}. Furthermore, as mentioned by Hurtado and Kirosingh~\cite{Hurtado:Kirosingh:2019} when they generalized the LCT, it is well-known that phase-type distributions (which includes exponential and Erlang mixtures) can approximate the probability distribution of any non-negative random variable (see, e.g., \cite[Chap.~III, Thm.~4.2]{Asmussen:2003} and~\cite[Thm.~3.2.9]{Bladt:Nielsen:2017}). However, it is only the cumulative distribution function and the moments of general phase-type distributions that converge, i.e., Guglielmi and Hairer's result cannot be applied directly to general phase-type distribution approximations of the memory kernel.}

Therefore, in this work, we \chgthree{propose the \emph{Erlang ODE approximation}} of DDEs with distributed time delays \chgthree{involving continuous and bounded kernels. It is based on} a\chgthree{n} Erlang mixture approximation \chgthree{of the involved kernel, and w}e prove that \chgthree{the integral of the absolute error} converges \chgthree{to zero, such that we can use Guglielmi and Hairer's result to show that the solution to the ODE approximation converges to that of the DDEs. We also show that} the \chgthree{ODE} approximat\chgthree{ion} can be used to assess the stability of the steady states of the DDEs.
\chg{In recent work, th\chgthree{e Erlang ODE} approximation was used to develop an algorithm for identifying kernels in DDEs based on partial measurements of the state~\cite{Ritschel:Wyller:2025}.}
\chgthree{W}e \chgthree{also} propose a least-squares approach for determining the Erlang mixture approximation coefficients\chgthree{, and w}e use numerical example\chg{s} to demonstrate that such an optimal approximation is more accurate than an approximation which uses \chgthree{the} theoretical coefficients from the \chgthree{convergence} proof\chg{\chgthree{. W}e compare with the accuracy of one of the exponential mixture approximations described by Guglielmi and Hairer~\cite{Guglielmi:Hairer:2025} for the model in~\cite{Krzyzanski:2019}}\chgthree{, and we use} two numerical examples \chgthree{to} demonstrate that the proposed approach can be used to approximately simulate and analyze the solution of DDEs with distributed time delays.

The remainder of the paper is structured as follows. In Section~\ref{sec:notation} and~\ref{sec:system}, we present the notation used and the system considered in this paper, respectively. Next, \chgtwo{we present the main theoretical results of the paper} in Section~\ref{sec:approx}\chgtwo{, and i}n Section~\ref{sec:algo}, we describe the least-squares approach\chgtwo{. I}n Section~\ref{sec:ex:conv} and~\ref{sec:ex}, we present numerical examples\chgtwo{, and} conclusions are presented in Section~\ref{sec:conclusions}.
	\section{Preliminaries}\label{sec:notation}%
The spaces $\Rnn = \{x \in \R \vert x \geq 0\}$ and $\Nnn = \{0, 1, \ldots\}$ contain the non-negative real numbers and integers, respectively, and $\Rp = \{x \in \R \vert x > 0\}$ \chg{and $\Rnp= \{x \in \R| x \leq 0\}$ are} the space\chg{s} of positive \chg{and non-\chgtwo{positive}} real numbers.
Furthermore, $\det A$ is the determinant of the matrix argument, $A$, and \chgtwo{\chgthree{w}e will use the $\infty$-norm defined by
\begin{align}
	\|a\|_\infty &= \max \,\{|a_i|\}_{i=1}^n, & \|A\|_\infty &= \max \, \left\{\sum_{j=1}^m |A_{ij}|\right\}_{i=1}^n
\end{align}
for a vector $a \in \R^n$ and a matrix $A \in \R^{n \times m}$.}
We will make repeated use of the identity
\begin{align}\label{eq:exponential}
	e^{at} &= \sum_{m=0}^\infty \frac{(at)^m}{m!}
\end{align}
in the proofs in Section~\ref{sec:approx}, and we denote by $[\,\cdot\,]^+ = \max\{0, \cdot\}$ the positive part of the argument.

\chgtwo{The main theoretical results of the paper are related to convergence~\cite[Chap.~7]{Rudin:1976}. A function $f_n: \Rnn \rightarrow \R$ converges \emph{pointwise} to $f: \Rnn \rightarrow \R$ (i.e., $f_n(x) \rightarrow f(x)$) as $n \rightarrow \infty$ if there exists an $N: \Rnn \rightarrow \N$ for every $\epsilon \in \Rp$ and $x \in \Rnn$ such that $|f_n(x) - f(x)| < \epsilon$ for all $n > N(x)$. If $N$ can be chosen independently of $x$, the convergence is \emph{uniform}. We will also show results on \emph{weak} convergence, i.e., that the integral converges: $\int_0^x f_n(y) \incr y \rightarrow \int_0^x f(y) \incr y$ as $n \rightarrow \infty$.  Finally, we will show convergence of the \chgthree{integral of the} absolute difference (convergence in $L^1$-norm): $\int_0^\infty |f_n(x) - f(x)| \incr x \rightarrow 0$ as $n \rightarrow \infty$. We will also use the result that if $f_n(x) \rightarrow g_n(x)$ and $g_n(x) \rightarrow f(x)$ as $n \rightarrow \infty$ where $g_n: \Rnn \rightarrow \R$, then $f_n(x) \rightarrow f(x)$ as well because the error can be bounded through a triangle inequality: $|f_n(x) - f(x)| = |f_n(x) - g_n(x) + g_n(x) - f(x)| \leq |f_n(x) - g_n(x)| + |g_n(x) - f(x)|$. Both terms can be made arbitrarily small. In particular, we will use this result together with Stirling's approximation~\cite{Romik:2000}: $\ln n! \approx n \ln n - n + \frac{1}{2} \ln 2 \pi n$. The error of this approximation decreases with $1/n$, i.e., the approximation converges as $n \rightarrow \infty$.
The concepts are similar if $n$ is a real number instead of a natural number and if the domain and range of $f_n$, $g_n$, and $f$ are different.}
	\section{System}\label{sec:system}%
We consider systems in the form
\begin{subequations}\label{eq:system}
	\begin{align}
		\label{eq:system:x0}
		x(t) &= x_0(t), & t &\in (-\infty, t_0], \\
		\label{eq:system:x}
		\dot x(t) &= f(x(t), z(t)), & t &\in [t_0, t_f],
	\end{align}
\end{subequations}
where $t \in \R$ is time, $t_0, t_f \in \R$ are the initial and final time, $x: \R \rightarrow \R^{n_x}$ is the state, and $x_0: \R \rightarrow \R^{n_x}$ is the initial state function. Furthermore, $f: \R^{n_x} \times \R^{n_z} \rightarrow \R^{n_x}$ is the right-hand side function, and the memory state, $z: \R \rightarrow \R^{n_z}$, is given by the convolution
\begin{subequations}\label{eq:system:delay}
	\begin{align}
		\label{eq:system:z}
		z(t) &= \int\limits_{-\infty}^t \alpha(t - s) r(s) \incr s, \\
		\label{eq:system:r}
		r(t) &= h(x(t)),
	\end{align}
\end{subequations}
where $r: \R \rightarrow \R^{n_{\chmthree{r}}}$ is the delayed variable, and each element of $\alpha: \Rnn \rightarrow \R^{n_z \chmthree{\times n_r}}$ is a\chg{n exponentially bounded} \emph{regular} kernel (see Definition~\ref{def:regular:kernel}--\chg{\ref{def:exponentially:bounded}}). Furthermore, $h: \R^{n_x} \rightarrow \R^{n_{\chmthree{r}}}$ is the memory function. We assume that $f$ and $h$ are differentiable in their arguments \chg{and that $r$ is bounded, i.e., that $\chmtwo{\|} r(t) \chmtwo{\|_\infty} \leq K_r$ for all $t \in (-\infty, t_f]$. Additionally, we assume that \chgtwo{$x_0$ is continuous and exponentially bounded, i.e., $\|x_0(t)\|_\infty \leq L_0 e^{-\sigma t}$, and that the rate parameter, $\sigma \in \Rp$, is strictly smaller than all of the rate parameters in the exponential bounds for the kernels.} W}e refer to the paper by Ponosov et al.~\cite[Thm.~1]{Ponosov:etal:2004} for more details on the existence and uniqueness of solutions to the initial value problem (IVP)~\eqref{eq:system}--\eqref{eq:system:delay}. See also the book by Hale and \chgthree{Verduyn} Lunel~\cite{Hale:Lunel:1993}.
\chgtwo{
	\begin{definition}\label{def:integrals}
		For a scalar-valued kernel, $\alpha: \Rnn \rightarrow \R$, the integrals $\beta: \Rnn \rightarrow \R$ and $\rho: \Rnn \rightarrow \Rnn$ are defined by
		\begin{align}\label{eq:integrals}
			\beta(t) &= \int_0^t  \alpha(s)  \incr s, &
			\rho (t) &= \int_0^t |\alpha(s)| \incr s.
		\end{align}
	\end{definition}
}
\begin{definition}\label{def:regular:kernel}
	A scalar-valued kernel, $\alpha: \Rnn \rightarrow \R$, is \emph{regular} if it satisfies the following \chgtwo{conditions}.
	\begin{enumerate}[label=\chgtwo{(\alph*)}]
		\item \label{def:regular:kernel:bounded}
		It is bounded, i.e., $\chmtwo{|}\alpha(t)\chmtwo{|} \leq K$ for all $t \in \Rnn$ and for some finite $K \in \Rp$.
		\item \label{def:regular:kernel:continuous}
		It is continuous, i.e., for all $\epsilon \in \Rp$ and $t \in \Rnn$, there exists a $\delta \in \Rp$ such that $|\alpha(s) - \alpha(t)| < \epsilon$ for all $s \in \Rnn$ satisfying $|s - t| < \delta$.
		\item \label{def:regular:kernel:integral}
		\chgtwo{The integral of the absolute value of $\alpha$ is finite:}
		\begin{align}\label{eq:finite:absolute:integral}
			\rho(\infty) &= \int_0^\infty |\alpha(s)| \incr s < \infty.
		\end{align}
	\end{enumerate}
\end{definition}
\begin{definition}\label{def:exponentially:bounded}
\chg{
A scalar-valued kernel, $\alpha: \Rnn \rightarrow  \R$, is exponentially bounded if there exist $L, \rho \in \Rp$ such that
\begin{align}
	\chmtwo{|}\alpha(t)\chmtwo{|} &\leq L e^{-\rho t}
\end{align}
for all $t \in \Rnn$.
}
\end{definition}
Next, we present a few corollaries \chgtwo{from the literature} about the steady states of~\eqref{eq:system:x}--\eqref{eq:system:delay} and their stability.
\begin{corollary}\label{thm:steady:state}
	A state $\bar x \in \R^{n_x}$ is a steady state of the system~\eqref{eq:system:x}--\eqref{eq:system:delay} if
	\begin{align}\label{eq:steady:state}
		0 &= f(\bar x, \bar z), &
		\bar z &= \chmtwo{\beta(\infty)} \bar r\chmtwo{,} &
		\chmthree{\bar r} &= h(\bar x)\chmthree{,}
	\end{align}
	\chgtwo{where the $i\chmthree{,j}$'th element of $\beta: \Rnn \rightarrow \R^{n_z \chmthree{\times n_r}}$ is given by~\eqref{eq:integrals} in Definition~\ref{def:integrals}.}
\end{corollary}

\begin{proof}
	In steady state, $x(t) = \bar x$ for all $t$. Consequently, $r(t) = \bar r = h(\bar x)$ and
	\begin{align}
		z(t)
		&= \int_{-\infty}^t \alpha(t - s) \bar r \incr s
		 = \int_{-\infty}^t \alpha(t - s) \incr s \, \bar r
		 = \chmtwo{\beta(\infty)} \bar r,
	\end{align}
	for all $t$\chgtwo{.}
\end{proof}
\begin{corollary}
	The \chgtwo{linearization of the system}~\eqref{eq:system:x}--\eqref{eq:system:delay} \chgtwo{around a steady state, $\bar x$, satisfying~\eqref{eq:steady:state}} describes the evolution of the deviation variable $X: \R \rightarrow \R^{n_x}$:
	\begin{align}\label{eq:linearized:system}
		\dot X(t) &= F X(t) + G \int_{-\infty}^t \alpha(t - s) H X(s) \incr s, &
		X(t) &= x(t) - \bar x.
	\end{align}
	The matrices $F \in \R^{n_x \times n_x}$, $G \in \R^{n_x \times n_z}$, and $H \in \R^{n_{\chmthree{r}} \times n_x}$ are the Jacobians of the right-hand side function and the delay function evaluated in the steady state:
	\begin{align}\label{eq:jacobians}
		F &= \pdiff{f}{x}(\bar x, \bar z), &
		G &= \pdiff{f}{z}(\bar x, \bar z), &
		H &= \pdiff{h}{x}(\bar x).
	\end{align}
\end{corollary}
\begin{corollary}\label{thm:stability}
	The system~\eqref{eq:system:x}--\eqref{eq:system:delay} is locally asymptotically stable around a steady state, $\bar x$, satisfying~\eqref{eq:steady:state} if $\real \lambda < 0$ for all $\lambda \in \C$ that satisfy the characteristic equation
	\begin{align}\label{eq:characteristic:equation}
		\chg{P(\lambda) = } \det\left(F - \lambda I + G \int_0^\infty e^{-\lambda s} \alpha(s) \incr s H\right) = 0,
	\end{align}
	where \chg{$P: \C \rightarrow \C$ is the characteristic function and} $I \in \R^{n_x \times n_x}$ is an identity matrix. \chgthree{Furthermore, $\lambda$ is an isolated root if $\real \lambda > -\sigma$, where $\sigma \in \Rp$ is the rate parameter in the exponential bound on the initial state, $x_0$, in~\eqref{eq:system:x0}.}
\end{corollary}

\begin{proof}
	\chgtwo{Theorem~4.7 in the paper by Diekmann and Gyllenberg~\cite{Diekmann:Gyllenberg:2012} provides conditions under which this result holds, and the assumptions on the system~\eqref{eq:system:x}--\eqref{eq:system:delay} ensure that these conditions are satisfied (see Section~\ref{sec:linearized:stability} in the supplementary material for further details). \chgthree{The condition for the roots to be isolated is presented in the same paper.}}
\end{proof}
	\section{\chgtwo{Main theoretical results}}\label{sec:approx}%
\chgtwo{In this section, we present and prove the main theoretical results of the paper: 1)~We can approximate regular kernels arbitrarily well using Erlang mixture approximations (Proposition~\ref{thm:erlang:mixture:approximation}), 2)~we can approximate the DDEs~\eqref{eq:system:x}--\eqref{eq:system:delay} by ODEs based on Erlang mixture approximations\chgthree{, i.e., the Erlang ODE approximation, and the solution converges to that of the DDEs (Theorem~\ref{thm:ivp:approximation}), and 3)~the} steady state of the ODEs and the roots of the characteristic polynomial \chgthree{converge to those for the DDEs} (Theorem~\ref{thm:ode:approximation}).} For simplicity of the presentation, we only consider scalar kernels in \chgtwo{Proposition~\ref{thm:erlang:mixture:approximation} on the convergence of Erlang mixture approximations}.

\chgtwo{First, we define Erlang kernels which constitute the basis functions in the Erlang mixture kernel approximation that we describe afterwards.} Figure~\ref{fig:erlang:kernels} shows examples of Erlang kernels and Erlang mixture kernels for different rate parameters and coefficients.
\begin{definition}\label{def:erlang:kernel}
	The $m$'th order Erlang kernel,
	$\ell_m: \Rnn \rightarrow \Rnn$, with rate parameter $a \in \Rp$ is given by~\cite{Ibe:2014}
	\begin{align}\label{eq:erlang:pdf}
		\ell_m(t) &= b_m t^m e^{-a t}, &
		b_m &= \frac{a^{m+1}}{m!},
	\end{align}
	where $b_m \in \chmthree{\Rp}$ is a normalization constant, and $m \in \Nnn$.
\end{definition}
The normalization constant, $b_m$, is defined such that the integral of the Erlang kernel is one:
\begin{align}\label{eq:erlang:kernel:integral}
	\int_0^\infty \ell_m(t) \incr t &= 1.
\end{align}
\begin{definition}[\chgthree{Erlang Mixture approximation}]\label{def:erlang:mixture:approximation}
	\chgthree{The $M$'th order Erlang mixture approximation, $\hat \alpha: \Rnn \rightarrow \R$, of the kernel $\alpha: \Rnn \rightarrow \R$ is given by}
	\begin{align}\label{eq:erlang:mixture:approximation}
		\hat \alpha(t) &= \sum_{m=0}^M c_m \ell_m(t), & c_m &= \int_{s_m}^{s_{m+1}} \alpha(s) \incr s, & s_m &= m \Delta s, & \Delta s &= 1/a,
	\end{align}
	\chgthree{where $\ell_m: \Rnn \rightarrow \Rnn$ is the $m$'th order Erlang kernel with rate parameter $a \in \Rp$. The coefficients,} $c_m \in \R$ \chgthree{for $m = 0, \ldots, M$, are} given by the integral of $\alpha$ over a set of intervals with boundaries $s_m \in \Rnn$ for $m = 0, \ldots, M+1$ where the width $\Delta s \in \Rp$ is equal to the inverse of the rate parameter, $a$.
\end{definition}
\begin{proposition}[\chgthree{Convergence of the} \chgtwo{Erlang mixture approximation}]\label{thm:erlang:mixture:approximation}%
	\chgtwo{Let $\alpha: \Rnn \rightarrow \R$ be a scalar-valued regular kernel, and let $\hat \alpha: \Rnn \rightarrow \R$ be \chgthree{the corresponding} Erlang mixture \chgthree{approximation} of order $M \in \Nnn$ \chgthree{with} rate parameter $a \in \Rp$\chgthree{.} Then,
	\begin{enumerate}[label=(\alph*)]
		\item $\hat \alpha$ is regular,
		\label{thm:erlang:mixture:approximation:regularity}
	\end{enumerate}
	and the following asymptotic properties hold as $a \rightarrow \infty$ provided that $M/a \rightarrow \infty$ as well.
	\begin{enumerate}[resume, label=(\alph*)]
		\item \label{thm:erlang:mixture:approximation:pointwise:convergence}
		$\hat \alpha$ converges \chgthree{pointwise} to $\alpha$ for all $t \in \Rnn$, i.e.,
		\begin{align}
			\hat \alpha(t) &\rightarrow \alpha(t).
		\end{align}
		\item \label{thm:erlang:mixture:approximation:integral:convergence}
		$\hat \alpha$ converges weakly to $\alpha$ for all $t \in \Rnn$, i.e.,
		\begin{align}
			\int_0^t \hat \alpha(s) \incr s &\rightarrow \int_0^t \alpha(s) \incr s.
		\end{align}
		\item \label{thm:erlang:mixture:approximation:absolute:integral:convergence}
		The integral of the absolute difference between $\hat \alpha$ and $\alpha$ converges to zero:
		\begin{align}
			\int_0^\infty |\hat \alpha(t) - \alpha(t)| \incr t &\rightarrow 0.
		\end{align}
	\end{enumerate}
	}%
\end{proposition}
\begin{corollary}\label{thm:erlang:mixture:approximation:exponential:boundedness}
	\chgtwo{The Erlang mixture \chgthree{approximation}, $\hat \alpha: \Rnn \rightarrow \R$, is exponentially bounded if $\alpha: \Rnn \rightarrow \R$ is exponentially bounded. Specifically, if $|\alpha(t)| \leq L e^{-\rho t}$ for all $t \in \Rnn$ where $L, \rho \in \Rp$, then $|\hat \alpha(t)| \leq L e^{-\bar \rho t}$ where $\bar \rho = \rho (1 - e^{-\omega})/\omega \in \Rp$ \chgthree{and $\omega \in \Rp$ satisfies $\omega \geq \rho/a$}. Furthermore, $0 \leq \bar \rho < \rho$\chgthree{,} and \chgthree{as $a \rightarrow \infty$, $\omega$ can go to zero, in which case} $\bar \rho \rightarrow \rho$.}
\end{corollary}
\begin{figure}[t]
	\centering
	\subfloat{\includegraphics[width=\textwidth, trim=0pt 0pt 0pt 0pt, clip]{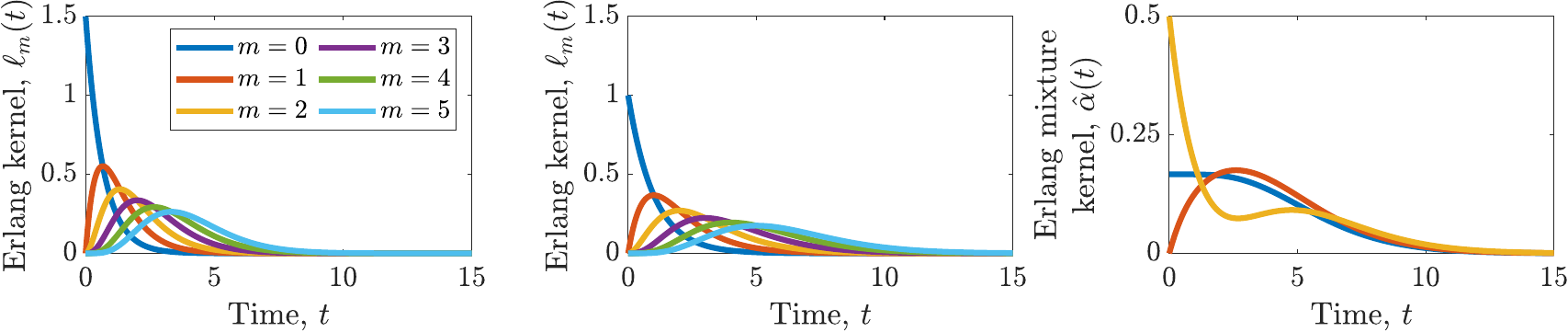}}
	\caption{Erlang kernels for $a = 1.5$ (left) and $a = 1$ (middle) and $m = 0, \ldots, 5$, and an Erlang mixture kernel (right) for $a = 1$, $M = 5$, and the following values of the coefficients. Blue: $c_m = 1/(M+1)$ \chg{for $m = 0, \ldots, M$}. Red: $c_0 = 0$ and $c_m = 1/M$ for $m = 1, \ldots, M$. Yellow: $c_0 = c_M = 1/2$ and $c_m = 0$ for $m = 1, \ldots, M-1$.}
	\label{fig:erlang:kernels}
\end{figure}

\begin{definition}[\chgthree{Erlang ODE approximation}]
	\chgthree{The Erlang ODE approximation of} the DDEs \eqref{eq:system:x}--\eqref{eq:system:delay}
	\chgthree{is parameterized by the rate parameters $a_{ij} \in \Rp$ and orders $M_{ij} \in \Nnn$ in Erlang mixture approximations of the kernel elements $\alpha_{ij}: \Rnn \rightarrow \R$ for $i = 1, \ldots, n_z$ and $j = 1, \ldots, n_r$. The ODE approximation is given by}
	\begin{subequations}\label{eq:lct:approximate:system:ODE}
		\begin{align}
			\label{eq:system:approximate:x}
			\dot{\hat x}(t) &= f(\hat x(t), \hat z(t)), \\
			\label{eq:system:approximate:Z}
			\dot{\hat Z}(t) &= A \hat Z(t) + B \hat r(t), \\
			\label{eq:system:approximate:z}
			\hat z(t) &= C \hat Z(t), \\
			\hat r(t) &= h(\hat x(t)),
		\end{align}
	\end{subequations}
	where $\hat x: \R \rightarrow \R^{n_x}$, $\hat z: \R \rightarrow \R^{n_z}$\chgthree{, $\hat r:\R \rightarrow \R^{n_r}$}, and $\hat Z: \R \rightarrow \R^{M + n_z \chmthree{n_r}}$. Here, $M = \sum_{i=1}^{n_z} \chmthree{\sum_{j=1}^{n_r}} M_{i\chmthree{j}}$, and the elements of $\hat Z$ are auxiliary memory states $\chmthree{\hat Z}_{mi\chmthree{j}}: \R \rightarrow \R$ for $m = 0, \ldots, M_{i\chmthree{j}}$\chgthree{,} $i = 1, \ldots, n_z$\chgthree{, and $j = 1, \ldots, n_r$}.
	\chgthree{T}he initial conditions are
	\begin{subequations}\label{eq:lct:approximate:system:ODE:x0}
		\begin{align}
			\hat x(t_{\chmthree{0}}) &= x_0(t_{\chmthree{0}}), \\
			\chmthree{\hat Z}_{mi\chmthree{j}}(t_0) &= \int_{-\infty}^{t_0} \ell_{mi\chmthree{j}}(t_0 - s) r_{\chmthree{j}}(s) \incr s, & m &= 0, \ldots, M_{i\chmthree{j}}, & i &= 1, \ldots, n_z, & \chmthree{j &= 1, \ldots, n_r,}
		\end{align}
	\end{subequations}
	where $x_0$ is the initial state function in~\eqref{eq:system:x0}\chgthree{, $r(t) = h(x_0(t))$ for $t \in (-\infty, t_0]$, and $\ell_{mij}: \Rnn \rightarrow \Rnn$ is an $m$'th order Erlang kernel with rate parameter $a_{ij}$}.
	\chgtwo{T}he matrices $A\in \R^{M + n_z \chmthree{n_r} \times M + n_z \chmthree{n_r}}$, $B \in \Rnn^{M + n_z \chmthree{n_r} \times n_{\chmthree{r}}}$, \chg{and} $C \in \chmtwo{\R}^{n_z \times M + n_z \chmthree{n_r}}$ \chgthree{have block structure}:
	\begin{subequations}
		\begin{align}
			A &= \chmthree{\blkdiag\left(A^{(1)}, \ldots, A^{(n_z)}\right), & A^{(i)} &= \blkdiag\left(A^{(i, 1)}, \ldots, A^{(i, n_r)}\right)}, \\
			B &=
			\chmthree{
				\begin{bmatrix}
					B^{(1)} \\ B^{(2)} \\ \vdots \\ B^{(n_z)}
				\end{bmatrix}, &
				B^{(i)}} &= \blkdiag\left(b^{(\chmthree{i,}1)}, b^{(\chmthree{i,}2)}, \ldots, b^{(\chmthree{i,}n_{\chmthree{r}})}\right), \\
			C &= \chmthree{\blkdiag\left(C^{(1)}, C^{(2)}, \ldots, C^{(n_z)}\right), &
				C^{(i)} &=
				\begin{bmatrix}
					c^{(i, 1)} & c^{(i, 2)} & \cdots & c^{(i, n_r)}
				\end{bmatrix},}
		\end{align}
	\end{subequations}
	\chgthree{for $i = 1, \ldots, n_z$. Each of the $n_z n_r$ blocks is associated with the Erlang mixture approximation of a different element of $\alpha$, and the auxiliary memory states are sorted row-wise in $\hat Z$ (i.e., first according to $m$, then $j$, and finally $i$). See also Appendix~\ref{sec:deriv}.}
	The block elements, $A^{(i\chmthree{j})} \in \R^{M_{i\chmthree{j}}+1 \times M_{i\chmthree{j}}+1}$, $b^{(i\chmthree{j})} \in \Rnn^{M_{i\chmthree{j}}+1 \times 1}$, and $c^{(i\chmthree{j})} \in \chmtwo{\R}^{1 \times M_{i\chmthree{j}}+1}$, are
	\begin{align}
		A^{(i\chmthree{j})} &= a_{i\chmthree{j}} (L^{(i\chmthree{j})} - I^{(i\chmthree{j})}), &
		b^{(i\chmthree{j})} &= a_{i\chmthree{j}} e_1^{(i\chmthree{j})}, &
		c^{(i\chmthree{j})} &=
		\begin{bmatrix}
			c_{0, i\chmthree{j}} & c_{1, i\chmthree{j}} & \cdots & c_{M_{i\chmthree{j}}, i\chmthree{j}}
		\end{bmatrix},
	\end{align}
	\chgthree{for $i = 1, \ldots, n_z$ and $j = 1, \ldots, n_r$} where $I^{(i\chmthree{j})}, L^{(i\chmthree{j})} \in \R^{M_{i\chmthree{j}}+1 \times M_{i\chmthree{j}}+1}$ are an identity matrix and a lower shift matrix, respectively, i.e., $L^{(i\chmthree{j})}$ contains ones in the first subdiagonal and zeros elsewhere. Furthermore, the first element of $e_1^{(i\chmthree{j})} \in \R^{M_{i\chmthree{j}}+1 \times 1}$ is one and the others are zero\chgthree{. Finally}\chgtwo{, \chgthree{$a_{ij} \in \Rp$ and} $c_{mi\chmthree{j}} \in \R$ \chgthree{are the rate parameter and} the $m$'th coefficient in the Erlang mixture approximation \chgthree{of $\alpha_{ij}$}}.
\end{definition}
\begin{corollary}
	\chgthree{There exists a unique solution to the IVP~\eqref{eq:lct:approximate:system:ODE}--\eqref{eq:lct:approximate:system:ODE:x0} in the Erlang ODE approximation of the IVP~\eqref{eq:system}--\eqref{eq:system:delay} involving DDEs subject to the assumptions described in Section~\ref{sec:system}.}
\end{corollary}
\begin{proof}
	\chgthree{The result follows directly from the differentiability, and thus Lipschitz continuity, of $f$ and $h$ and from the linearity of the right-hand side of~\eqref{eq:system:approximate:Z} in $\hat Z$ and $\hat r$~\cite[Thm.~3.10, Thm.~4.5]{Meiss:2007}.}
\end{proof}
\begin{corollary}\label{thm:initial:state:function}
	\chgtwo{If the initial state function, $x_0$, is constant, then $\chmthree{\hat Z}_{mi\chmthree{j}}(t_0) = r_{\chmthree{j}, 0}$ where $r_0 = h(x_0)$.}
\end{corollary}
\begin{proof}
	\chgtwo{The result follows from the fact that, by definition, the Erlang kernels in~\eqref{eq:erlang:pdf} integrate to one:}
	\begin{align}
		\chmtwo{\chmthree{\hat Z}_{mi\chmthree{j}}(t_0) &= \int_{-\infty}^{t_0} \ell_{mi\chmthree{j}}(t_0 - s) r_{\chmthree{j}}(s) \incr s = r_{\chmthree{j}, 0} \int_{-\infty}^{t_0} \ell_{mi\chmthree{j}}(t_0 - s) \incr s = r_{\chmthree{j}, 0}.}
	\end{align}
\end{proof}
\begin{theorem}[\chgthree{Convergence of the Erlang ODE approximation}]\label{thm:ivp:approximation}
	\chgthree{Let $x: \R \rightarrow \R^{n_x}$ be the solution to the IVP~\eqref{eq:system}--\eqref{eq:system:delay} subject to the assumptions stated in Section~\ref{sec:system}, and let $z: \R \rightarrow \R^{n_z}$ and $r: \R \rightarrow \R^{n_r}$ be the corresponding memory state and delayed quantity. Then, the solution, $\hat x: \R \rightarrow \R^{n_x}$,} to the IVP~\eqref{eq:lct:approximate:system:ODE}--\eqref{eq:lct:approximate:system:ODE:x0} \chgthree{in the corresponding Erlang ODE approximation} converge\chgthree{s} uniformly to \chgthree{$x$} for $t \in \chmthree{[t_0}, t_f]$ as $a_{i\chmthree{j}} \rightarrow \infty$ provided that $M_{i\chmthree{j}}/a_{i\chmthree{j}} \rightarrow \infty$ as well for $i = 1, \ldots, n_z$ \chgthree{and $j = 1, \ldots, n_r$:}
	\begin{align}
		\chmthree{\sup_{t\in[t_0, t_f]} \|\hat x(t) - x(t)\|_\infty \rightarrow 0.}
	\end{align}
	\chgthree{Furthermore, the memory state, $\hat z: \R \rightarrow \R^{n_z}$, and delayed quantity, $\hat r: \R \rightarrow \R^{n_r}$, corresponding to $\hat x$ also converge uniformly to $z$ and $r$:}
	\begin{align}
		\chmthree{\sup_{t\in[t_0, t_f]} \|\hat z(t) - z(t)\|_\infty &\rightarrow 0,} &
		\chmthree{\sup_{t\in[t_0, t_f]} \|\hat r(t) - r(t)\|_\infty &\rightarrow 0.}
	\end{align}
\end{theorem}
\begin{theorem}[\chgthree{Steady state and local stability}]\label{thm:ode:approximation}
	\chgthree{Let $\bar x \in \R^{n_x}$ and $x: \R \to \R^{n_x}$ be the steady state and solution to the IVP~\eqref{eq:system}--\eqref{eq:system:delay} subject to the assumptions stated in Section~\ref{sec:system}.}
	\chgtwo{Then, the following asymptotic properties hold \chgthree{for the corresponding Erlang ODE approximation} as \chgthree{the rate parameters} $a_{i\chmthree{j}} \rightarrow \infty$ provided that the ratio $M_{i\chmthree{j}}/a_{i\chmthree{j}} \rightarrow \infty$ as well for $i = 1, \ldots, n_z$ \chgthree{and $j = 1, \ldots, n_r$}.
	\begin{enumerate}[label=(\alph*)]
		\item \label{thm:ode:approximation:steady:state}
		\chgthree{Suppose that t}he steady state of the \chgthree{Erlang} ODE \chgthree{approximation}~\eqref{eq:lct:approximate:system:ODE} \chgthree{exists. Then, it} converges to $\bar x$ if the matrix $F + G \beta(\infty) H$ is invertible, where $F$, $G$, and $H$ are given by~\eqref{eq:jacobians}.
		\item \label{thm:ode:approximation:roots}
		For every isolated root, $\lambda \in \C$, of the characteristic function in~\eqref{eq:characteristic:equation}, there are roots of the characteristic \chgthree{polynomial} for the \chgthree{Erlang} ODE \chgthree{approximation}~\eqref{eq:lct:approximate:system:ODE} that converge to $\lambda$ provided that $\real \lambda \geq \lambda_{\min}$ where $-\rho_{i\chmthree{j}} < \lambda_{\min} \in \Rnp$ for $i = 1, \ldots, n_z$ \chgthree{and $j = 1, \ldots, n_r$}. The $i\chmthree{,j}$'th rate parameter, $\rho_{i\chmthree{j}} \in \Rp$, is obtained from the exponential bound on $\alpha_{i\chmthree{j}}$ (see Definition~\ref{def:exponentially:bounded}).
		\item \label{thm:ode:approximation:stability}
		\chgthree{Suppose that 1)~the system of DDEs~\eqref{eq:system:x}--\eqref{eq:system:delay} is locally asymptotically stable around $\bar x$, 2)~$x$ exists for $t \in [t_0, \infty)$, and 3)~$\sup_{t\in(-\infty, t_0]} e^{\varrho t} \|x_0(t) - \bar x\|_\infty$ is sufficiently small, where $x_0$ is the initial state function in~\eqref{eq:system:x0} and $\varrho \in \Rp$ satisfies $\sigma < \varrho < \rho_{ij}$ for $i = 1, \ldots, n_z$ and $j = 1, \ldots, n_r$. The rate parameter $\sigma$ is obtained from the exponential bound on $x_0$. Then, if the solution, $\hat x$, to the IVP~\eqref{eq:lct:approximate:system:ODE}--\eqref{eq:lct:approximate:system:ODE:x0} in the Erlang ODE approximation also exists for $t \in [t_0, \infty)$, the error $\hat x(t) - x(t) \rightarrow 0$ as $t \rightarrow \infty$.}
	\end{enumerate}
	}
\end{theorem}
\chgtwo{In the remainder of this section, we present proofs of the statements in Proposition~\ref{thm:erlang:mixture:approximation}\chgthree{, Corollary~\ref{thm:erlang:mixture:approximation:exponential:boundedness},} and Theorem\chgthree{s}~\ref{thm:ivp:approximation}--\ref{thm:ode:approximation}.}

\subsection{\chgtwo{Regularity and convergence of Erlang mixture approximations}}\label{sec:erlang:mixture:kernel:convergence}
\chgtwo{In this section, we prove that the Erlang mixture approximation in \chgthree{Definition}~\ref{def:erlang:mixture:approximation} is regular, that it converges \chgthree{pointwise} and weakly (i.e., the integral converges), and that the integral of the absolute error converges to zero as the rate parameter, $a$, goes to infinity provided that the ratio $M/a$ also goes to infinity, i.e., that the order, $M$, approaches infinity faster than $a$.}

\begin{proof}[Proof \chgtwo{of Proposition~\ref{thm:erlang:mixture:approximation}\ref{thm:erlang:mixture:approximation:regularity}}]
	The boundedness \chgtwo{in Definition~\ref{def:regular:kernel}\ref{def:regular:kernel:bounded}} follows directly from the substitution of the upper bound on $\alpha$ into the expression for the coefficients in~\chgtwo{\eqref{eq:erlang:mixture:approximation}}:
	\begin{align}
		\chmtwo{|}\hat \alpha(t)\chmtwo{|} &= \chmtwo{\left|}\sum_{m=0}^{\chmtwo{M}} \ell_m(t) \int_{s_m}^{s_{m+1}} \alpha(s) \incr s\chmtwo{\right|} \leq \sum_{m=0}^{\chmtwo{M}} \ell_m(t) \int_{s_m}^{s_{m+1}} K \incr s = K \sum_{m=0}^{\chmtwo{M}} \ell_m(t) \Delta s \\
		&= K \sum_{m=0}^{\chmtwo{M}} \frac{(at)^m}{m!} a e^{-at} \frac{1}{a} = K \sum_{m=0}^{\chmtwo{M}} \frac{(at)^m}{m!} e^{-at} \chmtwo{\leq} K e^{at} e^{-at} = K, \nonumber
	\end{align}
	where we have used the identity~\eqref{eq:exponential} \chgtwo{and that all terms in the sum are positive, i.e., $\sum_{m=0}^M (at)^m/m! \leq \sum_{m=0}^\infty (at)^m/m! = e^{at}$ for all $a \in \Rp$ and $M \in \Nnn$}.
	\chgtwo{The continuity of $\hat \alpha$ in Definition~\ref{def:regular:kernel}\ref{def:regular:kernel:continuous} follows directly from the fact that it is a weighted sum of the Erlang kernels, $\ell_m$, which are continuous because they are products of monomials and exponential functions. Finally, we show that the integral of the absolute value of $\hat \alpha$ is finite as in Definition~\ref{def:regular:kernel}\ref{def:regular:kernel:integral}:
	\begin{align}
		\int_0^\infty |\hat \alpha(s)| \incr s
		&= \int_0^\infty \left|\sum_{m=0}^M c_m \ell_m(s)\right| \incr s
		\leq \sum_{m=0}^M |c_m| \int_0^\infty \ell_m(s) \incr s
		= \sum_{m=0}^M |c_m| \\
		&= \sum_{m=0}^M \left|\int_{s_m}^{s_{m+1}} \alpha(s) \incr s\right|
		\leq \int_0^{s_{M+1}} |\alpha(s)| \incr s \leq \int_0^\infty |\alpha(s)| \incr s
		= \rho(\infty) < \infty. \nonumber
	\end{align}
	Here, we have used Minkowski's inequality and that the Erlang kernels are non-negative and integrate to one.}
\end{proof}

\begin{proof}[Proof \chgtwo{of Proposition~\ref{thm:erlang:mixture:approximation}\ref{thm:erlang:mixture:approximation:pointwise:convergence}}]
	\chgtwo{In order to prove that the Erlang mixture approximation $\hat \alpha$ converges \chgthree{pointwise} to $\alpha$, we express it in terms of the delta family, $\delta_a: \Rnn \times \Rnn \rightarrow \Rnn$, defined by~\eqref{eq:erlang:mixture:delta:family} in Lemma~\ref{thm:erlang:mixture:delta:family}:
	\begin{align}
		\hat \alpha(t) &= \int_0^\infty \delta_a(t, s) \alpha(s) \incr s.
	\end{align}
	This equality holds according to Lemma~\ref{thm:erlang:mixture:delta:family}\ref{thm:erlang:mixture:delta:family:kernel:identity}, and since $\delta_a$ satisfies Lemma~\ref{thm:erlang:mixture:delta:family}\ref{thm:erlang:mixture:delta:family:integral}--\ref{thm:erlang:mixture:delta:family:sub:integral}, Lemma~\ref{thm:convergence:delta:family} guarantees that $\hat \alpha$ converges \chgthree{pointwise} to $\alpha$ for all $t \in \Rnn$.}
\end{proof}

\begin{proof}[Proof \chgtwo{of Proposition~\ref{thm:erlang:mixture:approximation}\ref{thm:erlang:mixture:approximation:integral:convergence}}]
	\chgtwo{First, we bound the absolute difference between the two integrals:
	\begin{align}
		\left|\int_0^t \hat \alpha(s) \incr s - \int_0^t \alpha(s) \incr s\right|
		&\leq \int_0^t |\hat \alpha(s) - \alpha(s)| \incr s
		 \leq t \sup_{s \in [0, t]} |\hat \alpha(s) - \alpha(s)|.
	\end{align}
	For finite $t$, the rightmost expression goes to zero as $a \rightarrow \infty$ and $M/a \rightarrow \infty$ \chgthree{because $\hat \alpha \rightarrow \alpha$ uniformly on bounded intervals of $t$ (see also the proof of Lemma~\ref{thm:convergence:delta:family})}. When $t$ is infinite, we split the integral at $t_0 \in \Rnn$:
	\begin{align}
		\left|\int_0^\infty \hat \alpha(s) \incr s - \int_0^\infty \alpha(s) \incr s\right|
		&\leq \int_0^{t_0} |\hat \alpha(s) - \alpha(s)| \incr s + \int_{t_0}^\infty |\hat \alpha(s) - \alpha(s)| \incr s.
	\end{align}
	As $\alpha$ is regular, the integral of its absolute value is finite. Furthermore, according to Lemma~\ref{lem:tight:erlang:mixture:approximation}, $\hat \alpha$ is tight, i.e., we can choose $t_0$ large enough that the second term is below $\epsilon/2$ for all values of $a$ and $M$. Next, we choose $a$ and $M$ large enough that the first term is below $\epsilon/2$ as well, which concludes the proof.}
\end{proof}

\begin{proof}[Proof \chgtwo{of Proposition~\ref{thm:erlang:mixture:approximation}\ref{thm:erlang:mixture:approximation:absolute:integral:convergence}}]
\chg{
	As $\alpha$ is regular, it is bounded. According to \chgtwo{Proposition}~\chgtwo{\ref{thm:erlang:mixture:approximation}\ref{thm:erlang:mixture:approximation:regularity}}, this means that $\hat \alpha$ is also bounded \chgtwo{because it is also regular}. Consequently, according to Lemma~\ref{lem:uniformly:integrable:erlang:mixture:approximation}, $\hat \alpha$ is uniformly integrable. Furthermore, according to Lemma~\ref{lem:tight:erlang:mixture:approximation} and \chgtwo{Proposition}~\ref{thm:erlang:mixture:approximation}\chgtwo{\ref{thm:erlang:mixture:approximation:pointwise:convergence}}, $\hat \alpha$ is also tight and converges pointwise to $\alpha$. Finally, according to Vitali's convergence lemma~\cite[Sec.~19.1, p.~399]{Royden:Fitzpatrick:2010}, uniform integrability, tightness\chgtwo{,} and pointwise convergence implies that $\hat \alpha$ also converges in the $L^1$-norm, i.e., the stated result.
}
\end{proof}

\subsection{\chgtwo{Uniform convergence}}\label{sec:ode:approximation:uniform:convergence}
\chgtwo{In this section, we prove Theorem~\ref{thm:ivp:approximation} on the uniform convergence of the solution to the IVP~\eqref{eq:lct:approximate:system:ODE}--\eqref{eq:lct:approximate:system:ODE:x0} \chgthree{in the Erlang ODE approximation} to the solution of the original IVP~\eqref{eq:system}--\eqref{eq:system:delay} involving DDEs.}

\chg{First, we introduce the state error, $e_x: \R \rightarrow \R^{n_x}$, the error of the memory state, $e_z: \R \rightarrow \R^{n_z}$, \chgthree{the error of the delayed variable, $e_r: \R \rightarrow \R^{n_r}$,} and the kernel error, $e_\alpha: \Rnn \rightarrow \R^{n_z \chmthree{\times n_r}}$, given by}
\begin{align}\label{eq:errors}
	e_x(t) &= \hat x(t) - x(t), &
	e_z(t) &= \hat z(t) - z(t), &
	e_r(t) &= \hat r(t) - r(t), &
	e_\alpha(t) &= \hat \alpha(t) - \alpha(t),
\end{align}
where $\hat x$\chgthree{,} $\hat z$\chgthree{, and $\hat r$} are the solution to~\eqref{eq:lct:approximate:system:ODE}--\eqref{eq:lct:approximate:system:ODE:x0}\chgthree{,} $x$\chgthree{,} $z$\chgthree{, and $r$} are the solution to \eqref{eq:system}--\eqref{eq:system:delay}\chgthree{, $\alpha$ is the kernel, and $\hat \alpha$ is the Erlang mixture approximation of $\alpha$. We let $\hat x(t) = x_0(t)$ and $\hat r(t) = r(t)$ for $t \in (-\infty, t_0]$}.
\begin{proof}[\chgtwo{Proof of Theorem~\ref{thm:ivp:approximation}}]
	\chg{
		\chgtwo{First, we prove that the solution, $\hat x$, to the IVP~\eqref{eq:lct:approximate:system:ODE}--\eqref{eq:lct:approximate:system:ODE:x0} \chgthree{in the Erlang ODE approximation} converges uniformly to that of the IVP~\eqref{eq:system}--\eqref{eq:system:delay} with DDEs.}
		According to \chgtwo{Proposition}~\ref{thm:erlang:mixture:approximation}\ref{thm:erlang:mixture:approximation:absolute:integral:convergence}, we can choose $\bar a$ \chgtwo{and $\bar M$} large enough that
		\begin{align}
			\int_0^\infty \|e_\alpha(t)\|_\infty \incr t &< \epsilon
		\end{align}
		when $a_{i\chmthree{j}} > \bar a$ \chgtwo{and $M_{i\chmthree{j}} > \bar M$} for $i = 1, \ldots, n_z$ \chgthree{and $j = 1, \ldots, n_r$} \chgtwo{provided that $\bar M/\bar a \rightarrow \infty$ as $\bar a \rightarrow \infty$}. Furthermore, since 1)~the delayed variable, $r$, is assumed to be bounded and \chgtwo{2})~the functions $f$ and $h$ are assumed to be differentiable (and thus also Lipschitz continuous), the result follows from a proof similar to that of Theorem~1 in the paper by Guglielmi and Hairer~\cite{Guglielmi:Hairer:2025}. For completeness, we present the details of the proof in Section~\ref{sec:proof:error:bound} of the supplementary material.
		\chgtwo{Specifically,}
		\begin{align}\label{eq:solution:bound}
			\|\chmtwo{e_x(t)}\|_\infty &\leq \epsilon w(t) \leq \epsilon w(t_f),
		\end{align}
		\chgtwo{where t}he auxiliary function $w: \R \rightarrow \Rnn$ is the solution to the DDE
		\begin{subequations}\label{eq:solution:bound:dde}
			\begin{align}
				\label{eq:solution:bound:dde:w0}
				w(t_0) &= 0, \\
				\label{eq:solution:bound:dde:w}
				\dot w(t) &= L_z K_r + L_x w(t) + L_z L_r \int_{t_0}^t \|\hat \alpha(t - s)\|_\infty w(s) \incr s, & t &\in [t_0, t_f],
			\end{align}
		\end{subequations}
		$L_x \in \Rp$ and $L_z \in \Rp$ are the Lipschitz constants of the first and second arguments of $f$, and $L_r \in \Rp$ is the Lipschitz constant of $h$. Furthermore, $K_r$ is the supremum of \chgtwo{the absolute value of} $r$ in the time interval $(-\infty, t_f]$ (see also the assumptions in Section~\ref{sec:system}). \chgtwo{T}he second inequality in~\eqref{eq:solution:bound} follows from the non-negativity of the right-hand side of the DDE~\eqref{eq:solution:bound:dde:w} \chgtwo{and ensures that the convergence is uniform on the time interval $[t_0, t_f]$. In conclusion, the state error, $e_x$, can be made arbitrarily small by making the kernel error sufficiently small.}

		\chgtwo{Next, we prove that $\hat z$ and $\hat r$ converge uniformly. As $h$ is \chgthree{Lipschitz continuous}, $\chgthree{\|}e_r(t)\chgthree{\|_\infty} \leq L_r \chgthree{\|}e_x(t)\chgthree{\|_\infty} \leq L_r \epsilon w(t_f)$ \chgthree{for $t \in [t_0, t_f]$}, and therefore, the convergence of $\hat r$ is uniform. \chgthree{As the Erlang mixture approximation, $\hat \alpha$, is regular (Proposition~\ref{thm:erlang:mixture:approximation}\ref{thm:erlang:mixture:approximation:regularity}), the integral of its absolute value is bounded, and specifically, $\int_0^\infty |\hat \alpha_{ij}(t)| \incr t \leq \rho_{ij}(\infty)$}
		which is finite by assumption \chgthree{(see the proof of Proposition~\ref{thm:erlang:mixture:approximation}\ref{thm:erlang:mixture:approximation:integral:convergence} in Section~\ref{sec:erlang:mixture:kernel:convergence})}.
		\chgthree{We use this to derive the following bound:
		\begin{align}
			\int_0^\infty \|\hat \alpha(t)\|_\infty \incr s
			&= \int_0^\infty \max_{i = 1, \ldots, n_z} \sum_{j=1}^{n_r} |\hat \alpha_{ij}(t)| \incr t
			 \leq \sum_{i=1}^{n_z} \sum_{j=1}^{n_r} \int_0^\infty |\hat \alpha_{ij}(t)| \incr t
			 \leq \sum_{i=1}^{n_z} \sum_{j=1}^{n_r} \rho_{ij}(\infty) \\
			&\leq \sum_{i=1}^{n_z} \max_{k = 1, \ldots, n_z} \sum_{j=1}^{n_r} \rho_{kj}(\infty)
			 = n_z \|\rho(\infty)\|_\infty. \nonumber
		\end{align}
		}%
		Using these two bounds, we derive a bound on the absolute memory state error:}
		\begin{align}\label{eq:memory:state:error:rewrite}
			\chmthree{
			\|e_z(t)\|_\infty
			&= \left\|\int_{-\infty}^t \hat \alpha(t - s) \hat r(s) \incr s - \int_{-\infty}^t \alpha(t - s) r(s) \incr s \right\|_\infty \\
			&= \left\|\int_{-\infty}^t \hat \alpha(t - s) e_r(s) + e_\alpha(t - s) r(s) \incr s \right\|_\infty \nonumber \\
			&\leq \int_{-\infty}^t \|\hat \alpha(t - s) e_r(s) + e_\alpha(t - s) r(s) \|_\infty \incr s \nonumber \\
			&\leq \int_{-\infty}^t \|\hat \alpha(t - s)\|_\infty \|e_r(s)\|_\infty \incr s + \int_{-\infty}^t \|e_\alpha(t - s)\|_\infty \|r(s)\|_\infty \incr s \nonumber \\
			&\leq \int_{-\infty}^t \|\hat \alpha(t - s)\|_\infty \incr s \sup_{s \in (-\infty, t]} \|e_r(s)\|_\infty + \int_{-\infty}^t \|e_\alpha(t - s)\|_\infty \incr s \sup_{s \in (-\infty, t]} \|r(s)\|_\infty \nonumber \\
			&\leq (n_z \|\rho(\infty)\|_\infty L_r w(t_f) + K_r) \epsilon.} \nonumber
		\end{align}
		\chgtwo{From the first to the second equality, we add and subtract} \chgthree{$\hat \alpha(t - s) r(s)$} (i.e., a mixed term).
		\chgtwo{In conclusion, we can bound the pointwise error, and the bound is independent of $t$ for $t \in \chmthree{[t_0}, t_f]$, i.e., the convergence is uniform over that interval.}
	}%
\end{proof}

\subsection{Steady states and stability criteria}\label{sec:lct:steady:state:stability}
\chg{In this section, we \chgtwo{show} that the steady state of the approximate system of ODEs \chgtwo{converges to} that of the original system, and we present the steady state stability criterion. In Section~\ref{sec:lct:convergence:roots}, we show that the roots involved in the criterion converge to those in the criterion for the original system of DDEs.}
\begin{lemma}\label{thm:lct:approximate:system:ODE:steady:state}
	A steady state, $\bar x \in \R^{n_x}$, of the \chgthree{Erlang ODE approximation}~\eqref{eq:lct:approximate:system:ODE} satisfies
	\begin{align}\label{eq:lct:steady:state}
		0 &= f(\bar x, \bar z), & \bar z &= \chmtwo{d} \bar r\chmtwo{,} & \chmtwo{\bar r} &= h(\bar x),
	\end{align}
	\chgtwo{where $d \in \R^{n_z \chmthree{\times n_r}}$ and $d_{i\chmthree{j}} = \sum_{m=0}^{M_{i\chmthree{j}}} c_{mi\chmthree{j}}$ for $i = 1, \ldots, n_z$ \chgthree{and $j = 1, \ldots, n_r$}.}
\end{lemma}
\begin{proof}
	In steady state,
	\begin{align}
		0 &= f(\bar x, \bar z), & 0 &= A \bar Z + B \bar r.
	\end{align}
	Consequently,
	\begin{align}
		\bar Z &= -A^{-1} B \bar r, & \bar z &= C \bar Z = -C A^{-1} B \bar r.
	\end{align}
	Note that each block-diagonal element of $A$ is a lower bidiagonal matrix with nonzero elements in the diagonal and the first subdiagonal. Therefore, $A$ is invertible. Specifically,
	\begin{subequations}\label{eq:lct:ode:steady:state:proof}
		\begin{align}
			\label{eq:lct:ode:steady:state:proof:ainv}
			A^{-1} &= \blkdiag\left(\left(A^{(\chmthree{1,}1)}\right)^{-1}, \chmthree{\left(A^{(1,2)}\right)^{-1},} \ldots, \left(A^{(n_z\chmthree{, n_r} )}\right)^{-1}\right), \\
			\label{eq:lct:ode:steady:state:proof:cainvb}
			\chmthree{\left(}C A^{-1} B\chmthree{\right)_{ij}} &= c^{(\chmthree{ij})} \left(A^{(\chmthree{ij})}\right)^{-1} b^{(\chmthree{ij})}.
		\end{align}
	\end{subequations}
	The inverse of $A^{(i\chmthree{j})}$ is a multiple of a lower triangular matrix of ones\chgthree{,} and the elements \chgthree{in}~\eqref{eq:lct:ode:steady:state:proof:cainvb} can be derived analytically:
	\begin{align}
		\left(A^{(i\chmthree{j})}\right)^{-1} &= \frac{-1}{a\chg{_{i\chmthree{j}}}}
		\begin{bmatrix}
			1 \\
			1 & 1 \\
			\vdots & & \ddots \\
			1 & 1 & \cdots & 1
		\end{bmatrix}, &
		c^{(i\chmthree{j})} \left(A^{(i\chmthree{j})}\right)^{-1} b^{(i\chmthree{j})} &= -\sum_{m=0}^{M_{i\chmthree{j}}} c_{mi\chmthree{j}}.
	\end{align}
	Consequently, $-C A^{-1} B = \chmtwo{d}$, and
	\begin{align}
		\bar z &= \chmtwo{d} \bar r\chmtwo{,} & \chmtwo{\bar r} &= h(\bar x)\chmtwo{.}
	\end{align}
\end{proof}
\begin{proof}[\chgtwo{Proof of Theorem~\ref{thm:ode:approximation}\ref{thm:ode:approximation:steady:state}}]
	\chgtwo{First, according to Proposition~\ref{thm:erlang:mixture:approximation}\ref{thm:erlang:mixture:approximation:integral:convergence},
	\begin{align}
		d_{i\chmthree{j}} = \sum_{m=0}^{M_{i\chmthree{j}}} c_{mi\chmthree{j}} = \int_0^{\chmthree{s_{M_{ij}+1}}} \hat \alpha_{i\chmthree{j}}(s) \incr s \rightarrow \beta_{i\chmthree{j}}(\infty)
	\end{align}
	as $a_{i\chmthree{j}} \rightarrow \infty$ when $M_{i\chmthree{j}}/a_{i\chmthree{j}} \rightarrow \infty$ as well. Next, the implicit function theorem~\cite[Chap.~7, Thm. 3]{Protter:Morrey:1985} guarantees that the solution to the steady state equations~\eqref{eq:steady:state} for the DDEs~\eqref{eq:system:x}--\eqref{eq:system:delay} is a locally unique and differentiable function if the Jacobian with respect to $\bar x$ is invertible. Consequently, in that case, the solution to the steady state equations~\eqref{eq:lct:steady:state} for the ODEs~\eqref{eq:lct:approximate:system:ODE} converges to that of the DDEs~\eqref{eq:system:x}--\eqref{eq:system:delay} \chgthree{if it exists}. Furthermore, the Jacobian is given by
	\begin{align}
		\diff{f}{\bar x}(\bar x, \bar z) &= \pdiff{f}{x}(\bar x, \bar z) + \pdiff{f}{z}(\bar x, \bar z) \pdiff{\bar z}{\bar x}, &
		\pdiff{\bar z}{\bar x} &= \beta(\infty) \pdiff{\bar r}{\bar x}, & \pdiff{\bar r}{\bar x} &= \pdiff{h}{x}(\bar x),
	\end{align}
	and the matrix in Theorem~\ref{thm:ode:approximation}\ref{thm:ode:approximation:steady:state} is obtained by substituting the matrices $F$, $G$, and $H$ from~\eqref{eq:jacobians}.}
\end{proof}
\begin{lemma}\label{thm:lct:stability}
	The \chgthree{Erlang ODE approximation}~\eqref{eq:lct:approximate:system:ODE} is locally asymptotically stable around a steady state, $\bar x$, that satisfies~\eqref{eq:lct:steady:state} if $\real \lambda < 0$ for all eigenvalues $\lambda \in \C$ of the Jacobian matrix
	\begin{align}\label{eq:lct:ode:jacobian}
		J &=
		\begin{bmatrix}
			\chmtwo{\hat F} & \chmtwo{\hat G} C \\
			B \chmtwo{\hat H} & A
		\end{bmatrix},
	\end{align}
	where \chgtwo{$\hat F \in \R^{n_x \times n_x}$, $\hat G \in \R^{n_x \times n_z}$, and $\hat H \in \R^{n_{\chmthree{r}} \times n_x}$ are the Jacobian matrices~\eqref{eq:jacobians} evaluated in the steady state \chgthree{of} the \chgthree{Erlang ODE approximation}.}
	Furthermore, \chg{the eigenvalues, $\lambda$, that are not equal to $-a_{i\chmthree{j}}$ for $i = 1, \ldots, n_z$ and $j = 1, \ldots, n_r$ will also satisfy the reduced characteristic equation}
	\begin{align}\label{eq:lct:stability:characteristic:equation}
		\chg{\hat P(\lambda) =} \det(\chmtwo{\hat F} - \lambda I + \chmtwo{\hat G} Q(\lambda) \chmtwo{\hat H}) &= 0, &
		Q_{ij}(\lambda) &= \sum\limits_{m=0}^{M_{i\chmthree{j}}} c_{mi\chmthree{j}} \left(\frac{a_{i\chmthree{j}}}{a_{i\chmthree{j}} + \lambda}\right)^{m+1},
	\end{align}
	where \chg{$\hat P: \C \rightarrow \C$ is the approximate characteristic function and} $Q: \C \rightarrow \C^{n_z \times n_{\chmthree{r}}}$ \chg{is an auxiliary matrix}.
\end{lemma}
\begin{proof}
	An eigenvalue, $\lambda$, is a root of the characteristic polynomial
	\begin{align}\label{eq:lct:stability:proof:determinant}
		\det(J - \lambda I)
		&=
		\det
		\begin{bmatrix}
			\chmtwo{\hat F} - \lambda I & \chmtwo{\hat G} C \\
			B \chmtwo{\hat H} & A - \lambda I
		\end{bmatrix} \\
		&=
		\det\left(\chmtwo{\hat F} - \lambda I - \chmtwo{\hat G} C (A - \lambda I)^{-1} B \chmtwo{\hat H}\right) \det(A - \lambda I), \nonumber
	\end{align}
	where we have used Schur's determinant formula (see, e.g., \cite{Zhang:2005})\chg{, which presumes that $\det(A - \lambda I) \neq 0$}. Next, we exploit the block-diagonal structure of $A$:
	\begin{align}
		A - \lambda I &= \blkdiag\left(A^{(\chmthree{1,}1)} - \lambda I^{(\chmthree{1,}1)}, \chmthree{A^{(1,2)} - \lambda I^{(1,2)},} \ldots, A^{(n_z\chmthree{, n_r})} - \lambda I^{(n_z\chmthree{, n_r})}\right).
	\end{align}
	Consequently,
	\begin{align}
		\det(A - \lambda I) &= \prod_{i=1}^{n_z} \chmthree{\prod_{j=1}^{n_r}} \det (A^{(i\chmthree{j})} - \lambda I^{(i\chmthree{j})}) = \prod_{i=1}^{n_z} \chmthree{\prod_{j=1}^{n_r}} (-1)^{M_{i\chmthree{j}}\chmthree{+1}} (a_{i\chmthree{j}} + \lambda)^{M_{i\chmthree{j}}\chmthree{+1}},
	\end{align}
	\chg{which is nonzero by assumption such that $A - \lambda I$ is invertible.} Next, we consider the matrix \chgthree{elements}
	\begin{align}\label{eq:lct:stability:proof:camliinvb}
		\chmthree{\left(}C (A - \lambda I)^{-1} B\chmthree{\right)_{ij}} &= c^{(\chmthree{ij})} \left(A^{(\chmthree{ij})} - \lambda I^{(\chmthree{ij})}\right)^{-1} b^{(\chmthree{ij})}.
	\end{align}
	First, we rewrite the matrix
	\begin{align}
		A^{(i\chmthree{j})} - \lambda I^{(i\chmthree{j})} &= a_{i\chmthree{j}} \left(L^{(i\chmthree{j})} - I^{(i\chmthree{j})}\right) - \lambda I^{(i\chmthree{j})} = -(a_{i\chmthree{j}} + \lambda)\left(\frac{-a_{i\chmthree{j}}}{a_{i\chmthree{j}} + \lambda} L^{(i\chmthree{j})} + I^{(i\chmthree{j})}\right).
	\end{align}
	Consequently, the inverse is
	\begin{align}
		\left(A^{(i\chmthree{j})} - \lambda I^{(i\chmthree{j})}\right)^{-1} &= \frac{-1}{a_{i\chmthree{j}} + \lambda}
		\begin{bmatrix}
			1 \\
			\frac{a_{i\chmthree{j}}}{a_{i\chmthree{j}} + \lambda} & 1 \\
			\vdots & & \ddots \\
			\left(\frac{a_{i\chmthree{j}}}{a_{i\chmthree{j}} + \lambda}\right)^{M_{i\chmthree{j}}} & \left(\frac{a_{i\chmthree{j}}}{a_{i\chmthree{j}} + \lambda}\right)^{M_{i\chmthree{j}}-1} & \cdots & 1
		\end{bmatrix},
	\end{align}
	as can be verified by direct calculation, and the elements of the matrix in~\eqref{eq:lct:stability:proof:camliinvb} are
	\begin{align}
		c^{(i\chmthree{j})} \left(A^{(i\chmthree{j})} - \lambda I^{(i\chmthree{j})}\right)^{-1} b^{(i\chmthree{j})} &= -\sum_{m=0}^{M_{i\chmthree{j}}} c_{mi\chmthree{j}} \left(\frac{a_{i\chmthree{j}}}{a_{i\chmthree{j}} + \lambda}\right)^{m+1} = -Q_{i\chmthree{j}}(\lambda).
	\end{align}
	Finally, we substitute into the left determinant on the right-hand side of~\eqref{eq:lct:stability:proof:determinant} and obtain
	\begin{align}
		\det(\chmtwo{\hat F} - \lambda I + \chmtwo{\hat G} Q(\lambda) \chmtwo{\hat H}) &= 0.
	\end{align}
\end{proof}
\begin{corollary}\label{lem:lct:approximate:characteristic:function}
	\chg{
		Let each element of $\hat \alpha: \chmtwo{\Rnn} \rightarrow \R^{n_z \chmthree{\times n_r}}$ be an Erlang mixture kernel. Then, the approximate characteristic function $\hat P: \C \rightarrow \C$ and the auxiliary matrix $Q: \C \rightarrow \C^{n_z \times n_{\chmthree{r}}}$ defined in~\eqref{eq:lct:stability:characteristic:equation} in \chgtwo{Lemma}~\ref{thm:lct:stability} are equal to
		\begin{align}\label{eq:lct:approximate:characteristic:function}
			\hat P(\lambda) &= \det(\chmtwo{\hat F} - \lambda I + \chmtwo{\hat G} \int_0^\infty e^{-\lambda s} \hat \alpha(s) \incr s \chmtwo{\hat H}), &
			Q_{i\chmthree{j}}(\lambda) &= \int_0^\infty e^{-\lambda s} \hat \alpha_{i\chmthree{j}}(s) \incr s,
		\end{align}
		for $i = 1, \ldots, n_z$ \chgthree{and $j = 1, \ldots, n_r$}.
	}
\end{corollary}
\begin{proof}
	First, we write out the following integral for an Erlang kernel of order $m$ with rate parameter $a_{i\chmthree{j}}$:
	\begin{align}
		\int_0^\infty e^{-\lambda s} \ell_{mi\chmthree{j}}(s) \incr s &= b_{mi\chmthree{j}} \int_0^\infty s^m e^{-(a_{i\chmthree{j}} + \lambda) s} \incr s = b_{mi\chmthree{j}} \frac{m!}{(a_{i\chmthree{j}} + \lambda)^{m+1}} = \left(\frac{a_{i\chmthree{j}}}{a_{i\chmthree{j}} + \lambda}\right)^{m+1}.
	\end{align}
	This follows from the fact that the integral of an Erlang kernel of order $m$ with rate parameter $a_{i\chmthree{j}} + \lambda$ is one. Next, we use this result to write out the following integral involving an Erlang mixture kernel of order $M_{i\chmthree{j}}$ with rate parameter $a_{i\chmthree{j}}$:
	\begin{align}
		\int_0^\infty e^{-\lambda s} \hat \alpha_{i\chmthree{j}}(s) \incr s &= \sum_{m=0}^{M_{i\chmthree{j}}} c_{mi\chmthree{j}} \int_0^\infty e^{-\lambda s} \ell_{mi\chmthree{j}}(s) \incr s = \sum_{m=0}^{M_{i\chmthree{j}}} c_{mi\chmthree{j}} \left(\frac{a_{i\chmthree{j}}}{a_{i\chmthree{j}} + \lambda}\right)^{m+1} = Q_{i\chmthree{j}}(\lambda).
	\end{align}
	Consequently,
	\begin{align}
		\chmtwo{\hat G} \int_0^\infty e^{-\lambda s} \hat \alpha(s) \incr s \chmtwo{\hat H}
		&= \chmtwo{\hat G} Q(\lambda) \chmtwo{\hat H},
	\end{align}
	and the characteristic \chg{function in}~\eqref{eq:lct:approximate:characteristic:function} becomes
	\begin{align}
		\chg{\hat P(\lambda) =} \det(\chmtwo{\hat F} - \lambda I + \chmtwo{\hat G} Q(\lambda) \chmtwo{\hat H}).
	\end{align}
\end{proof}

\subsection{Convergence of the roots}\label{sec:lct:convergence:roots}
\chg{In this section, we use Hurwitz' convergence theorem~\cite[Thm.~5.6.4]{Hahn:Epstein:1996} to prove that the roots of the approximate characteristic function, $\hat P: \C \rightarrow \C$ in \chgtwo{Lemma}~\ref{thm:lct:stability}, converge to \chgtwo{the isolated roots} of $P: \C \rightarrow \C$ in Corollary~\ref{thm:stability} as the rate parameters \chgtwo{and order} in the Erlang mixture approximation go to infinity. That is, we show that stability analyses based on the \chgthree{Erlang ODE approximation} \chgtwo{can be used to infer information about the stability of the original DDEs}.}
\begin{lemma}\label{lem:approximate:characteristic:function:convergence}
	\chg{
		Let each element of $\alpha: \Rnn \rightarrow \chgtwo{\R}^{n_z \chmthree{\times n_r}}$ be a regular and exponentially bounded kernel, i.e., let there exist $L, \rho \in \Rp$ such that $\chmtwo{|}\alpha_{i\chmthree{j}}(t)\chmtwo{|} \leq L e^{-\rho t}$ for all $t \in \Rnn$ and for $i = 1, \ldots, n_z$ \chgthree{and $j = 1, \ldots, n_r$}. Furthermore, let the $i\chmthree{,j}$'th element of $\hat \alpha: \Rnn \rightarrow \chgtwo{\R}^{n_z \chmthree{\times n_r}}$ be \chgtwo{the} Erlang mixture approximation \chgtwo{of $\alpha_{i\chmthree{j}}$ of order $M_{i\chmthree{j}} \in \Nnn$} with rate parameter $a_{i\chmthree{j}} \in \Rp$.
		Then, for a given $\lambda_{\min} \in \Rnp$ for which $-\rho < \lambda_{\min}$, the approximate characteristic function $\hat P: \C \rightarrow \C$ defined in~\eqref{eq:lct:stability:characteristic:equation} converges uniformly to the characteristic function $P: \C \rightarrow \C$ defined in~\eqref{eq:characteristic:equation} for every $\lambda \in \C$ for which $\real \lambda \geq \lambda_{\min}$ as $a_{\chmtwo{i}\chmthree{j}} \rightarrow \infty$ \chgtwo{if $M_{i\chmthree{j}}/a_{i\chmthree{j}} \rightarrow \infty$ as well for $i = 1, \ldots, n_z$ \chgthree{and $j = 1, \ldots, n_r$}}.
	}
\end{lemma}
\begin{proof}
	\chg{%
		Both $\hat P$ and $P$ are determinants of characteristic matrices. Consequently, since the determinant is continuous, $\hat P$ converges to $P$ if the corresponding characteristic matrix converges.
		First, we bound the following expression and split the integral,
		\begin{multline}
			|Q_{i\chmthree{j}}(\lambda) - \int_0^\infty e^{-\lambda s} \alpha_{i\chmthree{j}}(s) \incr s|
			= \left|\int_0^\infty e^{-\lambda s} \hat \alpha_{i\chmthree{j}}(s) \incr s - \int_0^\infty e^{-\lambda s} \alpha_{i\chmthree{j}}(s) \incr s\right| \\
			\leq \int_0^\infty |e^{-\lambda s}| |\hat \alpha_{i\chmthree{j}}(s) - \alpha_{i\chmthree{j}}(s)| \incr s \\
			= \int_0^{s_0} e^{-\real \lambda s} |\hat \alpha_{i\chmthree{j}}(s) - \alpha_{i\chmthree{j}}(s)| \incr s + \int_{s_0}^\infty e^{-\real \lambda s} |\hat \alpha_{i\chmthree{j}}(s) - \alpha_{i\chmthree{j}}(s)| \incr s.
		\end{multline}
		Here, we have used the expression for $Q_{i\chmthree{j}}$ from Corollary~\ref{lem:lct:approximate:characteristic:function}.
		We consider the second integral first. The exponential boundedness of $\alpha$ implies that $\hat \alpha$ is also exponentially bounded (see \chgtwo{Corollary~\ref{thm:erlang:mixture:approximation:exponential:boundedness}}) such that
		\begin{align}
			\int_{s_0}^\infty e^{-\real \lambda s} |\hat \alpha_{i\chmthree{j}}(s) - \alpha_{i\chmthree{j}}(s)| \incr s
			&\leq \int_{s_0}^\infty e^{-\lambda_{\min} s} 2 L e^{-\bar \rho s} \incr s
			= 2 L \int_{s_0}^\infty e^{-(\lambda_{\min} + \bar \rho) s} \incr s \\
			&= \frac{2 L}{\lambda_{\min} + \bar \rho} e^{-(\lambda_{\min} + \bar \rho) s_0}. \nonumber
		\end{align}
		By assumption, $\lambda_{\min} > -\rho$, and according to \chgtwo{Corollary~\ref{thm:erlang:mixture:approximation:exponential:boundedness}}, we can always choose $\bar a \in \Rp$ \chgtwo{and $\bar M \in \Nnn$ (for which $\bar M/\bar a \rightarrow \infty$ as $\bar a \rightarrow \infty$)} such that $\bar \rho$ is arbitrarily close to $\rho$ for all $a_{i\chmthree{j}} > \bar a$ \chgtwo{and $M_{i\chmthree{j}} > \bar M$}, i.e., such that $\lambda_{\min} + \bar \rho > 0$.
		Next, we bound the first integral, and we exploit the non-positivity of $\lambda_{\min}$:
		\begin{align}
			\int_0^{s_0} e^{-\real \lambda s} |\hat \alpha_{i\chmthree{j}}(s) - \alpha_{i\chmthree{j}}(s)| \incr s
			&\leq \int_0^{s_0} e^{-\lambda_{\min} s} |\hat \alpha_{i\chmthree{j}}(s) - \alpha_{i\chmthree{j}}(s)| \incr s \\
			&\leq e^{-\lambda_{\min} s_0} \int_0^{s_0} |\hat \alpha_{i\chmthree{j}}(s) - \alpha_{i\chmthree{j}}(s)| \incr s. \nonumber
		\end{align}
		According to \chgtwo{Proposition}~\ref{thm:erlang:mixture:approximation}\ref{thm:erlang:mixture:approximation:absolute:integral:convergence} \chgtwo{and because the integrand is non-negative}, we can choose $\bar a$ \chgtwo{and $\bar M$ (for which $\bar M/\bar a \rightarrow \infty$ as $\bar a \rightarrow \infty$)} large enough that the last integral is arbitrarily small \chgtwo{for $a_{i\chmthree{j}} > \bar a$ and $M_{i\chmthree{j}} > \bar M$}. Finally, we bound the norm difference between the characteristic matrices in the expressions for $\hat P$ and $P$ in~\eqref{eq:lct:approximate:characteristic:function} and~\eqref{eq:characteristic:equation}, respectively, and we use the $\infty$-norm for convenience:
		\begin{align}
			\left\|\chmtwo{\hat F} - \lambda I + \chmtwo{\hat G} \int_0^\infty e^{-\lambda s} \hat \alpha(s) \incr s \chmtwo{\hat H} - \left(F - \lambda I + G \int_0^\infty e^{-\lambda s} \alpha(s) \incr s H\right)\right\|_\infty \\
			\chmtwo{
			= \Bigg\|\hat F - F + \hat G \int_0^\infty e^{-\lambda s} \hat \alpha(s) \incr s (\hat H - H)
			} \nonumber \\
			\chmtwo{
				+ \hat G \int_0^\infty e^{-\lambda s} (\hat \alpha(s) - \alpha(s)) \incr s H + (\hat G - G) \int_0^\infty e^{-\lambda s} \alpha(s) \incr s H \Bigg\|_\infty
		 	} \nonumber \\
		 	\chmtwo{
			\leq \|e_F\|_\infty + \|\hat G\|_\infty \int_0^\infty |e^{-\lambda s}| \|\hat \alpha(s)\|_\infty \incr s \|e_H\|_\infty \nonumber \\
			+ \|\hat G\|_\infty \left\|\int_0^\infty e^{-\lambda s} (\hat \alpha(s) - \alpha(s)) \incr s\right\|_\infty \|H\|_\infty + \|e_G\| \int_0^\infty |e^{-\lambda s}| \|\alpha(s)\|_\infty \incr s \|H\|_\infty
			} \nonumber
		\end{align}
		\chgtwo{Here, $e_F = \hat F - F \in \R^{n_x \times n_x}$, $e_G = \hat G - G \in \R^{n_x \times n_z}$, and $e_H = \hat H - H \in \R^{n_{\chmthree{r}} \times n_x}$. As $f$ and $h$ are assumed to be continuously differentiable, the Jacobians, $F$, $G$, and $H$, are continuous and since the steady state of the approximate system converges to that of the original system as $a_{i\chmthree{j}} \rightarrow \infty$ and $M_{i\chmthree{j}}/a_{i\chmthree{j}} \rightarrow \infty$, the errors $e_F$, $e_G$, and $e_H$ converge to zero. The norm of the integral in the third term is
		\begin{align}
			\left\|\int_0^\infty e^{-\lambda s} (\hat \alpha(s) - \alpha(s)) \incr s\right\|_\infty
			&= \max_{i = 1, \ldots, n_z} \chmthree{\sum_{j=1}^{n_r}} \left|\int_0^\infty e^{-\lambda s} (\hat \alpha_{i\chmthree{j}}(s) - \alpha_{i\chmthree{j}}(s)) \incr s\right| \\
			&\leq \max_{i = 1, \ldots, n_z} \chmthree{\sum_{j=1}^{n_r}} \int_0^\infty |e^{-\lambda s}| |\hat \alpha_{i\chmthree{j}}(s) - \alpha_{i\chmthree{j}}(s)| \incr s. \nonumber
		\end{align}
		}%
		As we have shown above, for any $\lambda \in \C$ for which $\real \lambda \chmthree{\geq} \lambda_{\min}$, we can choose $s_0$\chgtwo{,} $\bar a$\chgtwo{, and $\bar M$ (for which $\bar M/\bar a \rightarrow \infty$ as $\bar a \rightarrow \infty$)} such that this expression is below any $\epsilon \in \Rp$ whenever $a_{i\chmthree{j}} > \bar a$ \chgtwo{and $M_{i\chmthree{j}} > \bar M$} for $i = 1, \ldots, n_z$ \chgthree{and $j = 1, \ldots, n_r$}, which concludes the proof.
	}
\end{proof}
\begin{lemma}\label{lem:approximate:characteristic:function:holomorphic}
	\chg{
		Let the elements of $\alpha: \Rnn \rightarrow \chgtwo{\R}^{n_z \chmthree{\times n_r}}$ be exponentially bounded and regular, i.e., let there exist $L, \rho \in \Rp$ such that $\chmtwo{|}\alpha_{i\chmthree{j}}(t)\chmtwo{|} \leq L e^{-\rho t}$ for all $t \in \Rnn$ and $i = 1, \ldots, n_z$ \chgthree{and $j = 1, \ldots, n_r$}, and let the $i\chmthree{,j}$'th element of $\hat \alpha: \Rnn \rightarrow \chgtwo{\R}^{n_z \chmthree{\times n_r}}$ be \chgtwo{the} Erlang mixture \chgtwo{approximation of $\alpha_{i\chmthree{j}}$} of order \chgtwo{$M_{i\chmthree{j}} \in \Nnn$} with rate parameter $a_{i\chmthree{j}} \in \Rp$.
		Then, the characteristic function $\hat P: \C \rightarrow \C$ defined by~\eqref{eq:lct:stability:characteristic:equation} is holomorphic for $\lambda \in \C$ for which $\real \lambda > -\rho$.
	}
\end{lemma}
\begin{proof}
	\chg{
		We prove that the complex derivatives of the \chgthree{elements} of the auxiliary matrix $Q: \C \rightarrow \C^{n_z \times n_{\chmthree{r}}}$ defined in~\eqref{eq:lct:stability:characteristic:equation} exist for an arbitrary point and direction $\lambda_0, \Delta \lambda \in \C$, i.e., for $\lambda = \lambda_0 + w \Delta \lambda$ where $w \in \Rnn$. We derive the derivative from the expression in Corollary~\ref{lem:lct:approximate:characteristic:function}:
		\begin{align}
			\pdiff{Q_{i\chmthree{j}}}{w}(\lambda) &= -\int_0^\infty \pdiff{\lambda}{w} s e^{-\lambda s} \hat \alpha_{i\chmthree{j}}(s) \incr s = -\Delta \lambda \int_0^\infty s e^{-\lambda s} \hat \alpha_{i\chmthree{j}}(s) \incr s.
		\end{align}
		The derivative exists if the integral is finite. According to \chgtwo{Corollary}~\ref{thm:erlang:mixture:approximation:exponential:boundedness}, $\hat \alpha_{i\chmthree{j}}$ is exponentially bounded if $\alpha_{i\chmthree{j}}$ is. Specifically, $\hat \alpha_{i\chmthree{j}}(t) \leq L e^{-\bar \rho t}$ for all $t$ where $0 \leq \bar \rho < \rho$. Consequently,
		\begin{align}
			\left|\int_0^\infty s e^{-\lambda s} \hat \alpha_{i\chmthree{j}}(s) \incr s\right| &\leq \int_0^\infty s e^{-\real \lambda s} L e^{-\bar \rho s} \incr s = L \int_0^\infty s e^{-(\real \lambda + \bar \rho) s} \incr s = \frac{L}{(\real \lambda + \bar \rho)^2}
		\end{align}
		for $\real \lambda > -\bar \rho$. Furthermore, for any $\lambda$ for which $\real \lambda > -\rho$, we can choose $\bar a$ \chgtwo{and $\bar M$ (for which $\bar M/\bar a \rightarrow \infty$ as $\bar a \rightarrow \infty$)} large enough that $\real \lambda > -\bar \rho > -\rho$ whenever $a_{i\chmthree{j}} > \bar a$ \chgtwo{and $M_{i\chmthree{j}} > \bar M$} for $i = 1, \ldots, n_z$ \chgthree{and $j = 1, \ldots, n_r$}. Consequently, the derivative exists when $\real \lambda > -\rho$. Finally, as the determinant of a holomorphic function is also holomorphic and the remaining terms in the argument of the determinant in the expression for $\hat P$ in~\eqref{eq:lct:stability:characteristic:equation} are constant and linear in $\lambda$, $\hat P$ is holomorphic when $\real \lambda > \chmthree{-}\rho$.
	}
\end{proof}
\begin{proof}[\chgtwo{Proof of Theorem~\ref{thm:ode:approximation}\ref{thm:ode:approximation:roots}}]
	\chg{
		According to Lemma~\ref{lem:approximate:characteristic:function:holomorphic} and~\ref{lem:approximate:characteristic:function:convergence}, respectively, $\hat P$ is holomorphic and converges \chgtwo{uniformly} to $P$ for $\lambda \in \C$ with real part larger than or equal to $\lambda_{\min}$. Consequently, the convergence of the roots \chgtwo{of $\hat P$} follows from Hurwitz' convergence theorem~\cite[Thm.~5.6.4]{Hahn:Epstein:1996} since the roots are isolated by assumption.
	}
\end{proof}

\subsection{\chg{Error dynamics}}\label{sec:lct:error:dynamics}
\chg{In this section, we analyze the difference between the solution to the \chgtwo{IVP~\eqref{eq:lct:approximate:system:ODE}--\eqref{eq:lct:approximate:system:ODE:x0}} \chgthree{in the Erlang ODE approximation} and that of the original \chgtwo{IVP}~\eqref{eq:system}--\eqref{eq:system:delay} \chgtwo{involving DDEs}, and we show that it converges to zero \chgthree{as $t \rightarrow \infty$ and }as the Erlang mixture approximation \chgthree{converges to the true kernel if the assumptions in Theorem~\ref{thm:ode:approximation}\ref{thm:ode:approximation:stability} are satisfied}.}

The state error \chgthree{defined in~\eqref{eq:errors}} satisfies the \chgtwo{IVP}
\begin{subequations}\label{eq:error:system}
	\begin{align}
		\label{eq:error:system:e0}
		e_x(t) &= 0, & t &\in (-\infty, t_0], \\
		\label{eq:error:system:e}
		\dot e_x(t) &= f(\hat x(t), \hat z(t)) - f(x(t), z(t)), & t &\in [t_0, t_f].
	\end{align}
\end{subequations}
\begin{proof}[\chgtwo{Proof of Theorem~\ref{thm:ode:approximation}\ref{thm:ode:approximation:stability}}]
	\chg{The steady states of the original and approximate system are \chgtwo{not} identical (see \chgtwo{Lemma}~\ref{thm:lct:approximate:system:ODE:steady:state}). \chgtwo{Consequently, the steady state of the error, $\bar e_x \in \R^{n_x}$, will not be zero.}
	The linearized system corresponding to the DDE~\eqref{eq:error:system:e} describing the state error is
	\begin{align}
		\dot E_x(t) &= \hat F \hat X(t) + \hat G \int_{-\infty}^t \hat \alpha(t - s) \hat H \hat X(s) \incr s - \left(F X(t) + G \int_{-\infty}^t \alpha(t - s) H X(s) \incr s\right),
	\end{align}
	\chgtwo{where $E_x: \R \rightarrow \R^{n_x}$ is the deviation from the steady state error, i.e., $E_x(t) = e_x(t) - \bar e_x$, and $\hat X, X: \R \rightarrow \R^{n_x}$ are the deviations of the state variables from their respective steady states. Next, we rewrite the system by adding and subtracting terms such that $\hat X$ does not appear on the right-hand side:
	\begin{align}
		\dot E_x(t)
		&= \hat F E_x(t) + e_F X(t) + \chmthree{\hat G} \int_{-\infty}^t \hat \alpha(t - s) \hat H E_x(s) \incr s
	 + \hat G \int_{-\infty}^t \hat \alpha(t - s) e_H X(s) \incr s \\
		&+ \hat G \int_{-\infty}^t e_\alpha(t - s) H X(s) \incr s
	 + e_G \int_{-\infty}^t \alpha(t - s) H X(s) \incr s. \nonumber
	\end{align}
	}%
	As the system depends on the deviation, $X$, of the state variables from the steady state, $\bar x$, the stability cannot be investigated independently. Therefore, we augment it with the linearized system~\eqref{eq:linearized:system} to obtain
	\begin{align}\label{eq:augmented:linearized:system}
		\begin{bmatrix}
			\dot X(t) \\
			\chmtwo{\dot E}_x(t)
		\end{bmatrix}
		=
		\begin{bmatrix}
			F \\
			\chmtwo{e_F} & \chmtwo{\hat F}
		\end{bmatrix}
		\begin{bmatrix}
			X(t) \\
			\chmtwo{E}_x(t)
		\end{bmatrix}
		+
		\begin{bmatrix}
			G \\
			\chmtwo{e_G} & \chmtwo{\hat G}
		\end{bmatrix}
		\int_{-\infty}^t
		\begin{bmatrix}
			\alpha(t - s) \\
			e_\alpha(t - s) & \hat \alpha(t - s)
		\end{bmatrix}
		\begin{bmatrix}
			H \\
			\chmtwo{e_H} & \chmtwo{\hat H}
		\end{bmatrix}
		\begin{bmatrix}
			X(s) \\
			\chmtwo{E}_x(s)
		\end{bmatrix}
		\incr s.
	\end{align}
	The corresponding characteristic function is
	\begin{align}
		\det\left(
		\begin{bmatrix}
			F \\
			\chmtwo{e_F} & \chmtwo{\hat F}
		\end{bmatrix}
		-
		\begin{bmatrix}
			\lambda I \\
			& \lambda I
		\end{bmatrix}
		+
		\begin{bmatrix}
			G \\
			\chmtwo{e_G} & \chmtwo{\hat G}
		\end{bmatrix}
		\int_0^\infty e^{-\lambda s}
		\begin{bmatrix}
			\alpha(s) \\
			e_\alpha(s) & \hat \alpha(s)
		\end{bmatrix}
		\incr s
		\begin{bmatrix}
			H \\
			\chmtwo{e_H} & \chmtwo{\hat H}
		\end{bmatrix}
		\right) \\
		= \det(F - \lambda I + G \int_0^\infty e^{-\lambda s} \alpha(s) \incr s H) \det(\chmtwo{\hat F} - \lambda I + \chmtwo{\hat G} \int_0^\infty e^{-\lambda s} \hat \alpha(s) \incr s \chmtwo{\hat H}) \nonumber \\
		= P(\lambda) \hat P(\lambda), \nonumber
	\end{align}
	where $P$ and $\hat P$ are the characteristic functions corresponding to the original system of DDEs~\eqref{eq:system:x}--\eqref{eq:system:delay} and the approximate system \chgtwo{of ODEs~\eqref{eq:lct:approximate:system:ODE}}, respectively.
	Consequently, the roots of $P$ and $\hat P$ are also roots of the above characteristic function, and if their real parts are negative, the augmented system is asymptotically stable and the error dynamics are locally asymptotically stable.
	\chgthree{Furthermore, according to Theorem~\ref{thm:ode:approximation}\ref{thm:ode:approximation:roots}, the roots of $\hat P$ converge to the isolated roots of $P$ as $a_{ij} \rightarrow \infty$ when $M_{ij}/a_{ij} \rightarrow \infty$ as well for $i = 1, \ldots, n_z$ and $j = 1, \ldots, n_r$, and all roots with real part larger than $-\sigma$ are isolated, where $\sigma \in \Rp$ is the rate parameter in the exponential bound on the initial state, $x_0$ (see Corollary~\ref{thm:stability}). Additionally, by assumption, the roots have negative real parts because the DDEs~\eqref{eq:system:x}--\eqref{eq:system:delay} are locally asymptotically stable around their steady state. Consequently, $E_x(t) \rightarrow 0$ as $t \rightarrow \infty$ since $\bar e_x \rightarrow 0$ according to Theorem~\ref{thm:ode:approximation}\ref{thm:ode:approximation:steady:state}.}}
\end{proof}

	\section{Algorithm}\label{sec:algo}
\chgtwo{In this section,} we describe an algorithm for determining the coefficients and the rate parameter in an $M$'th order Erlang mixture approximation of a regular kernel, $\alpha: \Rnn \rightarrow \chmtwo{\R}$, for given $M \in \Nnn$. For simplicity, we only consider scalar kernels in this section. The algorithm consists of two steps: First, we identify a domain in which we approximate the kernel and then, we use a least-squares approach to determine the coefficients. Furthermore, we describe a reference approach based on the theoretical expressions \chgtwo{in}~\eqref{eq:erlang:mixture:approximation} for the coefficients. Finally, we describe the implementation of the algorithm and the simulation of the original DDEs~\eqref{eq:system:x}--\eqref{eq:system:delay} and the \chgthree{Erlang ODE approximation}~\eqref{eq:lct:approximate:system:ODE}.

\subsection{Domain}\label{sec:algo:domain}
In order to determine the approximation domain, we choose $t_h \in \Rnn$ as the solution to the equation $\chgtwo{\rho(\infty)} - \chgtwo{\rho}(t) = \epsilon$ \chgtwo{for a given threshold, $\epsilon \in \Rp$,} and approximate $\alpha$ in the interval $[0, t_h]$.
We use bisection to determine $t_h$ approximately.
We assume that an initial interval, $[t_0^\ell, t_0^u]$, is available and that it contains the root, i.e., that $\chgtwo{\rho(\infty)} - \chgtwo{\rho}(t_0^\ell) > \epsilon$ and $\chgtwo{\rho(\infty)} - \chgtwo{\rho}(t_0^u) < \epsilon$. If that is not the case, the interval can be shifted or scaled such that it does. Next, the $k$'th iteration has converged and an approximate solution has been found if $|\chgtwo{\rho(\infty)} - \chgtwo{\rho}(t_k^m) - \epsilon|$ is below a specified threshold, where $t_k^m = \frac{1}{2}(t_k^\ell + t_k^u)$ is the midpoint. Otherwise, the next iteration considers the lower half of the interval, $[t_{k+1}^\ell, t_{k+1}^u] = [t_k^\ell, t_k^m]$, if $\chgtwo{\rho(\infty)} - \chgtwo{\rho}(t_k^m) < \epsilon$. \chgtwo{Conversely}, if $\chgtwo{\rho(\infty)} - \chgtwo{\rho}(t_k^m) > \epsilon$, the next iteration considers the upper half of the interval, $[t_{k+1}^\ell, t_{k+1}^u] = [t_k^m, t_k^u]$.
\begin{remark}
	If the function $\chgtwo{\rho}$ \chgtwo{in}~\eqref{eq:integrals} cannot be derived analytically for a given kernel, $\alpha$, it can be approximated using numerical quadrature (e.g., with \matlab{}'s function \integral{}).
\end{remark}
\begin{remark}
	The equation $\chgtwo{\rho(\infty)} - \chgtwo{\rho}(t) = \epsilon$ can also be solved using a gradient-based approach, e.g., Newton's method. However, $\chgtwo{\rho(\infty)} - \chgtwo{\rho}(t)$ rounds to zero in finite-precision arithmetic when $t$ is too large, e.g., in a computer implementation. Consequently, Newton iterations may stall if an iterate happens to become too large.
\end{remark}

\subsection{Identify kernel parameters}\label{sec:algo:least:squares}
\chgtwo{
For a given order, $M \in \Nnn$, and $t_h \in \Rnn$, we choose the rate parameter $a \in \Rp$ as
\begin{align}\label{eq:algo:rate:parameter}
	a &= \frac{M+1}{t_h}.
\end{align}
Next, the objective is to determine the values of the coefficients, $c_m \in \R$ for $m = 0, \ldots, M$, that minimize the integral of the squared error over the interval $[0, t_h]$,
\begin{align}
	\int_0^{t_h} (\alpha(t) - \hat \alpha(t))^2 \incr t,
\end{align}
while satisfying the constraint that the coefficients must sum to the integral of $\alpha$ over the same interval. We approximate this integral by a left rectangle rule (such that the objective function includes $\alpha(0)$) with $N$ rectangles, and we determine the coefficients by solving the regularized and constrained quadratic program (QP)
\begin{subequations}\label{eq:algo:nlp}
	\begin{align}
		\label{eq:algo:nlp:obj}
		\min_{\{c_{m}\}_{m=0}^M} \hspace{16pt} &\phi = \frac{1}{2} \sum_{k=0}^{N-1} (\alpha(t_k) - \hat \alpha(t_k))^2 \Delta t + \frac{w}{2} \sum_{m=0}^M c_m^2, \\
		\label{eq:algo:nlp:sum}
		\text{subject to} \quad &
		\sum_{m=0}^M c_m = \beta(t_h),
	\end{align}
\end{subequations}
where the quadrature points are $t_k = k \Delta t$ for $k = 0, \ldots, N-1$\chg{, $\Delta t \in \Rp$ is the width of the rectangles, and $w \in \Rp$ is a regularization weight}. Furthermore, the objective function involves the regular kernel, $\alpha: \Rnn \rightarrow \R$, and the $M$'th order Erlang mixture kernel, $\hat \alpha: \Rnn \rightarrow \R$, with coefficients $\{c_m\}_{m=0}^M$ and rate parameter $a$, evaluated at the quadrature points.
\begin{proposition}
	The unique solution $c \in \R^{M+1}$ to the regularized and constrained least-squares problem~\eqref{eq:algo:nlp} is a column vector given by
	\begin{align}\label{eq:algo:nlp:sol}
		c &= (L^T L + w I)^{-1} (e \lambda + L^T g),
	\end{align}
	where $\lambda \in \R$ is a Lagrange multiplier, $e \in \R^{M+1}$ is a column vector of ones, and  $L \in \Rnn^{N \times M+1}$ and $g \in \R^N$ are a matrix and vector whose elements are equal to the basis functions and the kernel $\alpha$ evaluated in the quadrature points \chgthree{and scaled by $\sqrt{\Delta t}$}, respectively:
	\begin{align}\label{eq:algo:nlp:lagrange:multiplier}
		\lambda &= \frac{\beta(t_h) - e^T (L^T L + w I)^{-1} L^T g}{e^T (L^T L + w I)^{-1} e}, & L_{km} &= \chmthree{\sqrt{\Delta t}}\, \ell_m(t_k), & g_k &= \chmthree{\sqrt{\Delta t}}\, \alpha(t_k),
	\end{align}
	for $m = 0, \ldots, M$ and $k = 0, \ldots, N-1$.
\end{proposition}
\begin{proof}
	The QP~\eqref{eq:algo:nlp} is in the form
	\begin{subequations}
		\begin{align}
			\min_c \hspace{18pt} \phi &= \frac{1}{2} c^T (L^T L + w I) c \chmthree{-} g^T L c, \\
			\mathrm{subject~to} \quad & e^T c = \beta(t_h),
		\end{align}
	\end{subequations}
	where $I \in \R^{M+1 \times M+1}$ is an identity matrix. This is an equality-constrained quadratic program and the first-order necessary conditions for a solution is that there must exist a Lagrange multiplier, $\lambda$, such that the following linear system of equations is satisfied~\cite[Sec.~16.1]{Nocedal:Wright:2006}:
	\begin{align}
		\begin{bmatrix}
			L^T L + w I & -e \\
			-e^T
		\end{bmatrix}
		\begin{bmatrix}
			c \\ \lambda
		\end{bmatrix}
		&=
		\begin{bmatrix}
			L^T g \\ -\beta(t_h)
		\end{bmatrix}.
	\end{align}
	The expression for $c$ in~\eqref{eq:algo:nlp:sol} is the solution to the first row block. Next, we substitute it into the second row block:
	\begin{align}
		-e^T c &= -e^T (L^T L + w I)^{-1} (e \lambda + L^T g) = -\beta(t_h).
	\end{align}
	The expression for $\lambda$ in~\eqref{eq:algo:nlp:lagrange:multiplier} is obtained by isolating $\lambda$ in the rightmost equation. Finally, the solution is locally unique because the Hessian $L^T L + w I$ is positive definite when $w > 0$, i.e., $x^T (L^T L + w I) x = \|L x\|_2^2 + w \|x\|_2^2 > 0$ for all $x \neq 0$~\cite[Thm.~12.6]{Nocedal:Wright:2006}.
\end{proof}}

\subsection{Theoretical Erlang mixture approximation}\label{sec:algo:theoretical:approximation}
We will compare the accuracy obtained with the least-squares approach described in Section~\ref{sec:algo:least:squares} to that of a reference approach based on the theoretical values of the coefficients from \chgthree{Definition}~\ref{def:erlang:mixture:approximation}. First, we determine the size of the approximation interval, $t_h$, such that $\chmtwo{\rho(\infty)} - \chmtwo{\rho}(t_h) \approx \epsilon$, e.g., using bisection as described in Section~\ref{sec:algo:domain}. Next, for a given order, $M$, we determine the rate parameter by \chgtwo{\eqref{eq:algo:rate:parameter}}, and we use the theoretical expression \chgtwo{in}~\eqref{eq:erlang:mixture:approximation} to compute the coefficients, $c_m$, for $m = 0, \ldots, M$.

\subsection{Implementation}\label{sec:algo:implementation}
We exploit that \chg{the recursion} $\chg{\ell}_m(t) = (at/m) \chg{\ell}_{m-1}(t)$ for $m = 1, \ldots, M$ where $\chg{\ell}_0(t) = a e^{-at}$ is less prone to rounding errors than a straightforward evaluation of the expression~\eqref{eq:erlang:pdf} in the definition.
We simulate the \chgthree{Erlang ODE approximation}~\eqref{eq:lct:approximate:system:ODE} using either \matlab{}'s \odeff{} or \odeofs, and depending on the order\chgthree{s} of the Erlang mixture \chgthree{approximations}, we implement the matrices $A$, $B$, and $C$ as either dense or sparse. Furthermore, we provide the analytical Jacobian of the right-hand side function of the \chgthree{Erlang ODE approximation}~\eqref{eq:lct:approximate:system:ODE} to the simulator.
We implement the numerical method for non-stiff DDEs described in Appendix~\ref{sec:numerical:simulation:non:stiff} in both \matlab{} and C (with a \texttt{MEX} interface). The numerical method from Appendix~\ref{sec:numerical:simulation:stiff} is implemented in \matlab{}, we use \fsolve{} to solve the residual equations, and we provide the analytical Jacobian.  In all three implementations, the right-hand side function, $f$, the delay function, $h$, and the kernel, $\alpha$, are evaluated in \matlab{}. Finally, in Section~\ref{sec:ex}, we carry out parallel simulations and kernel approximations using a high-performance computing cluster~\cite{DTU:DCC:Resource}.
	\section{Numerical test of the accuracy}\label{sec:ex:conv}%
In this section, we construct a DDE in the form~\eqref{eq:system:x}--\eqref{eq:system:delay} with a known solution and use it to investigate the accuracy of the Erlang mixture approximation and the \chgthree{Erlang ODE approximation}. Furthermore, we assess the convergence rate of the two numerical methods for solving \chgtwo{IVP}s in the form~\eqref{eq:system}--\eqref{eq:system:delay} described in Appendix~\ref{sec:numerical:simulation}.
\chg{Finally, we compare the accuracy of the Erlang mixture approximation with that of one of the exponential mixture approximations described in~\cite{Guglielmi:Hairer:2025} for a system of DDEs involving a gamma kernel, which describes chemotherapy-induced myelosuppression~\cite{Krzyzanski:2019}.}

\subsection{\chg{Modified logistic equation}}\label{sec:ex:conv:logistic:equation}
We consider a modified logistic differential equation with a forcing term in the form
\begin{align}\label{eq:logistic:differential:equation}
	\dot x(t) &= \sigma x(t) \left(1 - \frac{z(t)}{\kappa}\right) + Q(t), &
	z(t) &= \int\limits_{-\infty}^{t} \alpha(t - s) x(s)\,\mathrm ds.
\end{align}
Here, $x: \R \rightarrow \Rnn$ is a population density, $z: \R \rightarrow \Rnn$ is the memory state, $\sigma \in \Rp$ is the logistic growth rate, $\kappa \in \Rp$ is a carrying capacity, and $Q: \R \rightarrow \R$ is the forcing term.
We choose the solution, $x^*: \R \rightarrow \Rnn$, and the kernel, $\alpha: \Rnn \rightarrow \Rnn$, as
\begin{align}\label{eq:verification:true:solution}
	x^*(t) &= 1 + e^{-\left(\frac{t}{\gamma}\right)^2}, &
	\alpha(t) &= \frac{2}{\sqrt{\pi}} e^{-t^2}, &
	\beta (t) &= \erf(t),
\end{align}
where $\gamma \in \Rp$ is a dilation parameter, $\beta: \Rnn \rightarrow [0, 1]$ is derived from~\eqref{eq:integrals}, and $\erf: \R \rightarrow [0, 1]$ is the error function.
Consequently, the true memory state, $z^*: \R \rightarrow \Rnn$, is
\begin{align}
	z^*(t) &= 1 + \frac{\gamma}{\sqrt{\gamma^2 + 1}} e^{-\frac{t^2}{\gamma^2 + 1}} \left(1 + \erf\left(\frac{t}{\gamma \sqrt{\gamma^2 + 1}}\right)\right),
\end{align}
and the forcing term is chosen such that $x^*$ in~\eqref{eq:verification:true:solution} is the solution:
\begin{align}
	Q(t)
	&= \dot x^*(t) - \sigma x^*(t) \left(1 - \frac{z^*(t)}{\kappa}\right).
\end{align}
This is referred to as the \emph{method of manufactured solutions}~\cite{Roache:2002}.
\begin{remark}
	We choose $x^*$ such that it is bounded away from zero because the system becomes unstable if the numerical approximation of the solution becomes negative. This can happen when the true solution is close to zero, in which case the numerical approximation may diverge.
\end{remark}

Figure~\ref{fig:verification} shows the approximation errors for the state and the kernel:
\begin{align}\label{eq:verification:error}
	E_x &= \chmtwo{\left(}\sum_{n=0}^{\chg{N}_x-1} (\hat x(t_{n+1}) - x^*(t_{n+1}))^2 \Delta t_{\chmtwo{x}}\chmtwo{\right)^{\frac{1}{2}}}, &
	E_\alpha &= \chmtwo{\left(}\sum_{k=0}^{\chg{N}_\alpha-1} (\hat \alpha(t_k) - \alpha(t_k))^2 \Delta t_{\chmtwo{\alpha}}\chmtwo{\right)^{\frac{1}{2}}}.
\end{align}
Here, $\hat x$ and $\hat \alpha$ denote the approximations\chgtwo{, and $\Delta t_x = (t_f - t_0)/N_x$ and $\Delta t_\alpha = t_h/N_\alpha$}. Note that the kernel error, $E_\alpha$, is analogous to the \chgtwo{square root of the first term in the} objective function in~\eqref{eq:algo:nlp:obj} but may contain a different number of terms. We have used the model parameter values $\sigma = 4$~mo$^{-1}$ and $\kappa = 1$ and the dilation parameter $\gamma = 10$~mo, and the initial and final times are $t_0 = 0$~mo and $t_f = 24$~mo. When we determine the Erlang mixture approximation of a given order, $M$, we use $N = 10^{\chg{3}}$ points \chgthree{and a weight of $w = 5\cdot 10^{-11}$} in the objective function in~\eqref{eq:algo:nlp:obj}. Furthermore, we use a threshold of $\epsilon = 10^{-14}$ and a convergence tolerance of $10^{-15}$ to determine the approximation interval length, $t_h$ \chgtwo{(see Table~\ref{tab:approximation:parameters})}. We also use the identified value of $t_h$ in the reference approach described in Section~\ref{sec:algo:theoretical:approximation}.
We use \matlab{}'s \odeff{} with absolute and relative tolerances of $10^{-12}$ to simulate the approximate system of ODEs~\eqref{eq:lct:approximate:system:ODE}, and we use $\chg{N}_x = 24\cdot 10^3$ quadrature points to evaluate the state error in~\eqref{eq:verification:error}, which corresponds to a time step size of $\Delta t_x = 10^{-3}$~mo. Furthermore, we use $\chg{N}_\alpha = 10^6$ quadrature points to evaluate the kernel error.
For the results obtained with the numerical methods described in Appendix~\ref{sec:numerical:simulation}, we use a memory horizon of $\Delta t_h = t_f - t_0 = 24$~mo and the above time step size, and for the method for stiff DDEs, we use a function tolerance of $10^{-12}$ and an optimality tolerance of $10^{-8}$ when we solve the involved residual equations with \matlab{}'s \fsolve{}.

The results demonstrate that for this example, the Erlang mixture approximation obtained with the proposed least-squares approach is several orders of magnitude more accurate than the reference approach based on the theoretical expressions for the coefficients. In both cases, there is a high correlation between the error of the kernel and the error of the approximate state trajectory. \chg{For the least-squares approach, the convergence rates are faster than polynomial, and for the reference approach, the convergence rates are \chgtwo{linear} for sufficiently large values of $M$ (see also Figure~\ref{fig:verification:extended}).} The kernel error is higher for $M$ \chg{equal to 128} than \chg{for \chgtwo{$64$}} because the objective function in the \chgtwo{Q}P~\eqref{eq:algo:nlp} uses a coarser resolution in the rectangle rule than the error in~\eqref{eq:verification:error}. Finally, the state error decreases \chgtwo{linearly} with the time step size for the two numerical methods described in Appendix~\ref{sec:numerical:simulation} because the involved integral \chg{and differential equation} have been discretized with first-order methods.
\begin{figure*}
	\subfloat{
		\includegraphics[width=0.48\textwidth]{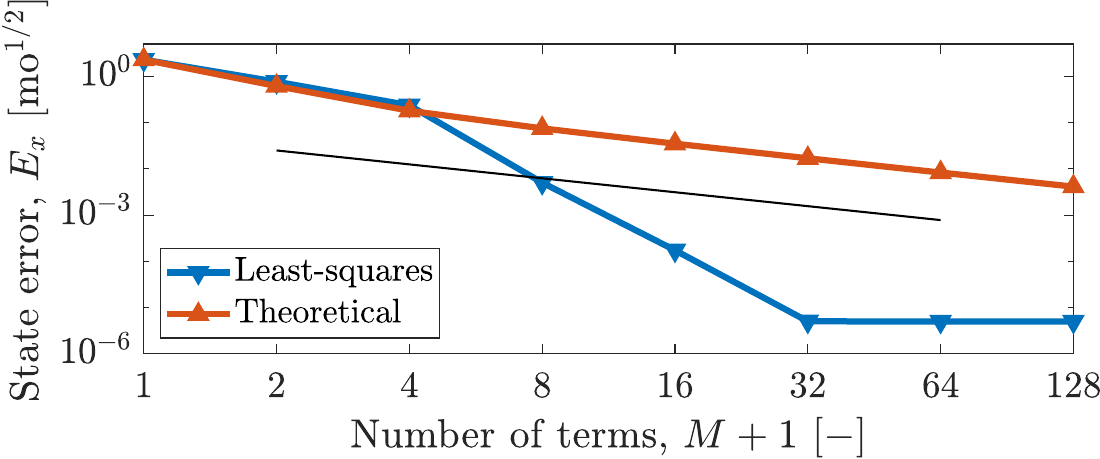}~
		\includegraphics[width=0.48\textwidth]{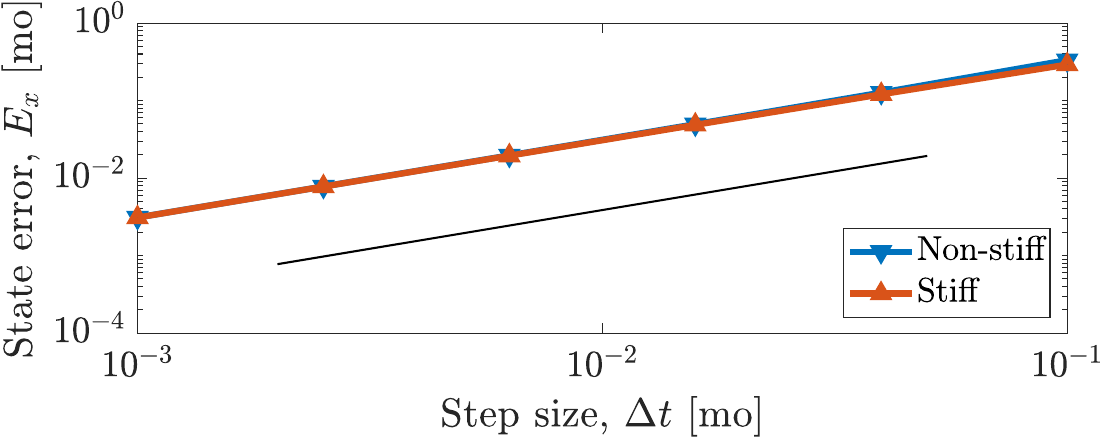}
	} \\
	\subfloat{
		\includegraphics[width=0.48\textwidth]{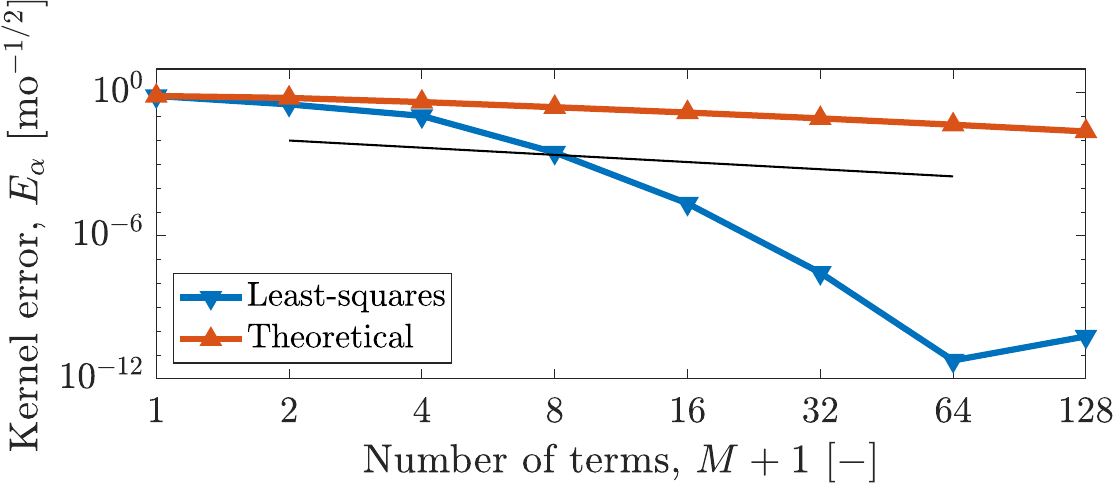}~
		\includegraphics[width=0.48\textwidth]{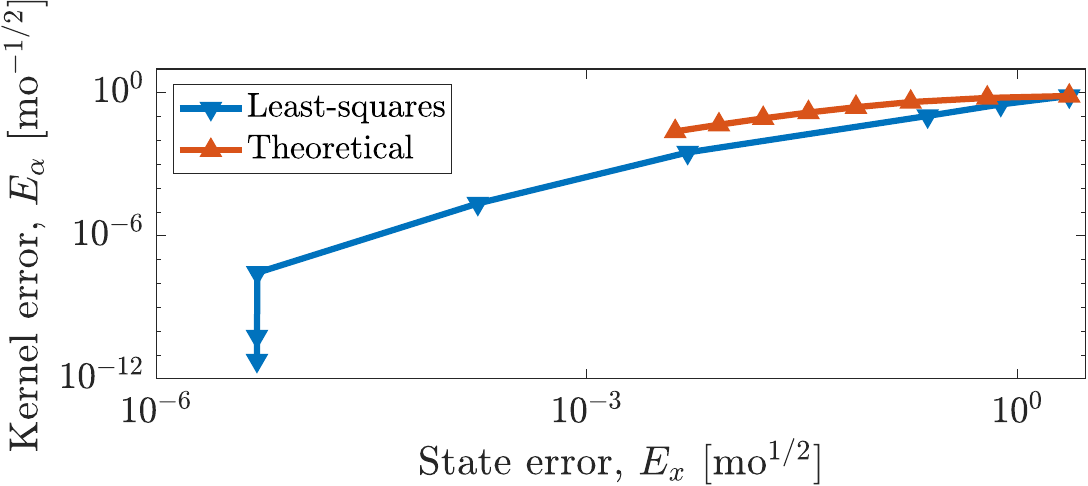}
	}
	\caption{\chg{Convergence analysis for the modified logistic equation.} The state (top left) and kernel (bottom left) errors for Erlang mixture approximations of different orders, $M$, obtained with the proposed least-squares approach and the reference approach based on the theoretical expressions for the coefficients. The bottom right \chg{plot} shows the state and kernel errors against each other, and the top right plot shows the state error obtained with the numerical approaches for non-stiff and stiff DDEs for different time step sizes, $\Delta t$. The black solid line\chg{s in the left column are proportional to $1/(M+1)$, and the one in the top right} is proportional to $\Delta t$.}
	\label{fig:verification}
\end{figure*}
\chgtwo{
\begin{table}[tb]
	\centering
	\caption{The Erlang mixture approximation order, $M$, the approximation horizon, $t_h$, identified with the bisection method from Section~\ref{sec:algo:domain}, and the value of the rate parameter, $a$, obtained with~\eqref{eq:algo:rate:parameter} for the problems considered in Section~\ref{sec:ex:conv} (the first two) and Section~\ref{sec:ex}. For the first two problems, $t_h$ is fixed, but $M$ varies. Therefore, $a$ varies as well, and this is marked with a v. The third and fourth problems are the bifurcation analyses with respect to the logistic growth rate, $\sigma$, and the kernel parameter, $\mu_2$, respectively, and in the latter case, the shown values correspond to the six values of $\mu_2$ considered in Figure~\ref{fig:logistic:equation:bifurcation:analysis} and~\ref{fig:logistic:equation:identified:kernels}. The fifth problem involves the six kernels shown in Figure~\ref{fig:nuclear:fission:identified:kernels}.}
	\label{tab:approximation:parameters}
	\footnotesize
	\begin{tabular}{lrr@{~}lr@{~}l}
		\toprule
		Problem & \multicolumn{1}{c}{$M$} & \multicolumn{2}{c}{$t_h$} & \multicolumn{2}{c}{$a$} \\
		\midrule
		Modified logistic equation (forced) & \multicolumn{1}{c}{v} & 4.32 & mo & \multicolumn{2}{c}{v} \\
		Myelosuppression & \multicolumn{1}{c}{v} & 489.0\chgthree{7} & h & \multicolumn{2}{c}{v} \\
		\midrule
		Modified logistic equation ($\sigma$) & 100 & 1.36 & mo & 74.41 & mo$^{-1}$ \\
		\midrule
		\multirow{6}{*}{Modified logistic equation ($\mu_2$)} & \multirow{6}{*}{350} & 0.92 & mo & 380.75 & mo$^{-1}$ \\
		&& 1.22 & mo & 287.08 & mo$^{-1}$ \\
		&& 1.45 & mo & 242.85 & mo$^{-1}$ \\
		&& 1.62 & mo & 217.31 & mo$^{-1}$ \\
		&& 1.91 & mo & 184.13 & mo$^{-1}$ \\
		&& 2.32 & mo & 151.27 & mo$^{-1}$ \\
		\midrule
		\multirow{6}{*}{Nuclear fission} & \multirow{6}{*}{500} & 12.13 & s & 41.32 & s$^{-1}$ \\
		&& 12.10 & s & 41.43 & s$^{-1}$ \\
		&& 11.97 & s & 41.86 & s$^{-1}$ \\
		&& 11.69 & s & 42.87 & s$^{-1}$ \\
		&& 10.53 & s & 47.57 & s$^{-1}$ \\
		&&  8.28 & s & 60.50 & s$^{-1}$ \\
		\bottomrule
	\end{tabular}
\end{table}
}

\subsection{Myelosuppression}\label{sec:ex:conv:myelosuppression}
\chg{
In this section, we compare the accuracy of the Erlang mixture approximation proposed in this work with an exponential mixture approximation (see Section~\ref{sec:gamma:approximation}) used by Guglielmi and Hairer~\cite{Guglielmi:Hairer:2025} for a system of DDEs describing myelosuppression. The system describes the evolution of precursor cells, $y: \R \rightarrow \Rnn$, which ultimately form granulocytes, $w: \R \rightarrow \Rnn$:
\begin{subequations}\label{eq:myelosuppression}
	\begin{align}
		\label{eq:myelosuppression:y}
		\dot y(t) &= \left(\kappa \left(\frac{w_0}{w(t)}\right)^\gamma - k_s C(t) - \kappa\right) y(t), \\
		\label{eq:myelosuppression:w}
		\dot w(t) &= -\kappa w(t) + \kappa \int_{-\infty}^t \alpha(t - s) y(s) \incr s, \\
		\label{eq:myelosuppression:A}
		\dot C(t) &= -\frac{V_{\max} C(t)}{K_m + C(t)}.
	\end{align}
\end{subequations}
The purpose of the model is to study the effect of a bolus injection of a drug (fluoracil) and $C: \R \rightarrow \Rnn$ is the concentration of the drug in the plasma, which is described by the Michaelis-Menten model in~\eqref{eq:myelosuppression:A}. The constants are $\kappa = (1 - \lambda)/55.6$~h$^{-1}$, $\lambda = -0.46$, $w_0 = 14.4 \cdot 10^9$~cells/L, $\gamma= 0.507$, $k_s = 0.0213$~L/(mg~h), $V_{\max} = 100$~mg/(L~h), and $K_m = 22$~mg/L, and the gamma kernel is given by
\begin{align}\label{eq:gamma:kernel}
	\alpha(t) &= \frac{\kappa^{1 - \lambda}}{\Gamma(1 - \lambda)} t^{-\lambda} e^{-\kappa t},
\end{align}
where $\Gamma: \C \rightarrow \C$ is the gamma function. The initial and final times are $t_0 = 0$~h and $t_f = 700$~h, and the initial states are $y(t) = w_0$ for $t \in (-\infty, t_0]$, $w(t_0) = w_0$, and $C(t_0) = A_0/V$ where $A_0 = 127$~mg/kg and $V = 1.03$~L/kg.
We use a threshold of $\epsilon = 10^{-\chmthree{5}}$ in the computation of $t_h$ \chgtwo{(see Table~\ref{tab:approximation:parameters})}, and we use $N = 10^4$ points \chgthree{and a weight of $5\cdot 10^{-10}$} in the objective function. As the system is stiff, we use \odeofs{} to simulate the \chgthree{Erlang ODE approximation}, and we compare to the \chgthree{kernel} error of the exponential mixture approximation described in~\cite{Guglielmi:Hairer:2025}, which depends on the tolerance $\epsilon_r \in \Rp$ that also indirectly determines the number of terms in the approximation. Lower tolerances lead to more terms, and we consider $\epsilon_r \in \{10^{-1}, 10^{-3}, 10^{-5}, 10^{-7}, 10^{-9}, 10^{-11}\}$. Furthermore, we compare the state obtained with the Erlang \chgthree{ODE approximation} to that obtained with $\epsilon_r = 10^{-11}$ using $\chg{N}_x = 700\cdot 10^3$ points in time, corresponding to a time step size of $\Delta t_x = 10^{-3}$~h. The remaining parameters are the same as in Section~\ref{sec:ex:conv:logistic:equation}.
}

\chg{
The results are shown in Figure~\ref{fig:gamma:approximation}, and as for the example in Section~\ref{sec:ex:conv:logistic:equation}, the least-squares approach described in Section~\ref{sec:algo} is significantly more accurate than the reference approach based on the theoretical expressions for the coefficients. However, the error is generally higher than for the modified logistic equation (see Figure~\ref{fig:verification}), the convergence rates are no longer faster than polynomial for the least-squares approach, and they are no longer \chgtwo{linear} for the reference approach based on the theoretical expressions for the coefficients. The figure also shows that Guglielmi and Hairer's approach \chgtwo{does have converge rates that are faster than polynomial} (note the difference in the vertical axes). The error of the Erlang mixture approximation is primarily high for $t$ close to zero. The reason for this may be that the derivative of the true kernel has a singularity at zero whereas for finite values of $M$ and $a$, the derivative of the Erlang mixture approximation does not. This may also describe why the convergence rates of the Erlang mixture approximation are reduced. However, we emphasize that the Erlang mixture approximation can be used for a broader class of kernels (bounded and continuous) than the \chgthree{specific} exponential mixture approximations \chgthree{considered by Guglielmi and Hairer}. \chgtwo{Furthermore, the state error is generally high because of the order of magnitude of the state variables ($10$--$30$ for $y$ and $w$) and the length of the simulation interval. A constant error of 1 would correspond to a state error of $E_x = \sqrt{700} \approx 26.5$. Even for $M+1 = 4$, the state error is below this value.}
}

\begin{figure*}
	\subfloat{
		\includegraphics[width=0.48\textwidth]{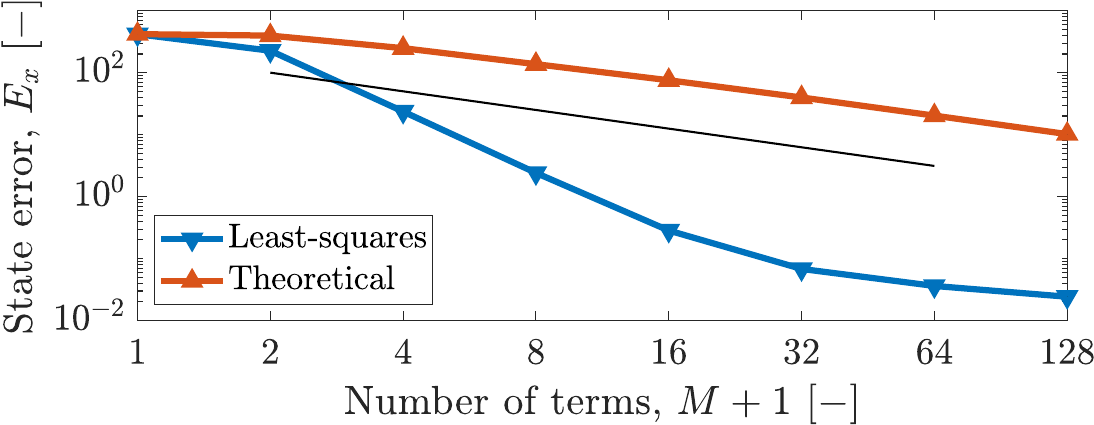}~
		\includegraphics[width=0.48\textwidth]{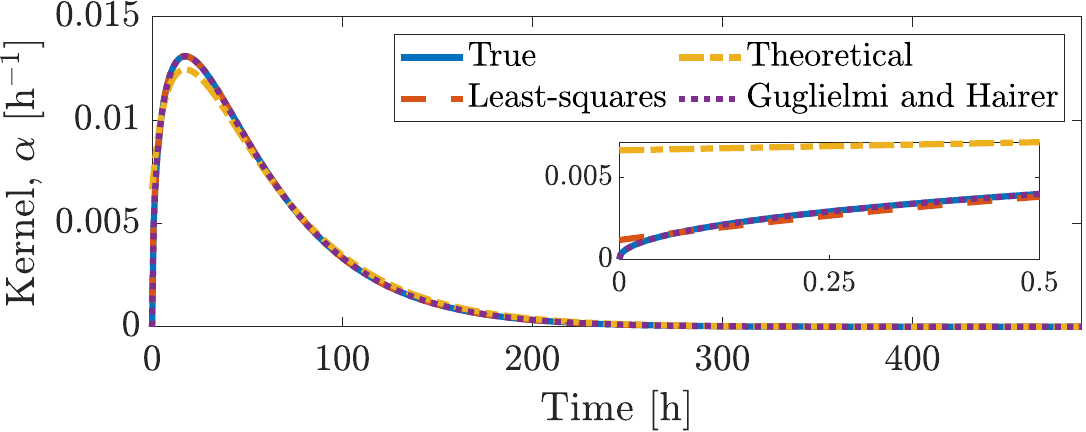}
	} \\
	\subfloat{
		\includegraphics[width=0.48\textwidth]{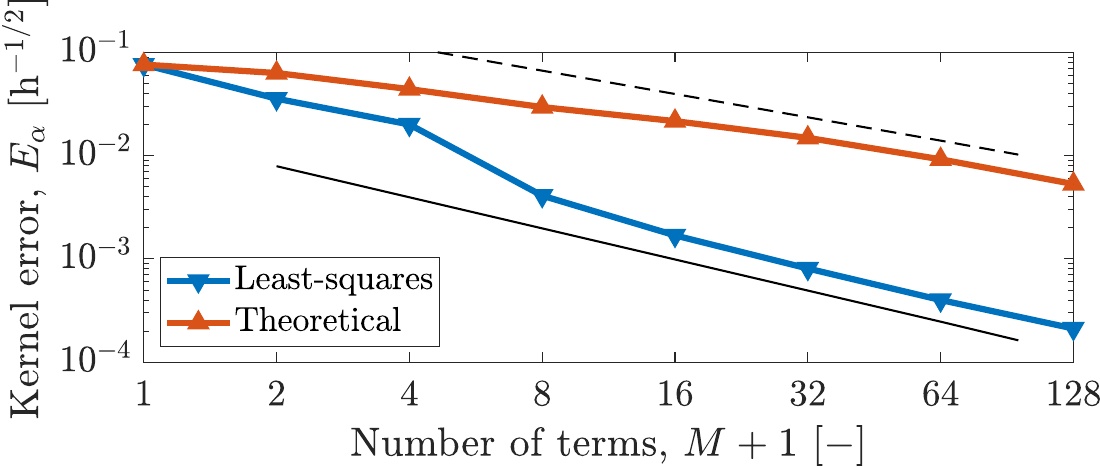}~
		\includegraphics[width=0.48\textwidth]{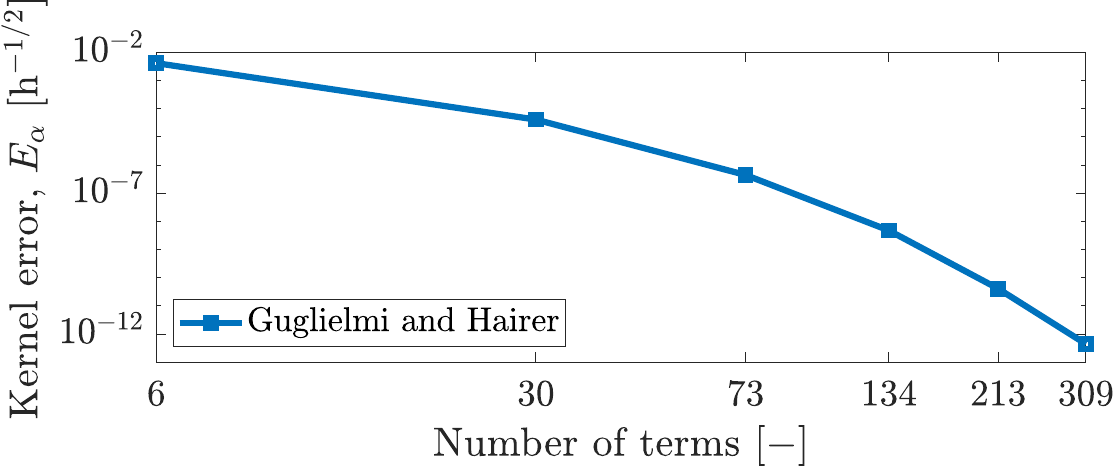}
	}
	\caption{\chg{Convergence analysis for the myelosuppression model. The \chgtwo{maximum} state \chgtwo{error} (top left) and \chgtwo{the} kernel \chgtwo{error} (bottom left) for Erlang mixture approximations of different orders, $M$, obtained with the least-squares approach and the reference approach based on theoretical values of the coefficients. The kernel error is shown for Guglielmi and Hairer's approach in the bottom right, and the true and approximate kernels are shown in the top right for $M = 128$ and $\epsilon_r = 10^{-11}$ (corresponding to 309 terms). The black solid line\chgthree{s} in the left \chgthree{column are} proportional to $1/(M+1)$ whereas the black dashed line in the bottom left is proportional to $1/(M+1)^{3/\chmtwo{4}}$.}}
	\label{fig:gamma:approximation}
\end{figure*}
	\section{Numerical examples}\label{sec:ex}%
In this section, we present two numerical examples that demonstrate how Erlang \chgthree{ODE} approximations can be used to approximately analyze and simulate DDEs with distributed time delays. In the first example, we use the \chgthree{Erlang ODE approximation} to perform a numerical bifurcation analysis of the modified logistic differential equation that was also considered in Section~\ref{sec:ex:conv}, and in the second example, we present a Monte Carlo simulation of a molten salt nuclear fission process, which is nonlinear, multivariate, and stiff. We use high orders, $M$, and strict tolerances in the numerical methods to mitigate any adverse effects on the results. However, in many practical contexts, lower orders and less restrictive tolerances can be used with limited effect on the accuracy (see also Section~\ref{sec:ex:conv}).

\subsection{Modified logistic equation}\label{sec:ex:logistic}
We consider the modified logistic differential equation~\eqref{eq:logistic:differential:equation} where the forcing term is $Q(t) = 0$ and the kernel is
\begin{align}\label{eq:logistic:equation:kernel}
	\alpha(t) &= \gamma_1 F(t; \mu_1, \sigma_1) + \gamma_2 F(t; \mu_2, \sigma_2).
\end{align}
Here, $\gamma_1, \gamma_2 \in [0, 1]$ are weights, $\mu_1, \mu_2 \in \R$ are location parameters, $\sigma_1, \sigma_2 \in \Rp$ are scale parameters, and $F: \Rnn \times \R \times \Rp \rightarrow \Rnn$ is given by the probability density function of a folded normal distribution (i.e., it is a folded normal kernel):
\begin{align}\label{eq:folded:normal:pdf}
	F(t; \mu_i, \sigma_i) &= \frac{\exp\left(-\frac{1}{2}\left(\frac{t - \mu_i}{\sigma_i}\right)^2\right) + \exp\left(-\frac{1}{2}\left(\frac{t + \mu_i}{\sigma_i}\right)^2\right)}{\sqrt{2 \pi}\, \sigma_i}, & i &= 1, 2.
\end{align}
Furthermore, we derive $\beta$ in \chgtwo{Definition~\ref{def:integrals}} analytically.

We use the parameter values in Table~\ref{tab:logistic:parameters}, and the initial state is $x(t) = 0.9$ for $t \leq t_0$. We use the same parameters as in Section~\ref{sec:ex:conv:logistic:equation} to identify Erlang mixture kernels of order $M = 100$ in the bifurcation analysis with respect to the logistic growth rate, $\sigma$, and $M = 350$ in the analysis with respect to the kernel parameter $\mu_2$. \chgtwo{Table~\ref{tab:approximation:parameters} shows the resulting values of $t_h$ and $a$.} In Figure~\ref{fig:logistic:equation:bifurcation:analysis}, the first row shows the eigenvalues of the Jacobian of the right-hand side function in the \chgthree{Erlang ODE approximation}. For completeness, we consider a large range of values for $\sigma$, but for numerical reasons, we have not been able to investigate larger values of $\mu_2$. The second row shows the largest real part of the eigenvalues and when the parameters become sufficiently large, the steady state $\bar x = \kappa$ becomes unstable. The two bottom rows show simulations obtained with the numerical method for non-stiff DDEs for the parameter values indicated in the second row (the colors match). We use a memory horizon of $\Delta t_h = t_f - t_0 = 24$~mo and $10,000$ time steps. In both cases, the stability analysis based on the \chgthree{Erlang ODE approximation} accurately predicts when the steady state is stable, marginally stable (where the largest real part of the eigenvalues is equal to zero), and unstable. For large values of $\sigma$, a limit cycle appears whereas for large values of $\mu_2$, the population density oscillates unstably.
Figure~\ref{fig:logistic:equation:identified:kernels} shows the identified kernels, the corresponding error, and the coefficients for different values of $\mu_2$. \chgtwo{In the left column, the coefficients exhibit oscillation, whereas in the right column, t}here is a similarity between the coefficients and the shape of the approximated kernel, e.g., in terms of bimodality. \chgtwo{The oscillations in the coefficients disappear for larger values of $M$, which also leads to larger values of $a$ (results not shown).}
\begin{table}[t]
	\centering
	\caption{Values of the parameters in the modified logistic equation. For $\sigma$ and $\mu_2$, only the nominal values are shown.}
	\label{tab:logistic:parameters}
	\footnotesize
	\begin{tabular}{cccccccccc}
		\toprule
		\multicolumn{10}{c}{Model, kernel, and simulation parameters} \\
		\midrule
		$\sigma$~[mo\chgthree{$^{-1}$}] & $\kappa$~[--] & $\gamma_1$~[--] & $\gamma_2$~[--] & $\mu_1$~[mo] & $\mu_2$~[mo] & $\sigma_1$~[mo] & $\sigma_2$~[mo] & $t_0$ [mo] & $t_f$ [mo] \\
		4 & 1 & 0.5 & 0.5 & 0.35 & 0.45 & 0.06 & 0.12 & 0 & 24 \\
		\bottomrule
	\end{tabular}
\end{table}
\begin{figure}
	\centering
	\subfloat{
		\includegraphics[width=0.48\textwidth]{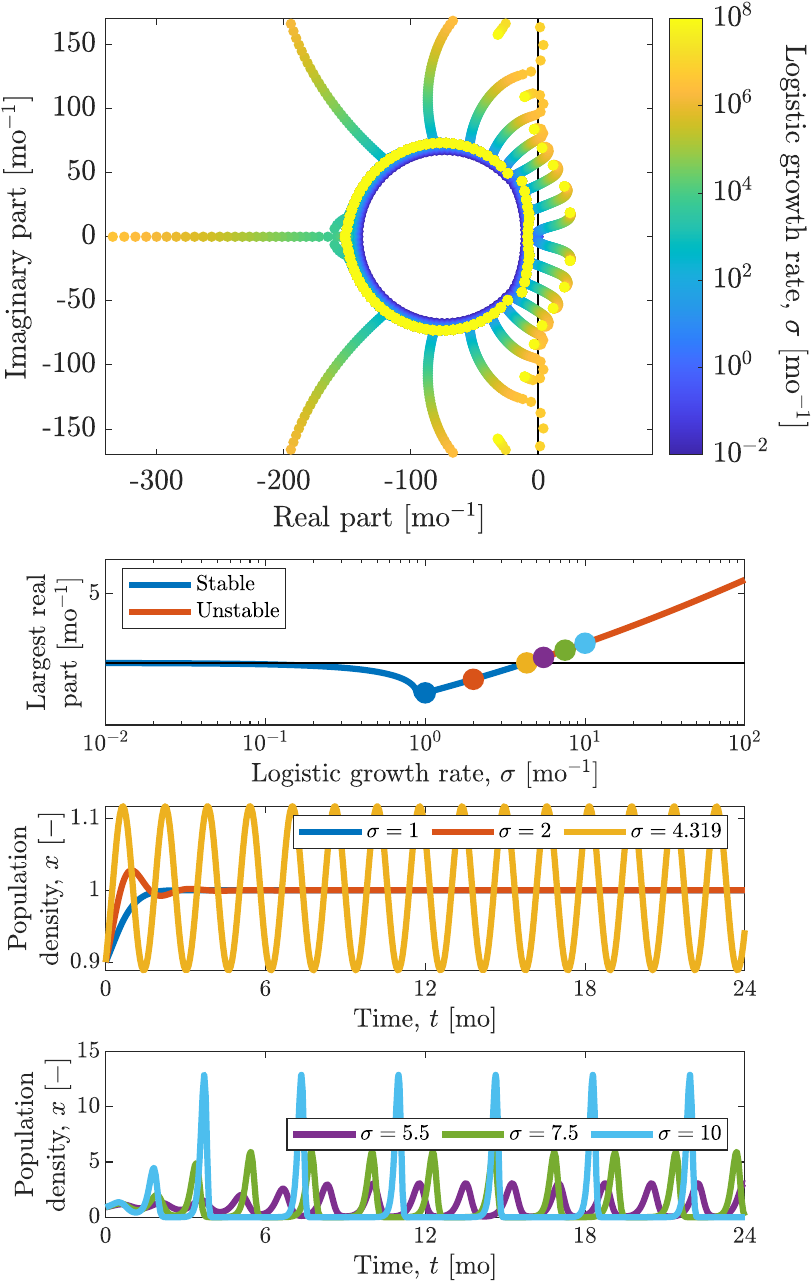}~
		\includegraphics[width=0.48\textwidth]{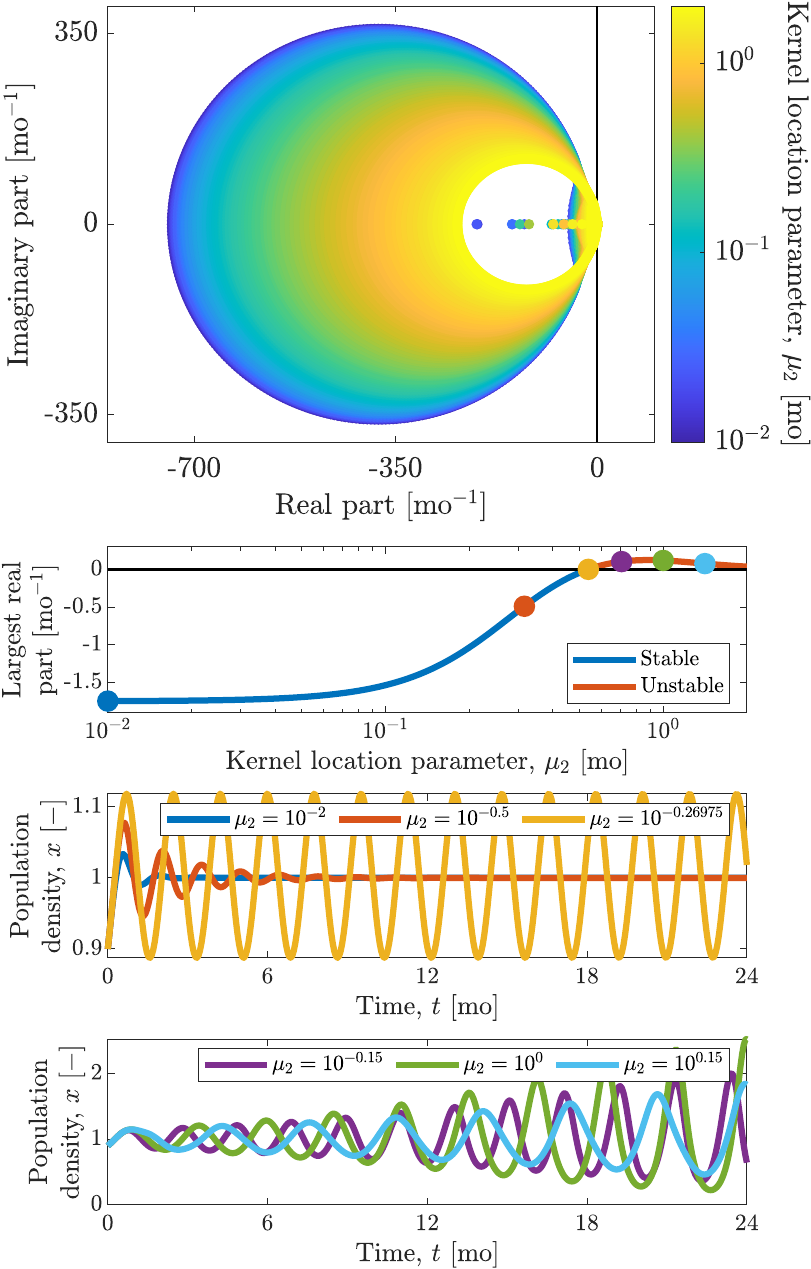}
	}
	\caption{Bifurcation analysis with respect to the model parameter $\sigma$ (left column) and the kernel parameter $\mu_2$ (right column) for the modified logistic equation. First row: Eigenvalues. Second row: The largest real part of the eigenvalues. Third and bottom row: Simulations for selected parameter values (obtained with the numerical method described in Appendix~\ref{sec:numerical:simulation:non:stiff}).}
	\label{fig:logistic:equation:bifurcation:analysis}
\end{figure}
\begin{figure}
	\centering
	\subfloat{
		\includegraphics[width=0.485\textwidth]{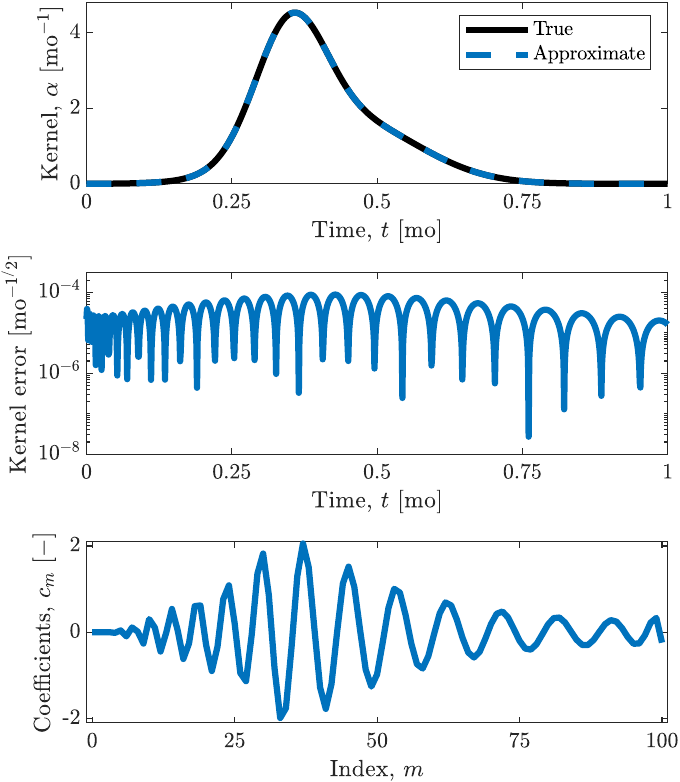}~
		\includegraphics[width=0.48\textwidth]{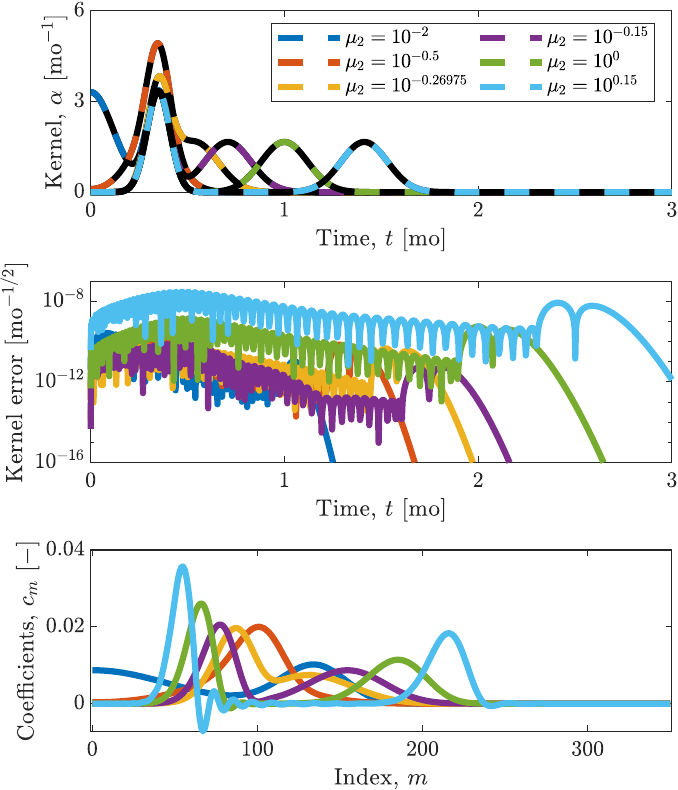}
	}
	\caption{Top row: The true kernels and the corresponding Erlang mixture approximations for $M = 100$ (left) and $M = 350$ (right) obtained in connection with the bifurcation analysis shown in Figure~\ref{fig:logistic:equation:bifurcation:analysis}. Middle row: The absolute errors of the kernel approximations. Bottom row: The coefficients in the Erlang mixture approximations. The colors are consistent across each column\chgtwo{, and the coefficients are visualized as solid lines to ease the interpretation of the results}.}
	\label{fig:logistic:equation:identified:kernels}
\end{figure}

\subsection{Nuclear fission}\label{sec:ex:nuclear:fission}
Next, we consider a point reactor kinetics model of a molten salt nuclear fission reactor~\cite{Duderstadt:Hamilton:1976, Wooten:Powers:2018} where the molten salt is circulated through a heat exchanger outside of the reactor core. The model describes 1)~the concentrations of $N_g = 6 \in \N$ neutron precursor groups, $C_i: \R \rightarrow \Rnn$ for $i = 1, \ldots, N_g$, which emit delayed neutrons, 2)~the concentration of neutrons, $C_n: \R \rightarrow \Rnn$ for $n = N_g + 1 \in \N$, and 3)~the reactivity, $\rho: \R \rightarrow \R$, which represents the relative number of neutrons created per fission event:
\begin{align}
	\dot C_i(t) &= (C_{i, in}(t) - C_i(t)) D + R_i(t), &
	\dot C_n(t) &= R_n(t), &
	\dot \rho(t) &= -\kappa H C_n(t).
\end{align}
The dilution rate, $D \in \Rnn$, is the ratio between the volumetric inlet and outlet flow rate and the reactor core volume, $\kappa \in \Rnn$ is a proportionality constant, and $H \in \Rnn$ is the ratio between the power production proportionality constant and the heat capacity of the reactor core. Furthermore, the production rate $R: \R \rightarrow \R^n$ is defined in terms of the sto\chg{i}chiometric matrix $S \chmthree{: \R \rightarrow} \R^{N_r \times n}$ where $N_r = n \in \N$ is the number of reactions and $r: \R \rightarrow \Rnn^{N_r}$ is a vector of reaction rates:
\begin{align}
	R(t) &= S^T(t) r(t), &
	S(t) &=
	\begin{bmatrix}
		-1      &         &             & 1               \\
		&  \ddots &             & \vdots          \\
		&         & -1          & 1               \\
		\beta_1 & \cdots  & \beta_{N_g} & \rho(t) - \beta
	\end{bmatrix}, &
	r(t) &=
	\begin{bmatrix}
		\lambda_1 C_1(t) \\ \vdots \\ \lambda_{N_g} C_{N_g}(t) \\ C_n(t)/\Lambda
	\end{bmatrix}.
\end{align}
Here, $\lambda_i, \beta_i, \Lambda \in \Rnn$ for $i = 1, \ldots, N_g$ are decay constants, delayed neutron fractions (i.e., the relative number of neutrons that are generated from the decay of neutron precursors), and the mean neutron generation time, respectively, and $\beta \in \Rnn$ is the sum of $\beta_i$ for $i = 1, \ldots, N_g$.
The inlet concentration, $C_{i, in}: \R \rightarrow \Rnn$, is the memory state given by
\begin{align}\label{eq:nuclear:fission:delay}
	C_{i, in}(t) &= \int\limits_{-\infty}^t \alpha_i(t - s) C_i(s)\,\mathrm ds, & i &= 1, \ldots, N_g.
\end{align}
The kernels are described in Appendix~\ref{sec:nuclear:fission:kernel} and represent 1)~a nonuniform velocity profile in the external circulation of the molten salt (and neutron precursors) and 2)~the (exponential) decay of the neutron precursors during the external circulation.

Table~\ref{tab:nuclear:fission:parameters} shows the parameter values, which (except for $D$, $\mu_1$, and $\sigma_1$) are taken from~\cite{Leite:etal:2016}.
\begin{table}[t]
	\centering
	\caption{Values of the parameters in the molten salt reactor model. For $\kappa$, only the nominal value is shown.}
	\label{tab:nuclear:fission:parameters}
	\footnotesize
	\begin{tabular}{cccccc}
		\toprule
		\multicolumn{6}{c}{Decay constants~[1/s]} \\
		\midrule
		$\lambda_1$ & $\lambda_2$ & $\lambda_3$ & $\lambda_4$ & $\lambda_5$ & $\lambda_6$ \\
		0.0124 & 0.0305 & 0.1110 & 0.3010 & 1.1300 & 3.0000 \\
		\midrule
		\multicolumn{6}{c}{Delayed neutron fractions~[--]} \\
		\midrule
		$\beta_1$ & $\beta_2$ & $\beta_3$ & $\beta_4$ & $\beta_5$ & $\beta_6$ \\
		0.00021 & 0.00141 & 0.00127 & 0.00255 & 0.00074 & 0.00027 \\
		\midrule
		\multicolumn{6}{c}{Other model parameters} \\
		\midrule
		$\Lambda$~[s] & $\kappa$~[1/K] & $H$~[K\,cm$^3$/s] & $D$~[1/s] & $\mu_1$~[s] & $\sigma_1$~[s]\\
		$5\cdot 10^{-5}$ & $3\cdot 10^{-4}$ & 0.05 & 2 & 2 & 0.1 \\
		\bottomrule
	\end{tabular}
\end{table}
We use $N = 1000$ points \chgthree{and a weight of $w = 10^{-10}$} to identify Erlang mixture kernel approximations of order $M = \chg{500}$. Furthermore, we use $\epsilon = 10^{-13}$ and a tolerance of $10^{-14}$ in the bisection algorithm \chgtwo{(see Table~\ref{tab:approximation:parameters})}, \chgthree{and} we use Matlab's \integral{} with absolute and relative tolerances of $10^{-15}$ to approximate \chgtwo{$\beta$} in~\eqref{eq:integrals}.
Figure~\ref{fig:nuclear:fission:monte:carlo:simulation:parameters} shows a Monte Carlo simulation for 1000 samples of $\kappa$ drawn from a normal distribution with mean $3\cdot 10^{-4}$ and standard deviation $7.5 \cdot 10^{-5}$. For brevity of the presentation, we omit the plots of the neutron precursor group concentrations. The initial states are $C_i(t) = 1$~kmol/cm$^3$ for $i = 1, \ldots, n$ and $\rho(t) = 1.1 \beta$ for $t \leq t_0$, and we use an absolute and relative tolerance of $10^{-8}$ in \odeofs{} to simulate the \chgthree{Erlang ODE approximation}~\eqref{eq:lct:approximate:system:ODE}. Furthermore, we supply the analytical Jacobian of the right-hand side function, and we represent it as a sparse matrix.
The figure also shows the relative difference between the simulations obtained 1)~as described above and 2)~with the numerical method for stiff DDEs described in Appendix~\ref{sec:numerical:simulation:stiff} for the mean value of $\kappa$. The latter method uses a time step size of $6.25\cdot 10^{-5}$~s, and the relative difference, $E_{r, i}: \R \rightarrow \Rnn$, in the $i$'th state variable and $n$'th time step is given by
\begin{align}\label{eq:relative:error}
	E_{r, i}(t_n) &= \frac{|\hat x_i(t_n) - x_i(t_n)|}{1 + |x_i(t_n)|}, & i &= 1, \ldots, n_x,
\end{align}
where $\hat x_i$ is the approximate solution to the \chgthree{Erlang ODE approximation} and $x_i$ is obtained with the numerical method for stiff DDEs.
For each point in time, the 95\% confidence intervals show the intervals from the 2.5 to the 97.5 percentiles and the span shows the interval of the minimum and maximum value of the states. For the neutron concentration, the confidence interval is almost symmetric (on the logarithmic axis) whereas the maximum is significantly higher above the mean than the minimum is below it. For the reactivity, the confidence interval and the span are hardly visible, and the simulation for the mean value of $\kappa$ is indistinguishable from the pointwise mean.
The maximum relative difference for all eight state variables is $2.14\cdot 10^{-3}$. There are three main sources of error that affect this difference: 1)~The kernel approximation error, 2)~the error from using \odeofs{} to simulate the approximate set of ODEs, and 3)~the error of the numerical method described in Appendix~\ref{sec:numerical:simulation:stiff}. If we use twice as large time steps in the latter, the maximum relative error is $4.28\cdot 10^{-3}$, i.e., twice as large (results not shown). Therefore, we believe that the maximum relative difference would be further reduced by using smaller time step sizes, i.e., that the main source of error is that of the numerical method for stiff DDEs and not the Erlang \chgthree{ODE} approximation. However, due to large memory requirements, it has not been possible to further reduce the time step size.
Finally, Figure~\ref{fig:nuclear:fission:identified:kernels} shows the Erlang mixture approximations of the kernels and the corresponding errors. It also shows the sum of folded normal kernels in~\eqref{eq:nuclear:fission:kernels} for the values of $\mu_1$ and $\sigma_1$ shown in Table~\ref{tab:nuclear:fission:parameters} and the coefficients in the Erlang mixture approximations. \chgtwo{The coefficients exhibit oscillations, which was also the case for the coefficients in the left column of Figure~\ref{fig:logistic:equation:identified:kernels}.}

\begin{figure}
	\centering
	\subfloat{
		\includegraphics[width=0.48\textwidth]{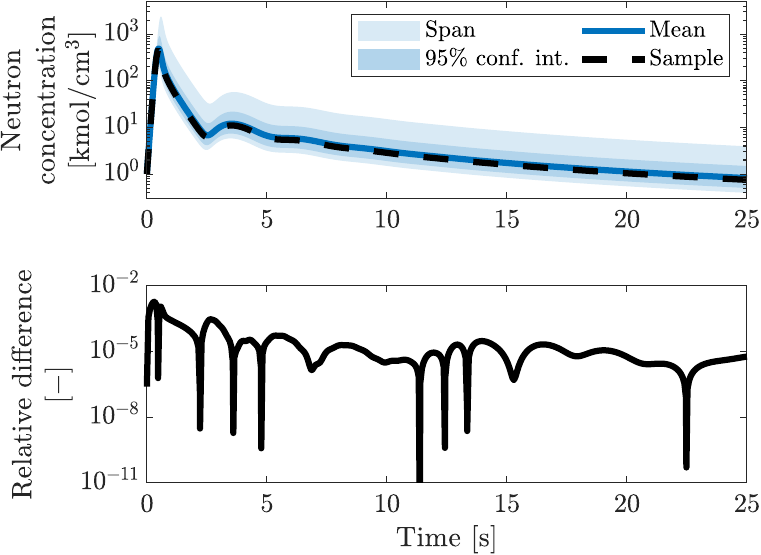}~
		\includegraphics[width=0.48\textwidth]{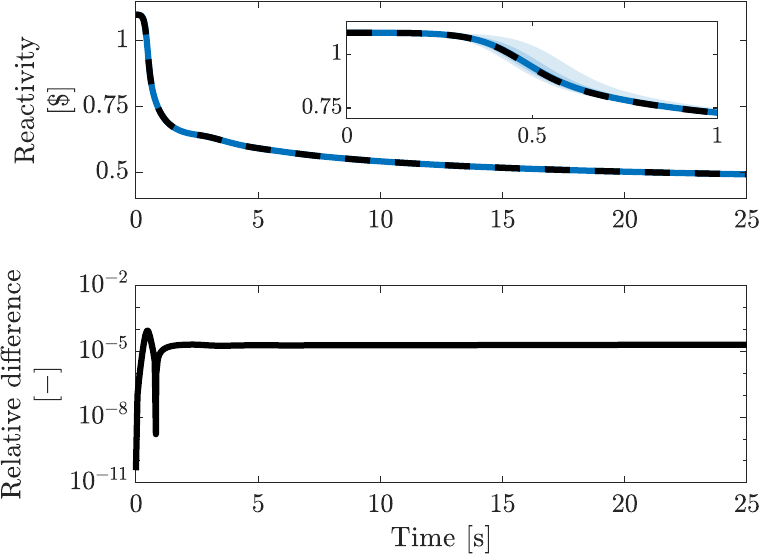}
	}
	\caption{Monte Carlo simulation of the molten salt reactor model. Top row: The neutron concentration (left) and the reactivity (right). For each point in time, the span shows the interval of the minimum and maximum state, and the 95\% confidence interval spans the 2.5 and 97.5 percentiles. The mean is computed pointwise, and the sample is the simulation corresponding to the mean value of $\kappa$. Bottom row: The pointwise relative error~\eqref{eq:relative:error} between simulating the approximate ODEs~\eqref{eq:lct:approximate:system:ODE} using \odeofs{} and simulating the original DDEs using the numerical method for stiff DDEs described in Appendix~\ref{sec:numerical:simulation:stiff}.}
	\label{fig:nuclear:fission:monte:carlo:simulation:parameters}
\end{figure}

\begin{figure}
	\subfloat{
		\includegraphics[width=0.475\textwidth]{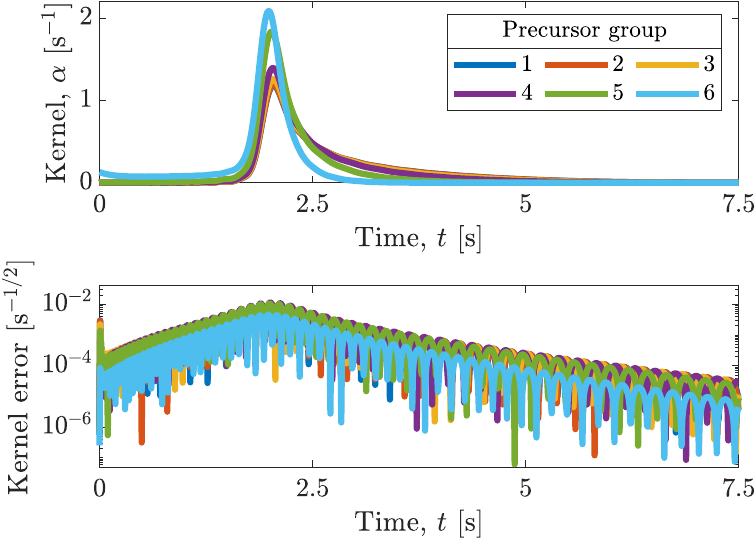}~
		\includegraphics[width=0.485\textwidth]{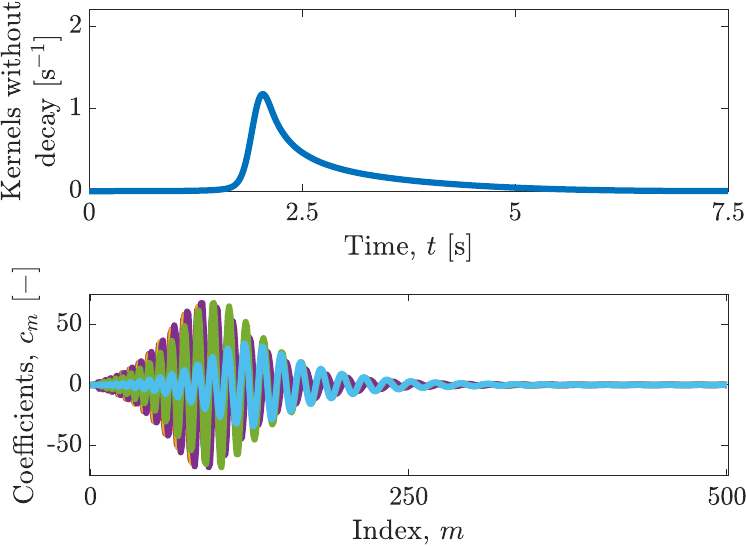}
	}
	\caption{Left column: The Erlang mixture approximations of the kernels in~\eqref{eq:nuclear:fission:kernels} used in the Monte Carlo simulation shown in Figure~\ref{fig:nuclear:fission:monte:carlo:simulation:parameters} (top) and the corresponding approximation errors (bottom). Right column: The normalized sum of folded normal kernels in the expression~\eqref{eq:nuclear:fission:kernels} \chgtwo{without the decay factor} (top) and the coefficients in the Erlang mixture kernel approximations. The colors are consistent across all but the top right figure.}
	\label{fig:nuclear:fission:identified:kernels}
\end{figure}
	\section{Conclusions}\label{sec:conclusions}%
In this paper, we propose an approximate approach for simulating and analyzing DDEs with distributed time delays based on conventional methods for ODEs. Specifically, we approximate the involved kernel by an Erlang mixture kernel and transform the resulting DDEs to ODEs using the LCT. \chgthree{We refer to the approximate ODEs as the Erlang ODE approximation of the DDEs.}
\chg{We prove that if the kernel $\alpha$ is regular (Definition~\ref{def:regular:kernel}) and the coefficients are chosen as in~\eqref{eq:erlang:mixture:approximation}, then an Erlang mixture approximation\chgtwo{,} $\hat \alpha$\chgtwo{, of order $M$} converges pointwise to $\alpha$ and the integral of the absolute error converges to zero as the rate parameter, $a$, \chgtwo{and the order, $M$,} go to infinity \chgtwo{provided that the ratio $M/a$ goes to infinit\chgthree{y} as well (Proposition~\ref{thm:erlang:mixture:approximation})}.
\chgtwo{The main theoretical result\chgthree{s} of the paper \chgthree{are} Theorem\chgthree{s}~\ref{thm:ivp:approximation} \chgthree{and}~\ref{thm:ode:approximation}, which state that} if the kernel $\alpha$ is also exponentially bounded\chgthree{, then} \chgtwo{1)}~the solution to the \chgthree{Erlang ODE approximation} converges \chgtwo{uniformly} to that of the original system \chgtwo{of DDEs, and 2)}~accurate stability analyses of the steady states of the original system of DDEs can be performed based on the \chgthree{Erlang ODE approximation}.
}%
Additionally, we propose an approach for determining optimal Erlang mixture approximations based on bisection and least-squares estimation, and we use numerical example\chg{s} to demonstrate that its accuracy and rate of convergence can be higher than using the theoretical expressions for the coefficients \chg{in~\eqref{eq:erlang:mixture:approximation}. We also compare with the accuracy of an exponential mixture approximation from the literature}. Finally, we present two numerical examples that demonstrate the efficacy of the proposed approach, and we compare the results with those obtained by two numerical methods (for non-stiff and stiff systems) formulated directly for DDEs with distributed time delays.

\chg{Future work may involve the derivation of error bounds, an extension to Erlang mixture approximations with multiple rate parameters, and an extension to systems with time-varying kernels.}

	\section*{Acknowledgments}
The author would like to acknowledge many inspiring discussions on distributed time delays and delay differential equations with Prof. John Wyller from the Norwegian University of Life Sciences, Norway. \chgtwo{The author would also like to thank Jonas Lolle Bj\"{o}rnsson for fruitful discussions on Erlang mixture approximations. Finally, the author would like to thank the reviewers for their efforts to improve the generality of the theory presented in the paper.}

	\bibliographystyle{siamplain}
\bibliography{./ref/ref}

	\appendix

	\section{\chgtwo{Convergence of Taylor polynomial of exponential function}}\label{sec:exponential:function}%
\chgtwo{In this appendix, we show a result that involves a truncation of the sum in~\eqref{eq:exponential}, which is used in many of the proofs in this work.
\begin{lemma}\label{thm:exponential:taylor:polynomial}
	For any $t \in \Rnn$, the following expression converges to one as $a \rightarrow \infty$ if $M/a \rightarrow \infty$ as well:
	\begin{align}\label{eq:exponential:taylor:polynomial}
		e^{-at} \sum_{m=0}^M \frac{(at)^m}{m!} \rightarrow 1.
	\end{align}
\end{lemma}
\begin{proof}
	We use the bound on the error of the Taylor polynomial of $e^{at}$ around $t = 0$ (see, e.g.,~\cite[Thm.~1.3.6 and its proof]{Christensen:Christensen:2005}):
	\begin{align}
		\left|1 - e^{-at} \sum_{m=0}^M \frac{(at)^m}{m!}\right|
		&= e^{-at} \left|e^{at}  - \sum_{m=0}^M \frac{(at)^m}{m!}\right|
		\leq e^{-at} e^{at} \frac{(at)^{M+1}}{(M+1)!} = \frac{(at)^{M+1}}{(M+1)!}
	\end{align}
	Next, we show that this bound converges to zero if $M/a \rightarrow \infty$. Specifically, we show that its logarithm converges to $-\infty$:
	\begin{align}\label{sec:exponential:function:convergence:inequality}
		\ln \frac{(at)^{M+1}}{(M+1)!}
		&= (M+1) \ln at - \ln (M+1)! \\
		&\approx (M+1) \ln at - \left((M+1) \ln (M+1) - (M+1) + \frac{1}{2} \ln 2\pi (M+1)\right) \nonumber \\
		&\leq (M+1) \left(\ln at - \ln (M+1) + 1\right)
		= (M+1) \left(\ln \frac{a}{M+1} + \ln t + 1\right). \nonumber
	\end{align}
	We have used Stirling's approximation (see Section~\ref{sec:notation}) of $\ln (M+1)!$, and the last expression converges to $-\infty$ if $a/(M+1) \rightarrow 0$ or, equivalently, if $(M+1)/a \rightarrow \infty$. Note that the error of Stirling's approximation converges to zero as $M \rightarrow \infty$, i.e., the result will hold even though Stirling's approximation was used.
\end{proof}
}%
	\section{\chg{Erlang mixture delta family}}\label{sec:convergence}%
\chg{In this \chgtwo{appendix}, we introduce the Erlang mixture delta family and use it to derive an identity for Erlang mixture kernel\chgtwo{s}, which is used to prove \chgthree{pointwise} convergence \chgtwo{of Erlang mixture approximations} in Section~\ref{sec:erlang:mixture:kernel:convergence}. For a given rate parameter, $a$, the identity also provides \chgtwo{the} theoretical expressions for the coefficients \chgtwo{in~\eqref{eq:erlang:mixture:approximation}}, which depend on the kernel $\alpha$ that is being approximated. Furthermore, we prove a number of \chgtwo{statements} about the delta family that will also be used in Section~\ref{sec:erlang:mixture:kernel:convergence}.} \chg{Throughout this \chgtwo{appendix}, the time argument $t$ is a fixed parameter.}

\begin{lemma}[\chgtwo{Erlang mixture delta family}]\label{thm:erlang:mixture:delta:family}
	\chgtwo{Let $\delta_a: \Rnn \times \Rnn \rightarrow \Rnn$ be an \emph{Erlang mixture delta family} of order $M \in \Nnn$ with rate parameter $a \in \Rp$ defined as the piecewise constant function
	\begin{align}\label{eq:erlang:mixture:delta:family}
		\delta_a(t, s) &=
		\begin{cases}
			\ell_0(t), & s \in [s_0, s_1), \\
			&\vdots \\
			\ell_M(t), & s \in [s_M, s_{M+1}), \\
			0, & s \in [s_{M+1}, \infty),
		\end{cases}
	\end{align}
	}%
	where $\ell_m$ is an Erlang kernel of $m$'th order with rate parameter $a$. The \chgtwo{boundaries of the intervals over which $\delta_a$ is constant are $s_m = m \Delta s \in \Rnn$ for $m = 0, \ldots, M+1$, and the width} of the intervals is $\Delta s = s_{m+1} - s_m \in \Rp$, which depends on $a$\chgtwo{. I}n each interval, $\delta_a$ is equal to an Erlang kernel of different order.
	\chgtwo{Then, the following statement holds.
	\begin{enumerate}[label=(\alph*)]
		\item \label{thm:erlang:mixture:delta:family:kernel:identity}
		The Erlang mixture \chgthree{approximation}, $\hat \alpha: \Rnn \rightarrow \R$, defined in~\eqref{eq:erlang:mixture:approximation} satisfies the identity
		\begin{align}
			\hat \alpha(t) &= \int_0^\infty \delta_a(t, s) \alpha(s) \incr s.
		\end{align}
	\end{enumerate}
	Furthermore, the following asymptotic properties hold as $a \rightarrow \infty$ when $M/a \rightarrow \infty$ as well.
	\begin{enumerate}[resume, label=(\alph*)]
		\item \label{thm:erlang:mixture:delta:family:integral}
		The integral of $\delta_a$ over the second argument converges to one from below, i.e.,
		\begin{align}
			\int_0^\infty \delta_a(t, s) \incr s &\rightarrow 1, & \int_0^\infty \delta_a(t, s) \incr s \leq 1.
		\end{align}
		\item \label{thm:erlang:mixture:delta:family:zero:convergence}
		For any $\gamma \in \Rp$, $\delta_a(t, s)$ converges uniformly to zero for all $s \in \Rnn$ satisfying $\gamma \leq |t - s| < \infty$:
		\begin{align}
			\delta_a(t, s) &\rightarrow 0.
		\end{align}
		\item \label{thm:erlang:mixture:delta:family:sub:integral}
		For all $\gamma \in \Rp$ and $t \in \Rnn$,
		\begin{align}\label{eq:erlang:mixture:delta:family:sub:integral}
			\int_{[t - \gamma]^+}^{t + \gamma} \delta_a(t, s) \incr s &\rightarrow 1.
		\end{align}
	\end{enumerate}
	}
\end{lemma}
\begin{corollary}\label{lem:erlang:mixture:delta:t:equal:zero}
	For $t = 0$, the Erlang mixture delta family is given by
	\begin{align}
		\delta_a(0, s) &=
		\begin{cases}
			a, & s \in [0, 1/a), \\
			0, & \mathrm{otherwise}.
		\end{cases}
	\end{align}
\end{corollary}
\begin{proof}
	By direct substitution, $\ell_0(0) = a$ and $\ell_m(0) = 0$ for $m > 0$. Furthermore, by definition, $[s_0, s_1) = [0, 1/a)$.
\end{proof}

\begin{figure}[t]
	\centering
	\subfloat{\includegraphics[width=\textwidth]{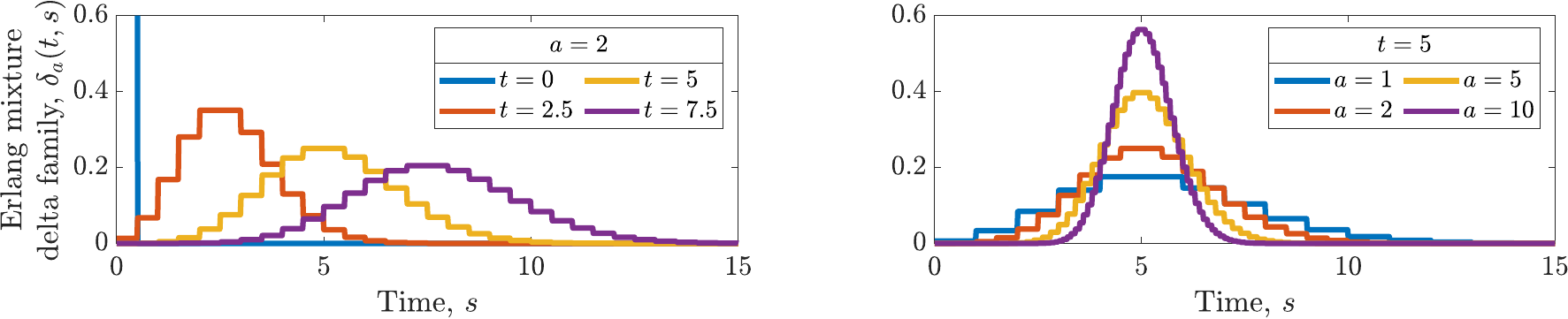}}
	\caption{\chgtwo{An infinite-order} Erlang mixture delta family for a fixed rate parameter, $a$, and different values of $t$ (left) and for fixed $t$ and different values of $a$ (right). For $t = 0$ in the left figure, the value of $\delta_a$ is 2 for $s \in [0, 0.5)$ (see also \chgtwo{Corollary}~\ref{lem:erlang:mixture:delta:t:equal:zero}).}
	\label{fig:erlang:mixture:delta:family}
\end{figure}
\begin{proof}[Proof \chgtwo{of Lemma~\ref{thm:erlang:mixture:delta:family}\ref{thm:erlang:mixture:delta:family:kernel:identity}}]
	\chg{
		The identity follows directly from the substitution of the definition of the Erlang mixture delta family in~\eqref{eq:erlang:mixture:delta:family}:
		\begin{align}
			\int_0^\infty \delta_a(t, s) \alpha(s) \incr s &= \sum_{m=0}^{\chgtwo{M}} \ell_m(t) \int_{s_m}^{s_{m+1}} \alpha(s) \incr s.
		\end{align}
		When the coefficients, $\{c_m\}_{m=0}^{\chmtwo{M}}$, are \chgtwo{defined as in~\eqref{eq:erlang:mixture:approximation}}, the expression on the right-hand side is \chgtwo{equal to the Erlang mixture \chgthree{approximation}, $\hat \alpha(t)$, in the same equation.}
	}
\end{proof}
\begin{proof}[Proof \chgtwo{of Lemma~\ref{thm:erlang:mixture:delta:family}\ref{thm:erlang:mixture:delta:family:integral}}]
	We substitute the expression for $\delta_a$ in~\eqref{eq:erlang:mixture:delta:family} and simplify:
	\begin{align}
		\int_0^\infty \delta_a(t, s) \incr s
		&= \sum_{m=0}^{\chmtwo{M}} \ell_m(t) \int_{s_m}^{s_{m+1}} 1 \incr s = \sum_{m=0}^{\chmtwo{M}} \ell_m(t) \Delta s = \sum_{m=0}^{\chmtwo{M}} \frac{(at)^m}{m!} a e^{-at} \frac{1}{a}
		 = \sum_{m=0}^{\chmtwo{M}} \frac{(at)^m}{m!} e^{-at}\chmtwo{.}
	\end{align}
	\chgtwo{The sum contains a finite number of terms in the expression~\eqref{eq:exponential} for $e^{at}$, and all terms are non-negative. Consequently, the expression is less than or equal to one. Furthermore, according to Lemma~\ref{thm:exponential:taylor:polynomial}, it converges to one} as $a \rightarrow \infty$ \chgtwo{when $M/a \rightarrow \infty$ as well}.
\end{proof}
\begin{lemma}\label{lem:erlang:mixture:delta:family:monotonicity}
	\chgtwo{For a sufficiently large ratio $M/a$, a}n Erlang mixture delta family, $\delta_a: \Rnn \times \Rnn \rightarrow \Rnn$ \chgtwo{of order $M \in \Nnn$ with rate parameter $a \in \Rp$}, is non-decreasing in the second argument, $s$, when it is below the first argument, $t$, and non-increasing when it is above. Specifically, since $\delta_a$ is piecewise constant over intervals of size $\Delta s$, $\delta_a(t, s) \leq \delta_a(t, s + \Delta s)$ for $s < t$, and $\delta_a(t, s) \geq \delta_a(t, s + \Delta s)$ for $s \geq t$.
\end{lemma}
\begin{proof}
	Let $s \in [s_{m-1}, s_m)$. Then, $s + \Delta s \in [s_m, s_{m+1})$, and we derive a condition on $m$ for $\delta_a$ to be non-decreasing:
	\begin{align}
		\delta_a(t, s) &\leq \delta_a(t, s + \Delta s).
	\end{align}
	\chgtwo{When the ratio $M/a$ is sufficiently large, the left- and right-hand sides are equal to $\ell_{m-1}(t)$ and $\ell_m(t)$, respectively, and we substitute their expressions:}
	\begin{align}
		\frac{(at)^{m-1}}{(m-1)!} a e^{-at} &\leq \frac{(at)^m}{m!} a e^{-at}\chmtwo{.}
	\end{align}
	\chgtwo{Most factors cancel out,} and \chgtwo{we} obtain the condition
	\begin{align}
		m &\leq at.
	\end{align}
	This corresponds to $s_m = m \Delta s = m/a \leq t$, and $s < s_m$ by assumption. Consequently, $\delta_a$ is non-decreasing in $s$ for $s < t$, and the proof that $\delta_a$ is non-increasing in $s$ for $s \geq t$ is analogous.
\end{proof}
\begin{proof}[Proof \chgtwo{of Lemma~\ref{thm:erlang:mixture:delta:family}\ref{thm:erlang:mixture:delta:family:zero:convergence}}]
	In the proof, we will use Stirling's approximation \chgtwo{(see Section~\ref{sec:notation}) and that the approximation converges, i.e., asymptotic results for the approximate quantity will also hold for the original quantity.}
	We will show that for given $\epsilon \in \Rp$ and for fixed $s$ and $t$, it is possible to choose $a$ sufficiently large that the inequality
	\begin{align}
		\delta_a(t, s) &< \epsilon
	\end{align}
	is satisfied if $s \neq t$. \chgtwo{Since $M/a \rightarrow \infty$, $\delta_a(t, s) = \ell_m(t)$ for sufficiently large $a$, and f}or simplicity, we assume that $a$ is chosen such that $m = as$ is integer. First, we assume that $t > 0$ and substitute $\ell_m$ and the value of $m$:
	\begin{align}
		\chmthree{\delta_a(t, s) =} \frac{(at)^{as}}{(as)!} a e^{-at} &< \epsilon.
	\end{align}
	\chgthree{For $s = 0$, we obtain $a e^{-at} < \epsilon$, and since $t > 0$, the inequality is satisfied for sufficiently large $a$.} Next, \chgthree{for $s > 0$,} we take the logarithm on both sides \chgthree{in order to apply Stirling's approximation:}
	\begin{align}
		\chmthree{\ln \delta_a(t, s) =} as \ln at - \ln (as)! + \ln a - at &< \ln \epsilon\chmthree{.}
	\end{align}
	\chgthree{We use} Stirling's approximation \chgthree{to approximate $\delta_a$ (specifically, the term $\ln (as)!$), and we denote the result $\hat \delta_a$. The objective is now to show that $\hat \delta_a(t, s) \rightarrow 0$ for $s \neq t$ and $t, s > 0$, i.e., that we can choose $a$ large enough that}
	\begin{align}
		\chmthree{\ln \hat \delta_a(t, s) =} as \ln at - \left(as \ln as - as + \frac{1}{2} \ln 2 \pi as\right) + \ln a - at &< \ln \epsilon.
	\end{align}
	Finally, we rearrange terms:
	\begin{align}
		a \left(s \left(\ln \frac{t}{s} + 1\right) - t\right) + \frac{1}{2} \ln a &< \ln \epsilon + \frac{1}{2} \ln 2 \pi s.
	\end{align}
	For $s \neq t$, the factor of $a$ in the first term is always negative, and it is always possible to satisfy the bound for sufficiently large $a$. This follows from the inequality $\ln x \leq x - 1$ for $x > 0$\chgthree{, where equality only holds for $x = 1$}. In contrast, if $s = t$, the factor of $a$ is zero, and the inequality is only satisfied if
	\begin{align}
		\frac{1}{2} \ln a &< \ln \epsilon + \frac{1}{2} \ln 2 \pi s & \mathrm{or} &&
		a < 2 \pi s \epsilon^2,
	\end{align}
	i.e., if $a$ is bounded from above. \chgthree{In conclusion, $\hat \delta_a(t, s) \to 0$ for $s \neq t$, which shows that $\delta_a(t, s) \to 0$ for $s \neq t$ and $t > 0$ since $\hat \delta_a \to \delta_a$ pointwise. Specifically, we can choose $a$ large enough that $\hat \delta_a(t, s) < \epsilon/2$ and $|\hat \delta_a(t, s) - \delta_a(t, s)| < \epsilon/2$ and use a triangle inequality:}
	\begin{align}
		\chmthree{\delta_a(t, s) &= \delta_a(t, s) - \hat \delta_a(t, s) + \hat \delta_a(t, s) \leq |\delta_a(t, s) - \hat \delta_a(t, s)| + \hat \delta_a(t, s) \leq \epsilon/2 + \epsilon/2 = \epsilon.}
	\end{align}

	For $t = 0$, $\delta_a(t, s)$ is only nonzero for $s < 1/a$ (see \chgtwo{Corollary}~\ref{lem:erlang:mixture:delta:t:equal:zero}). Consequently, for a fixed $s \neq t$, it is always possible to choose $a$ large enough that $\delta_a(t, s) = 0$.
	\chgtwo{Finally, the convergence is uniform due to the monotonicity properties shown in Lemma~\ref{lem:erlang:mixture:delta:family:monotonicity}, i.e., for a given $\gamma > 0$, $a$ can be chosen large enough that $\max\{\delta_a(t, [t-\gamma]^+), \delta_a(t, t+\gamma)\} < \epsilon$ and the same will hold for all other values of $s \in \Rnn$ satisfying $\gamma \leq |t - s| < \infty$.}
\end{proof}
\begin{proof}[Proof \chgtwo{of Lemma~\ref{thm:erlang:mixture:delta:family}\ref{thm:erlang:mixture:delta:family:sub:integral}}]
	First, we assume that $t > 0$ and bound the integral from above and below using Lemma~\ref{lem:theorem:kallenberg}:
	\begin{align}
		\chmtwo{\frac{1}{k}} - \frac{\sigma^2(t)}{\bar \gamma^2\chmtwo{(t)}} \leq \int_{\chmtwo{k} \mu(t) - \bar \gamma\chmtwo{(t)}}^{\chmtwo{k} \mu(t) + \bar \gamma\chmtwo{(t)}} \delta_a(t, s) \incr s \leq \int_{[\chmtwo{k} \mu(t) - \gamma]^+}^{\chmtwo{k} \mu(t) + \gamma} \delta_a(t, s) \incr s \leq \int_0^\infty \delta_a(t, s) \incr s \chmtwo{= \frac{1}{k}}.
	\end{align}
	Here, $\bar \gamma\chmtwo{(t)} = \min\{\chmtwo{\chmthree{k}\mu(}t\chmtwo{)}, \gamma\}$ such that the limits in the leftmost integral are non-negative\chgtwo{, and $k = \left(\int_0^\infty \delta_a(t, s) \incr s\right)^{-1}$}. Consequently, since $\mu(t) \rightarrow t$\chgtwo{,} $\sigma^2(t) \rightarrow 0$\chgtwo{, $k \rightarrow 1$, and $\chmthree{\bar \gamma}(t) \rightarrow \min\{t, \gamma\} > 0$} as $a \rightarrow \infty$ \chgtwo{and $M/a \rightarrow \infty$ (see Lemma~\ref{thm:erlang:mixture:delta:family}\ref{thm:erlang:mixture:delta:family:integral} and Lemma~\ref{lem:erlang:mixture:delta:family:mean}--\ref{lem:erlang:mixture:delta:family:variance})}, the integral in~\eqref{eq:erlang:mixture:delta:family:sub:integral} approaches one. Next, if $t = 0$, the integral in~\eqref{eq:erlang:mixture:delta:family:sub:integral} is equal to one for all $a \geq 1/\gamma$.
\end{proof}
	\section{\chgtwo{Erlang mixture delta family: Mean, variance, and Chebyshev's inequality}}%
\chgtwo{In this appendix, we use results from probability theory to prove bounds that are used in the proof of Lemma~\ref{thm:erlang:mixture:delta:family}\ref{thm:erlang:mixture:delta:family:sub:integral}. W}e will draw on the analogy of $\delta_a$ to the probability density function of a non-negative random variable, e.g., by determining its mean and variance.
\begin{definition}
	\chg{For given $t\in\Rnn$, t}he mean, $\mu: \Rnn \rightarrow \Rnn$, associated with a delta family, $\delta_a: \Rnn \times \Rnn \rightarrow \Rnn$, is given by the integral
	\begin{align}\label{eq:definition:mean}
		\mu(t)
		&= \int_0^\infty \delta_a(t, s) s \incr s.
	\end{align}
\end{definition}
\begin{definition}
	\chg{For given $t\in\Rnn$, t}he variance, $\sigma^2: \Rnn \rightarrow \Rnn$, associated with a delta family, $\delta_a: \Rnn \times \Rnn \rightarrow \Rnn$, is given by the integral
	\begin{align}
		\sigma^2(t) &= \int_0^\infty \delta_a(t, s) (s - \mu(t))^2 \incr s,
	\end{align}
	or, equivalently, by
	\begin{align}\label{eq:erlang:mixture:delta:family:variance}
		\sigma^2(t) &= \mu_2(t) - \mu^2(t), &
		\mu_2(t) &= \int_0^\infty \delta_a(t, s) s^2 \incr s,
	\end{align}
	where $\mu: \Rnn \rightarrow \Rnn$ is the mean given by~\eqref{eq:definition:mean}.
\end{definition}
\begin{lemma}\label{lem:erlang:mixture:delta:family:mean}
	\chg{For given $t \in \Rnn$, t}he mean, $\mu: \Rnn \rightarrow \Rnn$, associated with an Erlang mixture delta family, $\delta_a: \Rnn \times \Rnn \rightarrow \Rnn$, \chgtwo{of order $M \in \Nnn$} with rate parameter $a$ \chgtwo{converges to $t$,}
	\begin{align}\label{eq:erlang:mixture:delta:family:mean}
		\mu(t) &\chmtwo{\rightarrow} t,
	\end{align}
	\chgtwo{as $a \rightarrow \infty$ if $M/a \rightarrow \infty$ as well. Furthermore,} for finite $a \in \Rp$, $\mu(t) > 0$ for all $t \in \Rnn$.
\end{lemma}
\begin{proof}
	We substitute the definition of $\delta_a$ in~\eqref{eq:erlang:mixture:delta:family}:
	\begin{align}\label{eq:erlang:mixture:delta:family:mu:proof}
		\mu(t)
		&= \int_0^\infty \delta_a(t, s) s \incr s = \sum_{m=0}^{\chmtwo{M}} \ell_m(t) \int_{s_m}^{s_{m+1}} s \incr s.
	\end{align}
	Next, we write out the expression for the integral,
	\begin{align}
		\int_{s_m}^{s_{m+1}} s \incr s
		&= \frac{1}{2} (s_{m+1}^2 - s_m^2) = \frac{1}{2} (s_{m+1} - s_m)(s_{m+1} + s_m) = \frac{1}{2} \Delta s ((m+1) + m) \Delta s \\
		&= \frac{2m + 1}{2} \Delta s^2, \nonumber
	\end{align}
	and substitute into~\eqref{eq:erlang:mixture:delta:family:mu:proof}:
	\begin{align}
		\mu(t)
		&= \sum_{m=0}^{\chmtwo{M}} \frac{(at)^m}{m!} a e^{-at} \frac{2m+1}{2} \Delta s^2 = t \sum_{m=1}^{\chmtwo{M}} \frac{(at)^{m-1}}{(m-1)!} e^{-at} + \frac{1}{2a} \sum_{m=0}^{\chmtwo{M}} \frac{(at)^m}{m!} e^{-at} \chmtwo{\rightarrow t}
	\end{align}
	\chgtwo{as $a \rightarrow \infty$ when $M/a \rightarrow \infty$ as well.} Here, we have used that $\Delta s = 1/a$\chgtwo{, and Lemma~\ref{thm:exponential:taylor:polynomial}. Specifically, the first term converges to the factor $t$, and the second term converges to zero because of the factor $1/(2a)$}. \chgthree{For $t > 0$, $\mu$ is positive because all terms are positive, and for $t = 0$, $\mu(t) = 1/(2a) > 0$.}
\end{proof}
\begin{lemma}\label{lem:erlang:mixture:delta:family:variance}
	\chg{For given $t \in \Rnn$, t}he variance, $\sigma^2: \Rnn \rightarrow \Rnn$, associated with an Erlang mixture delta family, $\delta_a: \Rnn \times \Rnn \rightarrow \Rnn$, \chgtwo{of order $M \in \Nnn$} with rate parameter $a$ \chgtwo{converges to zero,}
	\begin{align}
		\sigma^2(t) &\rightarrow 0,
	\end{align}
	\chgtwo{as $a \rightarrow \infty$ if $M/a \rightarrow \infty$ as well.}
\end{lemma}

\begin{proof}
	First, we write out the integral in the definition of $\mu_2$ in~\eqref{eq:erlang:mixture:delta:family:variance}:
	\begin{align}\label{eq:erlang:mixture:delta:family:variance:proof}
		\mu_2(t) &= \int_0^\infty \delta_a(t, s) s^2 \incr s = \sum_{m=0}^{\chmtwo{M}} \ell_m(t) \int_{s_m}^{s_{m+1}} s^2 \incr s.
	\end{align}
	Next, we write out the integral of $s^2$,
	\begin{align}
		\int_{s_m}^{s_{m+1}} s^2 \incr s
		&= \frac{1}{3} (s_{m+1}^3 - s_m^3) = \frac{1}{3} (s_{m+1} - s_m) (s_{m+1}^2 + s_{m+1} s_m + s_m^2) \\
		&= \frac{1}{3} \Delta s ((m+1)^2 + (m+1) m + m^2) \Delta s^2 \nonumber \\
		&= \left(m (m-1) + 2m + \frac{1}{3}\right) \Delta s^3, \nonumber
	\end{align}
	and substitute it into the right-hand side of~\eqref{eq:erlang:mixture:delta:family:variance:proof}:
	\begin{align}
		\mu_2(t) &= \sum_{m=0}^{\chmtwo{M}} \frac{(at)^m}{m!} a e^{-at} \left(m (m-1) + 2m + \frac{1}{3}\right) \Delta s^3.
	\end{align}
	We \chgtwo{investigate the convergence of} each of the three sums separately:
	\begin{subequations}\label{eq:erlang:mixture:delta:family:variance:terms}
		\begin{align}
			\sum_{m=0}^{\chmtwo{M}} \frac{(at)^m}{m!} a e^{-at} m (m-1) \Delta s^3
			= t^2 \sum_{m=2}^{\chmtwo{M}} \frac{(at)^{m-2}}{(m-2)!} e^{-at} &\chmtwo{\rightarrow t^2}, \\
			\sum_{m=0}^{\chmtwo{M}} \frac{(at)^m}{m!} a e^{-at} 2m \Delta s^3
			= 2 \frac{t}{a} \sum_{m=1}^{\chmtwo{M}} \frac{(at)^{m-1}}{(m-1)!} e^{-at} &\chmtwo{\rightarrow 0}, \\
			\sum_{m=0}^{\chmtwo{M}} \frac{(at)^m}{m!} a e^{-at} \frac{1}{3} \Delta s^3
			= \frac{1}{3 a^2} \sum_{m=0}^{\chmtwo{M}} \frac{(at)^m}{m!} e^{-at} &\chmtwo{\rightarrow 0,}
		\end{align}
	\end{subequations}
	\chgtwo{as $a \rightarrow \infty$ when $M/a \rightarrow \infty$ as well.} In all three cases, we have used that $\Delta s = 1/a$ \chgtwo{and Lemma~\ref{thm:exponential:taylor:polynomial}}. \chgtwo{The first term converges to the factor $t^2$, and the last two terms convergence to zero because of the factors $2t/a$ and $1/(3a^2)$.}
	\chgtwo{Consequently, $\mu_2(t) \rightarrow t^2$, and we use}
	the expression for \chgthree{the} \chgtwo{limit of} the mean~\eqref{eq:erlang:mixture:delta:family:mean} \chgtwo{and} the definition of $\sigma^2$ in~\eqref{eq:erlang:mixture:delta:family:variance} \chgtwo{to conclude the proof}:
	\begin{align}
		\sigma^2(t) &= \mu_2(t) - \mu^2(t) \rightarrow t^2 - t^2 = 0.
	\end{align}
	\chgtwo{We have used that $\mu^2(t) \rightarrow t^2$ when $\mu(t) \rightarrow t$ because the square is continuous.}
\end{proof}

\begin{lemma}\label{lem:theorem:kallenberg}
	Let $\delta: \Rnn \times \Rnn \rightarrow \Rnn$ be an Erlang mixture delta family with mean $\mu: \Rnn \rightarrow \Rnn$, \chg{variance $\sigma^2: \Rnn \rightarrow \Rnn$,} and rate parameter $a \in \Rp$. \chgtwo{Furthermore, let $k = \left(\int_0^\infty \delta_a(t, s) \incr s\right)^{-1} \in \Rnn$.} Then, the inequality
	\begin{align}\label{eq:chebyshev:inequality}
		\chmtwo{\frac{1}{k}} - \int_{\chmtwo{k} \mu(t) - \gamma}^{\chmtwo{k} \mu(t) + \gamma} \delta_a(t, s) \incr s \leq \frac{\sigma^2(t)}{\gamma^2}
	\end{align}
	holds for all $\gamma \in \Rp$ that satisfy $\gamma \leq \chmthree{k} \mu(t)$.
\end{lemma}
\begin{proof}
	This follows from Chebyshev's (or Markov's) well-known inequality in probability theory and the fact that $\mu(t) > 0$ for all $t \in \Rnn$ and finite rate parameters, $a \in \Rp$. See, e.g., Lemma~4.1 in the book by Kallenberg~\cite{Kallenberg:2002}. \chgtwo{In order to apply Chebyshev's inequality, we derive the mean and variance for the scaled Erlang mixture delta family, $k \delta_a$. The mean is also scaled by $k$,
	\begin{align}
		\int_0^\infty k \delta_a(t, s) s \incr s &= k \int_0^\infty \delta_a(t, s) s \incr s = k \mu(t),
	\end{align}
	whereas the variance is not directly proportional to $\sigma^2$:
	\begin{subequations}
		\begin{align}
			\int_0^\infty k \delta_a(t, s) s^2 \incr s &= k \int_0^\infty \delta_a(t, s) s^2 \incr s = k \mu_2(t), \\
			\int_0^\infty k \delta_a(t, s) (s - k \mu(t))^2 \incr s &= k \mu_2(t) - k^2 \mu^2(t).
		\end{align}
	\end{subequations}
	However, since $k \geq 1$, $-k^2 \leq -k$ and the variance is bounded by $k \mu_2(t) - k \mu^2(t) = k \sigma^2(t)$. Next, Chebyshev's inequality for $k \delta_a$ is
	\begin{align}
		1 - \int_{k \mu(t) - \gamma}^{k \mu(t) + \gamma} k \delta_a(t, s) \incr s &\leq \frac{k \mu_2(t) - k^2 \mu^2(t)}{\gamma} \leq \frac{k \sigma^2(t)}{\gamma},
	\end{align}
	where we have used the bound on the variance of $k \delta_a$. The inequality~\eqref{eq:chebyshev:inequality} is obtained by dividing by $k$.
	}
\end{proof}
	\section{\chgtwo{Proof of \chgthree{pointwise} convergence based on delta families}}%
\chg{As the delta family $\delta_a$ presented in \chgtwo{Appendix}~\chgtwo{\ref{sec:convergence}} depends on both $t$ and $s$ (and not only their difference), the proof \chgtwo{that the Erlang mixture approximation converges \chgthree{pointwise}} requires} an adaptation \chg{(\chgtwo{Lemma}~\ref{thm:convergence:delta:family})} of the theorem presented in Section~9.3 of \chg{Chapter~2 in} the book by Korevaar~\cite{Korevaar:1968} \chg{on} the convergence of approximations based on delta families.
\begin{lemma}\label{thm:convergence:delta:family}
	Let $\delta_a: \Rnn \times \Rnn \rightarrow \Rnn$ be a delta family parametrized by $a \in \Rp$ which satisfies the following conditions \chg{for any given $t \in \Rnn$}.
	\begin{enumerate}
		\item For all $\gamma \in \Rp$,
		\begin{align}
			\int_{[t - \gamma]^+}^{t + \gamma} \delta_a(t, s) \incr s \rightarrow 1
		\end{align}
		as $a \rightarrow \infty$.
		\item \chg{For any $\gamma \in \Rp$, $\delta_a(t, s) \rightarrow 0$ uniformly for $s \in \Rnn$ satisfying $\gamma \leq |t-s| < \infty$ as $a \rightarrow \infty$.}
		\item \chg{For all $s \in \Rnn$ and $a \in \Rp$, the delta family is non-negative, i.e., $\delta_a(t, s) \geq 0$}.
	\end{enumerate}
	Furthermore, let $\alpha: \Rnn \rightarrow \chmtwo{\R}$ be a regular kernel.
	Then, the approximation $\hat \alpha: \Rnn \rightarrow \chmtwo{\R}$ given by
	\begin{align}\label{eq:theorem:approximation}
		\hat \alpha(t) &= \int_0^\infty \delta_a(t, s) \alpha(s) \incr s
	\end{align}
	converges \chgthree{pointwise} to $\alpha$ for all $t \in \Rnn$ as $a \rightarrow \infty$.
\end{lemma}
\begin{proof}
	First, we split the integral into two terms:
	\begin{align}
		\hat \alpha(t)
		&= \int_0^\infty \delta_a(t, s) \alpha(s) \incr s \\
		&= \underbrace{\int_{[t - \gamma]^+}^{t + \gamma} \delta_a(t, s) \alpha(s) \incr s}_{\mathcal I_1(t)} + \underbrace{\int_0^{[t - \gamma]^+} \delta_a(t, s) \alpha(s) \incr s + \int_{t + \gamma}^\infty \delta_a(t, s) \alpha(s) \incr s}_{\mathcal I_2(t)}\hspace{-2pt}. \nonumber
	\end{align}
	Next, we add zero to the first term and split the integral:
	\begin{align}
		\mathcal I_1(t) &= \underbrace{\alpha(t) \int_{[t - \gamma]^+}^{t + \gamma} \delta_a(t, s) \incr s}_{\mathcal I_3(t)} + \underbrace{\int_{[t - \gamma]^+}^{t + \gamma} \delta_a(t, s) (\alpha(s) - \alpha(t)) \incr s}_{\mathcal I_4(t)}\hspace{-2pt}.
	\end{align}
	The \chgtwo{absolute value of the} second term,  $\mathcal I_4(t)$, can be bounded through the continuity of $\alpha$. Specifically, choose $\gamma$ such that $|\alpha(s) - \alpha(t)| < \epsilon$ for $|s - t| < \gamma$. Then,
	\begin{align}
		\chmtwo{|}\mathcal I_4(t)\chmtwo{|} < \chmtwo{\int_{[t - \gamma]^+}^{t + \gamma} \delta_a(t, s) |\alpha(s) - \alpha(t)| \incr s <} \epsilon \int_{[t - \gamma]^+}^{t + \gamma} \delta_a(t, s) \incr s \chmthree{\leq} \epsilon.
	\end{align}
	Next, we prove that the term $\mathcal I_2(t) \rightarrow 0$ as $a \rightarrow \infty$. We bound $\chmtwo{|}\mathcal I_2\chmtwo{|}$ from above by substituting the supremum of $\delta_a$, bounding the integrals, and utilizing the non-negativity of $\delta_a$:
	\begin{align}
		\chmtwo{|}\mathcal I_2(t)\chmtwo{|} &\chgtwo{\leq} \int_0^{[t - \gamma]^+} \delta_a(t, s) \chmtwo{|}\alpha(s)\chmtwo{|} \incr s + \int_{t + \gamma}^\infty \delta_a(t, s) \chmtwo{|}\alpha(s)\chmtwo{|} \incr s \leq \chmtwo{\rho(\infty)} \sup_{\substack{s \in \Rnn \\ \gamma \leq |t - s| < \infty}} \delta_a(t, s).
	\end{align}
	Furthermore, we note that by assumption, the supremum of $\delta_a$ over values of $s$ outside of an interval around $t$ goes to zero as $a \rightarrow \infty$.
	We choose $a$ large enough that $\chmtwo{|}\mathcal I_2(t)\chmtwo{|} < \epsilon$. Additionally, we choose $a$ large enough that
	\begin{align}
		1 - \int_{[t - \gamma]^+}^{t + \gamma} \delta_a(t, s) \incr s < \epsilon.
	\end{align}
	Then,
	\begin{align}
		|\mathcal I_3(t) - \alpha(t)| &= \chmtwo{|}\alpha(t)\chmtwo{|} \left(1 - \int_{[t - \gamma]^+}^{t + \gamma} \delta_a(t, s) \incr s\right) < \chmtwo{|}\alpha(t)\chmtwo{|} \epsilon,
	\end{align}
	and by utilizing the above bounds, we can choose $a$ sufficiently large that
	\begin{align}
		|\hat \alpha(t) - \alpha(t)|
		&\chmthree{\leq} |\mathcal I_1(t) - \alpha(t)| + \chmtwo{|}\mathcal I_2(t)\chmtwo{|} \\
		&\chmthree{\leq} |\mathcal I_3(t) - \alpha(t)| + \chmtwo{|}\mathcal I_4(t)\chmtwo{|} + \chmtwo{|}\mathcal I_2(t)\chmtwo{|} \nonumber \\
		&< (\chmtwo{|}\alpha(t)\chmtwo{|} + 2) \epsilon, \nonumber \\
		&\chmtwo{\leq (K+2) \epsilon} \nonumber
	\end{align}
	for any $\epsilon \in \Rp$. We have made repeated use of triangle inequalities. To summarize, for any $\bar \epsilon \in \Rp$, we can choose $\bar a \in \Rp$ such that $|\hat \alpha(t) - \alpha(t)| < \bar \epsilon$ for all $a \geq \bar a$ \chgtwo{when $\alpha$ is bounded}. \chgthree{Note that although the bound is independent of $t$, the convergence is not uniform because $\gamma$ must be chosen based on $t$, and $a$ must be chosen based on $\gamma$. However, any continuous function is uniformly continuous on a bounded interval~\cite[Thm.~4.19]{Rudin:1976}, such that $\hat \alpha$ converges uniformly to $\alpha$ for bounded intervals of $t$.}
\end{proof}
	\section{Properties of Erlang mixture \chgthree{approximations}}\label{sec:erlang:mixture:kernel:properties}%
\chg{In this \chgtwo{appendix}, we prove \chgtwo{Corollary~\ref{thm:erlang:mixture:approximation:exponential:boundedness} and present lemmas on two other} properties of Erlang mixture \chgthree{approximations} given by~\chgtwo{\eqref{eq:erlang:mixture:approximation}} that \chgthree{is} used either to prove convergence of the integral of the absolute error in Section~\ref{sec:erlang:mixture:kernel:convergence} or to prove a relation between the stability analyses of the original system in Section~\ref{sec:system} and the \chgthree{Erlang ODE approximation} presented in \chgtwo{Theorem~\ref{thm:ode:approximation}}.}
\begin{proof}[Proof \chgtwo{of Corollary~\ref{thm:erlang:mixture:approximation:exponential:boundedness}}]
	\chg{
		First, we bound the following exponential\footnote{\chg{The convexity of decaying exponential functions implies that $e^{-x} \leq 1 - \frac{1 - e^{-\omega}}{\omega} x$ for $x \in [0, \omega]$.}} for $\rho/a \leq \omega$:
		\begin{align}
			e^{-\rho s_m} &= e^{-\rho m/a} = \left(e^{-\rho/a}\right)^m \leq \left(1 - \frac{1 - e^{-\omega}}{\omega} \frac{\rho}{a}\right)^m = \left(1 - \frac{\bar \rho}{a}\right)^m = \left(\frac{a - \bar \rho}{a}\right)^m.
		\end{align}
		Next, we substitute the exponential bound in the expression for the coefficients \chgtwo{in~\eqref{eq:erlang:mixture:approximation}}:
		\begin{align}
			\chmtwo{|}c_m\chmtwo{|} &\chmtwo{\leq} \int_{s_m}^{s_{m+1}} \chmtwo{|}\alpha(s)\chmtwo{|} \incr s
			\leq L \int_{s_m}^{s_{m+1}} e^{-\rho s} \incr s = -\frac{L}{\rho} (e^{-\rho s_{m+1}} - e^{-\rho s_m}) \\
			&= \frac{L}{\rho} e^{-\rho s_m} \left(1 - e^{-\rho (s_{m+1} - s_m)}\right)
			= \frac{L}{\rho} e^{-\rho s_m} (1 - e^{-\rho \Delta s})
			\leq L e^{-\rho s_m} \Delta s
			= L e^{-\rho s_m} \frac{1}{a} \nonumber \\
			&\leq L \left(\frac{a - \bar \rho}{a}\right)^m \frac{1}{a}. \nonumber
		\end{align}
		Here, we have used the bound $1 - e^{-\rho \Delta s} \leq \rho \Delta s$. Next, we substitute this bound into the Erlang mixture kernel:
		\begin{align}
			\chmtwo{|}\hat \alpha(t)\chmtwo{|}
			&\chmtwo{\leq} \sum_{m=0}^{\chmtwo{M}} \chmtwo{|}c_m\chmtwo{|} \ell_m(t)
			\leq \sum_{m=0}^{\chmtwo{M}} L \left(\frac{a - \bar \rho}{a}\right)^m \frac{1}{a} \frac{(at)^m}{m!} a e^{-at}
			= L \sum_{m=0}^{\chmtwo{M}} \frac{((a - \bar \rho) t)^m}{m!} e^{-at} \\
			&= L e^{-\bar \rho t} \sum_{m=0}^{\chmtwo{M}} \frac{((a - \bar \rho) t)^m}{m!} e^{-(a - \bar \rho) t}
			\chmtwo{\leq} L e^{-\bar \rho t} e^{(a - \bar \rho) t} e^{-(a - \bar \rho) t} = L e^{-\bar \rho t}. \nonumber
		\end{align}
		Again, we have used the identity~\eqref{eq:exponential} \chgtwo{and that all terms in the sum are positive for sufficiently large $a$}. The bound $1 - e^{-\omega} \chmthree{<} \omega$ \chgthree{for $\omega \in \Rp$} provides the bound $\bar \rho \chmthree{<} \rho$, and using the identity~\eqref{eq:exponential}, we show that $\bar \rho \rightarrow \rho$ as $\omega \rightarrow 0$:
		\begin{align}
			\bar \rho
			&= \rho \frac{1 - e^{-\omega}}{\omega}
			= \rho \frac{1 - \sum\limits_{m=0}^\infty \frac{(-\omega)^m}{m!}}{\omega}
			= \rho \sum\limits_{m=1}^\infty \frac{(-\omega)^{m-1}}{m!}.
		\end{align}
		The first term in the sum is one, and all other terms converge to zero as $\omega \rightarrow 0$.
	}
\end{proof}
\begin{lemma}\label{lem:uniformly:integrable:erlang:mixture:approximation}
\chg{
	If the Erlang mixture kernel $\hat \alpha: \Rnn \rightarrow \chmtwo{\R}$ is bounded, it is uniformly integrable, i.e., for all $a \in \Rp$ and for every $\epsilon \in \Rp$, there exists a $\Delta t \in \Rp$ such that if $|t_1 - t_0| < \Delta t$, then,
	\begin{align}
		\int_{t_0}^{t_1} \chmtwo{|}\hat \alpha(t)\chmtwo{|} \incr t &< \epsilon.
	\end{align}
}
\end{lemma}
\begin{proof}
\chg{
	The uniform integrability follows directly from the boundedness:
	\begin{align}
		\int_{t_0}^{t_1} \chmtwo{|}\hat \alpha(t)\chmtwo{|} \incr t &\leq \int_{t_0}^{t_1} K \incr t = K |t_1 - t_0| < K \Delta t.
	\end{align}
	Here, $K$ is the upper bound on $\hat \alpha$, and the bound is satisfied for $\Delta t = \epsilon/K$.
}
\end{proof}
\begin{lemma}\label{lem:tight:erlang:mixture:approximation}
\chg{
	Let \chgtwo{$\alpha: \Rnn \rightarrow \R$ be a regular kernel, and let} $\hat \alpha: \Rnn \rightarrow \chmtwo{\R}$ be an Erlang mixture approximation \chgtwo{of order $M \in \Nnn$ given by~\eqref{eq:erlang:mixture:approximation} with} rate parameter $a \in \Rp$. Then, $\hat \alpha$ is \emph{tight}~\chgtwo{\cite[Sec.~18.3, p.~376]{Royden:Fitzpatrick:2010}}, i.e., for every $\epsilon \in \Rp$, there exists a $t_0 \in \Rnn$ such that
	\begin{align}
		\int_{t_0}^\infty \chmtwo{|}\hat \alpha(t)\chmtwo{|} \chmthree{\incr t} < \epsilon
	\end{align}
	for all $a \in \chmthree{\Rp}$. The result also holds for infinite orders, $M$.
}
\end{lemma}
\begin{proof}
	\chgtwo{In order to prove the statement, we present an approach for choosing $t_0$ for a given $\epsilon$. First, we bound the expression and split the resulting sum:
	\begin{align}\label{eq:tight:proof:split:sum}
		\int_{t_0}^\infty |\hat \alpha(t)| \incr t
		&= \int_{t_0}^\infty \left|\sum_{m=0}^M c_m \ell_m(t)\right| \incr t
		\leq \sum_{m=0}^\infty |c_m| \int_{t_0}^\infty \ell_m(t) \incr t \\
		&= \sum_{m=0}^{m^*} |c_m| \int_{t_0}^\infty \ell_m(t) \incr t + \sum_{m=m^*+1}^\infty |c_m| \int_{t_0}^\infty \ell_m(t) \incr t. \nonumber
	\end{align}
	Note that we include infinitely many terms in the first inequality.
	Next, we bound the second sum:
	\begin{align}\label{eq:tight:proof:sum:2}
		\sum_{m=m^*+1}^\infty |c_m| \int_{t_0}^\infty \ell_m(t) \incr t
		&\leq \sum_{m=m^*+1}^\infty |c_m|
		= \sum_{m=m^*+1}^\infty \left|\int_{s_m}^{s_{m+1}} \alpha(s) \incr s\right| \\
		&\leq \sum_{m=m^*+1}^\infty \int_{s_m}^{s_{m+1}} |\alpha(s)| \incr s
		= \int_{s_{m^*+1}}^\infty |\alpha(s)| \incr s \nonumber \\
		&= \rho(\infty) - \rho(s_{m^*+1}) = \rho(\infty) - \rho\left(\frac{m^*+1}{a}\right). \nonumber
	\end{align}
	For the remainder of the proof, we fix the ratio between $m^*+1$ and $a$ to $\gamma$, i.e., $(m^* + 1)/a = \gamma$, and we choose $\gamma$ large enough that $\rho(\infty) - \rho(\gamma) < \epsilon/2$.
	In order to bound the first sum in the rightmost expression in~\eqref{eq:tight:proof:split:sum}, we write out the integral of the Erlang kernel:
	\begin{align}
		\int_{t_0}^\infty \ell_m(t) \incr t
		&= 1 - \int_0^{t_0} \ell_m(t) \incr t
		= 1 - \left(1 - \sum_{n=0}^m \frac{(at_0)^n}{n!} e^{-at_0}\right) 
		 = \sum_{n=0}^m \frac{(at_0)^n}{n!} e^{-at_0}. 
	\end{align}
	As the order, $m$, only determines the number of terms and since all terms are non-negative, the integral is monotonically non-decreasing in $m$. Therefore,
	\begin{align}
		\int_{t_0}^\infty \ell_m(t) \incr t &\leq \int_{t_0}^\infty \ell_{m^*}(t) \incr t
	\end{align}
	for $m \leq m^*$, and we substitute the value of $a$, which depends on $m^*$ and $\gamma$:
	\begin{align}
		\int_{t_0}^\infty \ell_{m^*}(t) \incr t
		&= \sum_{n=0}^{m^*} \frac{\left(\frac{m^*+1}{\gamma} t_0\right)^n}{n!} e^{-(m^*+1) t_0/\gamma}.
	\end{align}
	For sufficiently large ratios $t_0/\gamma$, the exponentially decaying factor will cause the expression to be monotonically decreasing for all values of $m^*$. Note that the convergence proof in Appendix~\ref{sec:exponential:function} would also not hold under the conditions considered here, i.e., the rightmost expression in~\eqref{sec:exponential:function:convergence:inequality} would not approach zero. Consequently, we can bound the expression by substituting $m^*=0$:
	\begin{align}
		\int_{t_0}^\infty \ell_{m^*}(t) \incr t &\leq e^{-t_0/\gamma}.
	\end{align}
	We use these bounds to derive an upper bound on the first sum in the rightmost expression in~\eqref{eq:tight:proof:split:sum}:
	\begin{align}
		\sum_{m=0}^{m^*} |c_m| \int_{t_0}^\infty \ell_m(t) \incr t
		&\leq \int_{t_0}^\infty \ell_{m^*}(t) \incr t \sum_{m=0}^{m^*} |c_m|
		\leq e^{-t_0/\gamma} \sum_{m=0}^{m^*} |c_m| \\
		&\leq e^{-t_0/\gamma} \rho(s_{m^*+1})
		= e^{-t_0/\gamma} \rho\left(\frac{m^*+1}{a}\right) \nonumber \\
		&= e^{-t_0/\gamma} \rho(\gamma). \nonumber
	\end{align}
	Finally, we combine the two derived bounds:
	\begin{align}
		\int_{t_0}^\infty |\hat \alpha(t)| \incr t
		&\leq \rho(\infty) - \rho(\gamma) + e^{-t_0/\gamma} \rho(\gamma).
	\end{align}
	To summarize, we first choose $\gamma$ large enough that $\rho(\infty) - \rho(\gamma) < \epsilon/2$ and then we choose $t_0$ large enough that $e^{-t_0/\gamma} \rho(\gamma) < \epsilon/2$. As $\gamma$ is constant for all combinations of $m^*$ and $a$, the bound holds for all values of the rate parameter, $a$. Furthermore, the result is independent of the order, $M$, of the Erlang mixture kernel. This concludes the proof.}
\end{proof}
	\section{Derivation of the linear chain trick}\label{sec:deriv}%
In this appendix, we show the derivation of the LCT which transforms DDEs with Erlang mixture kernels to ODEs~\cite{MacDonald:1978}.
\begin{lemma}
\chgtwo{
	The approximate system of DDEs obtained by approximating each element of $\alpha$ in~\eqref{eq:system:x}--\eqref{eq:system:delay} with an Erlang mixture approximation can be transformed analytically to the system of ODEs~\eqref{eq:lct:approximate:system:ODE} \chgthree{in the Erlang ODE approximation}.
}
\end{lemma}
\begin{proof}
	Consider the approximate system of DDEs
	\begin{subequations}\label{eq:lct:approximate:system:DDE:derivation}
		\begin{align}
			\label{eq:lct:approximate:system:DDE:derivation:x}
			\dot{\hat x}(t) &= f(\hat x(t), \hat z(t)), \\
			\label{eq:lct:approximate:system:DDE:derivation:z}
			\hat z(t) &= \int_{-\infty}^t \hat \alpha(t - s) \hat r(s) \incr s, \\
			\label{eq:lct:approximate:system:DDE:derivation:r}
			\hat r(t) &= h(\hat x(t)),
		\end{align}
	\end{subequations}
	where each element of the kernel $\hat \alpha: \Rnn \rightarrow \R^{n_z \chgthree{\times n_r}}$ is an Erlang mixture \chgtwo{approximation of $\alpha_{i\chmthree{j}}$ (see \chgthree{Definition}~\ref{def:erlang:mixture:approximation})}:
	\begin{align}\label{eq:lct:approximate:kernel:DDE}
		\hat \alpha_{i\chmthree{j}}(t) &= \sum_{m=0}^{M_{i\chmthree{j}}} c_{mi\chmthree{j}} \ell_{mi\chmthree{j}}(t), & i &= 1, \ldots, n_z\chmthree{, & j &= 1, \ldots, n_r}.
	\end{align}
	The subscript\chgthree{s} $i$ \chgthree{and $j$} on the Erlang kernel, $\ell_{mi\chmthree{j}}: \Rnn \rightarrow \Rnn$, indicate that it depends on the $i\chmthree{, j}$'th rate parameter, $a_{i\chmthree{j}} \in \Rp$. Furthermore, $\hat x: \R \rightarrow \R^{n_x}$\chgthree{,} $\hat z: \R \rightarrow \R^{n_z}$\chgthree{, and $\hat r: \R \rightarrow \R^{n_r}$} are approximations of $x$, $z$, and $r$, respectively.

	First, we introduce the auxiliary memory state $\chmthree{\hat Z}_{mi\chmthree{j}}: \R \rightarrow \R$ given by
	\begin{align}\label{eq:approximation:revisited:z}
		\chmthree{\hat Z}_{mi\chmthree{j}}(t) &= \int_{-\infty}^t \ell_{mi\chmthree{j}}(t - s) \hat r_{\chmthree{j}}(s) \incr s, & m &= 0, \ldots, M_{i\chmthree{j}}, & i &= 1, \ldots, n_z\chmthree{, & j &= 1, \ldots, n_r}.
	\end{align}
	Next, we rewrite the expression for the memory state,
	\begin{align}\label{eq:approximation:z}
		\hat z_i(t) &= \int_{-\infty}^t \chmthree{\sum_{j=1}^{n_r}} \hat \alpha_{i\chmthree{j}}(t - s) \hat r_{\chmthree{j}}(s) \incr s = \chmthree{\sum_{j=1}^{n_r}} \sum_{m=0}^{M_{i\chmthree{j}}} c_{mi\chmthree{j}} \int_{-\infty}^t \ell_{mi\chmthree{j}}(t - s) \hat r_{\chmthree{j}}(s) \incr s \\
		&= \chmthree{\sum_{j=1}^{n_r}} \sum_{m=0}^{M_{i\chmthree{j}}} c_{mi\chmthree{j}} \chmthree{\hat Z}_{mi\chmthree{j}}(t), \nonumber
	\end{align}
	for $i = 1, \ldots, n_z$\chgthree{, and $j = 1, \ldots, n_r$}. The time derivatives of the Erlang kernels are
	\begin{subequations}\label{eq:erlang:pdf:derivative}
		\begin{align}
			\dot{\ell}_{0,i\chmthree{j}}(t) &= -a_{i\chmthree{j}} \ell_{0, i\chmthree{j}}(t), &&& i &= 1, \ldots, n_z\chmthree{, & j &= 1, \ldots, n_r}, \\
			\dot{\ell}_{mi\chmthree{j}}(t)
			&= a_{i\chmthree{j}} (\ell_{m-1, i\chmthree{j}}(t) - \ell_{mi\chmthree{j}}(t)), & m &= 1, \ldots, M_{i\chmthree{j}}, & i &= 1, \ldots, n_z\chmthree{, & j &= 1, \ldots, n_r},
		\end{align}
	\end{subequations}
	and the normalization factor satisfies the recursion
	\begin{align}
		b_{mi\chmthree{j}} &= \frac{a_{i\chmthree{j}}^{m+1}}{m!} = \frac{a_{i\chmthree{j}}}{m} \frac{a_{i\chmthree{j}}^m}{(m-1)!} = \frac{a_{i\chmthree{j}}}{m} b_{m-1, i\chmthree{j}}, & m &= 1, \ldots, M_{i\chmthree{j}}, & i &= 1, \ldots, n_z\chmthree{, & j &= 1, \ldots, n_r}.
	\end{align}
	In order to derive differential equations for the auxiliary memory states, we use Leibniz' integral rule~\cite[Thm.~3, Chap.~8]{Protter:Morrey:1985} to differentiate~\eqref{eq:approximation:revisited:z}:
	\begin{align}
		\chmthree{\dot{\hat Z}}_{mi\chmthree{j}}(t)
		&= \ell_{mi\chmthree{j}}(0) \hat r_{\chmthree{j}}(t) + \int_{-\infty}^t \dot{\ell}_{mi\chmthree{j}}(t - s) \hat r_{\chmthree{j}}(s) \incr s, & m &= 0, \ldots, M_{i\chmthree{j}}.
	\end{align}
	\chgthree{for $i = 1, \ldots, n_z$ and $j = 1, \ldots, n_r$.} Next, we use the time derivatives of the Erlang kernels and the fact that
	\begin{align}
		\ell_{mi\chmthree{j}}(0) &=
		\begin{cases}
			a_{i\chmthree{j}}, 	& \text{for}~m = 0, \\
			0,   				& \text{for}~m = 1, \ldots, M_{i\chmthree{j}},
		\end{cases} & i &= 1, \ldots, n_z\chmthree{, & j &= 1, \ldots, n_r},
	\end{align}
	to obtain the differential equations
	\begin{subequations}
		\begin{align}
			\chmthree{\dot{\hat Z}}_{0, i\chmthree{j}}(t)
			&= a_{i\chmthree{j}} (\hat r_{\chmthree{j}}(t) - \chmthree{\hat Z}_{0, i\chmthree{j}}(t)), \\
			\chmthree{\dot{\hat Z}}_{mi\chmthree{j}}(t)
			&= a_{i\chmthree{j}} (\chmthree{\hat Z}_{m-1, i\chmthree{j}}(t) - \chmthree{\hat Z}_{mi\chmthree{j}}(t)), \quad m = 1, \ldots, M_{i\chmthree{j}}.
		\end{align}
	\end{subequations}
	Finally, by introducing the vector of all \chgthree{auxiliary} memory states,
	\setcounter{MaxMatrixCols}{20} 
	\begin{align}
		\hat Z &=
		\chmthree{
		\begin{bmatrix}
			\chmthree{\hat Z}^{(1, 1)} \\
			\chmthree{\hat Z}^{(1, 2)} \\
			\vdots \\
			\chmthree{\hat Z}^{(1, n_r)} \\
			\chmthree{\hat Z}^{(2, 1)} \\
			\vdots \\
			\chmthree{\hat Z}^{(n_z, n_r)}
		\end{bmatrix}, &
		\chmthree{\hat Z}^{(ij)} &=
		\begin{bmatrix}
			\chmthree{\hat Z}_{0, ij} \\
			\chmthree{\hat Z}_{1, ij} \\
			\vdots \\
			\chmthree{\hat Z}_{M_{ij}, ij}
		\end{bmatrix},
		& i &= 1, \ldots, n_z, & j &= 1, \ldots, n_r,
	}
	\end{align}
	the system of DDEs can be transformed to the system of ODEs~\eqref{eq:lct:approximate:system:ODE}.
\end{proof}
	\section{Numerical simulation}\label{sec:numerical:simulation}%
In this appendix, we present two numerical methods for simulating DDEs with distributed time delays, i.e., for approximating the solution to \chgthree{IVP}s in the form~\eqref{eq:system}--\eqref{eq:system:delay}. The first is intended for non-stiff systems, and it uses Euler's explicit method to discretize the differential equations and a left rectangle rule to discretize the integral in the convolution in~\eqref{eq:system:z}. In contrast, the second is intended for stiff systems, and it uses Euler's implicit method and a right rectangle rule to discretize the differential equations and the integral, respectively. In both cases, we truncate the integral:
\begin{align}\label{eq:numerical:simulation:system:z}
	z(t) &= \int_{-\infty}^t \alpha(t - s) r(s) \incr s \approx \int_{t - \Delta t_h}^t \alpha(t - s) r(s) \incr s.
\end{align}
The memory horizon, $\Delta t_h \in \Rp$, can, e.g., be determined with the bisection approach described in Section~\ref{sec:algo:domain}.

\subsection{Non-stiff systems}\label{sec:numerical:simulation:non:stiff}
For non-stiff systems, we obtain the approximate system
\begin{subequations}\label{eq:numerical:simulation:non:stiff}
	\begin{align}
		\label{eq:numerical:simulation:non:stiff:x}
		x_{n+1} &= x_n + f(x_n, z_n) \Delta t, \\
		\label{eq:numerical:simulation:non:stiff:z}
		z_{n+1} &= \sum_{j=1}^{N_h} \alpha(j \Delta t) r_{n-j+1} \Delta t, \\
		\label{eq:numerical:simulation:non:stiff:r}
		r_{n+1} &= h(x_{n+1}),
	\end{align}
\end{subequations}
where $t_n = t_0 + n \Delta t \in \R$ for some time step size, $\Delta t \in \Rp$. We assume that the memory horizon is a multiple, $N_h \in \N$, of the time step size, $\Delta t$. For $t_n \leq t_0$, $x_n \in \R^{n_x}$\chgthree{,} $z_n \in \R^{n_z}$\chgthree{, and $r_n \in \R^{n_r}$} are equal to $x(t_n)$, $z(t_n)$, and $r(t_n)$, and for $t_n > t_0$, they are approximations. This is an \emph{explicit} approach, i.e., for given values of $x_n$, $z_n$, and $r_n$, $z_{n+1}$ can be computed using~\eqref{eq:numerical:simulation:non:stiff:z} and $x_{n+1}$ can be evaluated with~\eqref{eq:numerical:simulation:non:stiff:x}. Finally, $r_{n+1}$ can be evaluated with~\eqref{eq:numerical:simulation:non:stiff:r}, and the process is repeated for the subsequent time steps.

\subsection{Stiff systems}\label{sec:numerical:simulation:stiff}
For stiff systems, the approximate system is
\begin{subequations}\label{eq:numerical:simulation:stiff}
	\begin{align}
		\label{eq:numerical:simulation:stiff:x}
		x_{n+1} &= x_n + f(x_{n+1}, z_{n+1}) \Delta t, \\
		\label{eq:numerical:simulation:stiff:z}
		z_{n+1} &= \sum_{j=0}^{N_h-1} \alpha(j \Delta t) r_{n-j+1} \Delta t, \\
		\label{eq:numerical:simulation:stiff:r}
		r_{n+1} &= h(x_{n+1}),
	\end{align}
\end{subequations}
which is a set of algebraic equations that, in general, must be solved numerically. In this work, we solve the residual equation
\begin{align}\label{eq:numerical:simulation:stiff:residual}
	R_n(x_{n+1}; x_n) &= x_{n+1} - x_n - f(x_{n+1}, z_{n+1}) \Delta t = 0,
\end{align}
for each time step sequentially. Here, $R_n: \R^{n_x} \times \R^{n_x} \rightarrow \R^{n_x}$ is a residual function, and we consider $z_{n+1}$ and $r_{n+1}$ to be functions of $x_{n+1}$.
Many numerical methods for solving nonlinear equations can utilize the Jacobian of the involved residual function. The Jacobian of the residual function in~\eqref{eq:numerical:simulation:stiff:residual} is given by
\begin{align}
	\pdiff{R_n}{x_{n+1}}(x_{n+1}; x_n) &= I - \Bigg(\pdiff{f}{x}(x_{n+1}, z_{n+1}) + \pdiff{f}{z}(x_{n+1}, z_{n+1}) \pdiff{z_{n+1}}{x_{n+1}}\Bigg) \Delta t,
\end{align}
where $I \in \R^{n_x \times n_x}$ is an identity matrix and the Jacobians of the memory state and the delayed variables are
\begin{subequations}
	\begin{align}
		\pdiff{z_{n+1}}{x_{n+1}} &= \alpha(0) \pdiff{r_{n+1}}{x_{n+1}} \Delta t, \\
		\pdiff{r_{n+1}}{x_{n+1}} &= \pdiff{h}{x}(x_{n+1}).
	\end{align}
\end{subequations}
	\section{Kernels in the molten salt reactor model}\label{sec:nuclear:fission:kernel}%
In this appendix, we describe the kernel in the molten salt nuclear reactor model from Section~\ref{sec:ex:nuclear:fission}. The $i$'th kernel, $\alpha_i: \R \rightarrow \Rnn$, is given by
\begin{align}\label{eq:nuclear:fission:kernels}
	\alpha_i(t) &= \gamma_i e^{-\lambda_i t} \sum_{j=1}^{N_s} F(t; \mu_j, \sigma_j), & i &= 1, \ldots, N_g,
\end{align}
where $\gamma_i \in \Rp$ is a normalization constant, $F: \Rnn \times \R \times \Rp \rightarrow \Rnn$ is the folded normal kernel~\eqref{eq:folded:normal:pdf}, and $N_s = 7 \in \N$ is the number of terms. The sum of folded normal kernels represents a nonuniform velocity profile in the external circulation of the molten salt (and neutron precursors), and the second factor describes the decay of the neutron precursors during the external circulation. For given $\mu_1 \in \R$ and $\sigma_1 \in \Rp$, the location and scale parameters are
\begin{align}
	\sigma_{j+1} &= \frac{3}{2} \sigma_j, &
	\mu_{j+1} &= \mu_j + \sigma_j, & j &= 1, \ldots, N_s - 1.
\end{align}
Finally, the normalization constant is
\begin{align}
    \gamma_i &= \left(\sum_{j=1}^{N_s} \frac{1}{2} e^{\frac{1}{2} \sigma_j^2 \lambda_i^2} \left(e^{\lambda_i \mu_j} \left(1 - \erf\left(\frac{\lambda_i \sigma_j^2 + \mu_j}{\sqrt{2} \sigma}\right)\right) + e^{-\lambda_i \mu_j} \left(1 - \erf\left(\frac{\lambda_i \sigma_j^2 - \mu_j}{\sqrt{2} \sigma_j}\right)\right)\right)\right)^{-1},
\end{align}
for $i = 1, \ldots, N_g$.
\end{document}